\documentclass[11pt]{amsart}

\usepackage[margin=1in]{geometry}
\usepackage{mathrsfs}
\usepackage{caption}
\usepackage{subcaption}
\usepackage[backref=page]{hyperref}
\hypersetup{
    colorlinks=true,
    linkcolor=blue,
    filecolor=blue,      
    urlcolor=black,
    citecolor=blue,
}
\usepackage[labelformat=simple]{subcaption}
\usepackage{amsmath}
\usepackage{amsthm}
\usepackage{amssymb}
\usepackage{amsrefs}
\usepackage{wasysym}  %\Leftcircle
\usepackage{enumerate}
\usepackage{tikz}
\usepackage{tikz-cd}
\usetikzlibrary{calc,positioning,patterns,external}
\newcommand{\ldb}{\mathopen{\lbrack\!\lbrack}} 
\newcommand{\rdb}{\mathclose{\rbrack\!\rbrack}}
\newcommand{\fram}[1]{\ldb #1\rdb}
\newcommand{\zrange}[1]{\ldb #1\rdb}
\newcommand{\range}[1]{[#1]}

\usepackage{bbm} % \mathbbm{1}.

\newcommand*{\udots}{\reflectbox{$\ddots$}}

\makeatletter
\DeclareRobustCommand*\cal{\@fontswitch\relax\mathcal}
\makeatother

\makeatletter
\DeclareRobustCommand*\frak{\@fontswitch\relax\mathfrak}
\makeatother

\numberwithin{equation}{section}
\newtheorem{thm}[equation]{Theorem}
\newtheorem*{thm*}{Theorem}
\newtheorem{mainthm}{Theorem}

\newtheorem{submainthm}{Theorem}

\newtheorem{lem}[equation]{Lemma}
\newtheorem{prop}[equation]{Proposition}

\newtheorem{lemma}[equation]{Lemma}
\newtheorem{coro}[equation]{Corollary}
\newtheorem*{coro*}{Corollary}
\newtheorem{quest}[equation]{Question}

\theoremstyle{remark}
\newtheorem*{remark*}{Remark}
\newtheorem{remark}[equation]{Remark}

\theoremstyle{definition}
\newtheorem{defn}[equation]{Definition}
\newtheorem{ex}[equation]{Example}

\numberwithin{figure}{section}
\numberwithin{table}{section}

\DeclareMathOperator{\rad}{rad }
\DeclareMathOperator{\Hom}{hom }
\DeclareMathOperator{\End}{End}

\DeclareMathOperator{\Der}{Der }
\DeclareMathOperator{\GL}{GL}

\DeclareMathOperator{\Nuc}{Nuc}
\DeclareMathOperator{\Cen}{Cen}
\DeclareMathOperator{\gl}{\mathfrak{gl}}

\DeclareMathOperator{\Aut}{Aut}

\DeclareMathOperator{\Cent}{Cen}
\DeclareMathOperator{\im}{im }
\DeclareMathOperator{\Ann}{Ann }

\DeclareMathOperator{\rank}{rank }

\DeclareMathOperator{\supp}{supp}
\DeclareMathOperator{\ad}{ad}

\newenvironment{bsmallmatrix}
  {\left[\begin{smallmatrix}}
  {\end{smallmatrix}\right]}

\newcommand{\Comm}[1]{{\mathsf{Comm}\textrm{-}{#1}}}
\newcommand{\Set}{{\mathsf{Set}}}

\newcommand{\M}{\mathbb{M}}

\newcommand{\op}{\mathrm{op}}

\newcommand{\Ten}[3][]{\mathbf{{N}}^{#1}\left(#2,#3\right)}
\newcommand{\Id}[3][]{\mathbf{{I}}^{#1}\left(#2,#3\right)}
\newcommand{\Op}[3][]{\mathbf{{Z}}^{#1}\left(#2,#3\right)}
\newcommand{\Den}[1]{\Leftcircle\hspace*{-1mm}#1\hspace*{-1mm} \Rightcircle}
\newcommand{\OpSet}[4][]{\mathbf{{Z}}^{#1}\left(#2,#3\right)\left(#4\right)}

\newcommand{\lversor}{\,\reflectbox{\ensuremath{\oslash}}\,}%\obackslash}
\newcommand{\rversor}{\oslash}

\newcommand{\bmto}{\rightarrowtail}

\newcommand{\la}{\langle}
\newcommand{\ra}{\rangle}

\makeatletter
\DeclareFontFamily{OMX}{MnSymbolE}{}
\DeclareSymbolFont{MnLargeSymbols}{OMX}{MnSymbolE}{m}{n}
\SetSymbolFont{MnLargeSymbols}{bold}{OMX}{MnSymbolE}{b}{n}
\DeclareFontShape{OMX}{MnSymbolE}{m}{n}{
    <-6>  MnSymbolE5
   <6-7>  MnSymbolE6
   <7-8>  MnSymbolE7
   <8-9>  MnSymbolE8
   <9-10> MnSymbolE9
  <10-12> MnSymbolE10
  <12->   MnSymbolE12
}{}
\DeclareFontShape{OMX}{MnSymbolE}{b}{n}{
    <-6>  MnSymbolE-Bold5
   <6-7>  MnSymbolE-Bold6
   <7-8>  MnSymbolE-Bold7
   <8-9>  MnSymbolE-Bold8
   <9-10> MnSymbolE-Bold9
  <10-12> MnSymbolE-Bold10
  <12->   MnSymbolE-Bold12
}{}

\let\llangle\@undefined
\let\rrangle\@undefined
\DeclareMathDelimiter{\llangle}{\mathopen}%
                     {MnLargeSymbols}{'164}{MnLargeSymbols}{'164}
\DeclareMathDelimiter{\rrangle}{\mathclose}%
                     {MnLargeSymbols}{'171}{MnLargeSymbols}{'171}
\makeatother
\newcommand{\lla}{\llangle}
\newcommand{\rra}{\rrangle}
\newcommand{\comp}[1]{\bar{#1}}

\newcommand{\bra}[1]{\la #1|}
\newcommand{\ket}[1]{|#1\ra}
\newcommand{\1}{\mathbbm{1}}

\usepackage{cjhebrew}
\DeclareFontFamily{U}{rcjhbltx}{}
\DeclareFontShape{U}{rcjhbltx}{m}{n}{<->rcjhbltx}{}
\DeclareSymbolFont{hebrewletters}{U}{rcjhbltx}{m}{n}

\let\aleph\relax\let\beth\relax
\let\gimel\relax\let\daleth\relax

\DeclareMathSymbol{\aleph}{\mathord}{hebrewletters}{39}
\DeclareMathSymbol{\beth}{\mathord}{hebrewletters}{98}
\DeclareMathSymbol{\gimel}{\mathord}{hebrewletters}{103}
\DeclareMathSymbol{\daleth}{\mathord}{hebrewletters}{100}

\DeclareMathSymbol{\lamed}{\mathord}{hebrewletters}{108}
\DeclareMathSymbol{\mem}{\mathord}{hebrewletters}{109}
\DeclareMathSymbol{\ayin}{\mathord}{hebrewletters}{96}
\DeclareMathSymbol{\tsadi}{\mathord}{hebrewletters}{118}
\DeclareMathSymbol{\qof}{\mathord}{hebrewletters}{114}
\DeclareMathSymbol{\shin}{\mathord}{hebrewletters}{152}
\DeclareMathSymbol{\waw}{\mathord}{hebrewletters}{119}
\DeclareMathSymbol{\vavv}{\mathord}{hebrewletters}{119}

\DeclareMathOperator{\vav}{{{\mathtt{v}}}}

\newcommand{\zerovav}{\fram{\vav}}

\renewcommand{\subseteq}{\subset}
\tikzset{pics/.cd,
cube/.style args={#1/#2/#3/#4}{code={
\coordinate (O) at (0,0,0);
\coordinate (A) at (0,#2,0);
\coordinate (B) at (0,#2,#3);
\coordinate (C) at (0,0,#3);
\coordinate (D) at (#1,0,0);
\coordinate (E) at (#1,#2,0);
\coordinate (F) at (#1,#2,#3);
\coordinate (G) at (#1,0,#3);
\draw[black,fill=black!80] (O) -- (C) -- (G) -- (D) -- cycle;
\draw[black,fill=black!30] (O) -- (A) -- (E) -- (D) -- cycle;
\draw[black,fill=black!10] (O) -- (A) -- (B) -- (C) -- cycle;
\draw[black,fill=black!20,opacity=0.8] (D) -- (E) -- (F) -- (G) -- cycle;
\draw[black,fill=black!20,opacity=0.6] (C) -- (B) -- (F) -- (G) -- cycle;
\draw[black,fill=black!20,opacity=0.8] (A) -- (B) -- (F) -- (E) -- cycle;
\node at (0.5*#1,0.5*#2,0.5*#3) {#4};
}}}
\tikzset{pics/.cd,
   tcube/.style args={#1/#2/#3/#4/#5}{
      code={
         \coordinate (O) at (0,0,0);
         \coordinate (A) at (0,#2,0);
         \coordinate (B) at (0,#2,#3);
         \coordinate (C) at (0,0,#3);
         \coordinate (D) at (#1,0,0);
         \coordinate (E) at (#1,#2,0);
         \coordinate (F) at (#1,#2,#3);
         \coordinate (G) at (#1,0,#3);
         \draw[black,fill=black!5,opacity=0.6] (O) -- (A) -- (B) -- (C)-- cycle;
         \draw[black,fill=black!5,opacity=1] (O) -- (A) -- (E) -- (D)-- cycle;
         \draw[black,fill=black!5,opacity=0.8] (O) -- (D) -- (G) -- (C)-- cycle;
         \node at (0.5*#1,0.6*#2,0.5*#3) {{\small #5}};

         \draw[dotted] (E) -- (F) -- (G);
         \draw[dotted] (B) -- (F);
      }
   }
}
\tikzset{pics/.cd,
   pcube/.style args={#1/#2/#3/#4}{
      code={
	  }
   }
}
\tikzset{pics/.cd,
ccube/.style args={#1/#2/#3/#4}{code={
   \coordinate (zzz) at (0,0,0);
   \coordinate (zoz) at (0,#2,0);
   \coordinate (zoo) at (0,#2,#3);
   \coordinate (zzo) at (0,0,#3);
   \coordinate (ozz) at (#1,0,0);
   \coordinate (ooz) at (#1,#2,0);
   \coordinate (ooo) at (#1,#2,#3);
   \coordinate (ozo) at (#1,0,#3);
   \draw[black,fill=#4!80] (zzz) -- (ozz) -- (ozo) -- (zzo) -- cycle;
   \draw[black,fill=#4!40] (zzz) -- (ozz) -- (ooz) -- (zoz) -- cycle;
   \draw[black,fill=#4!20] (zzo) -- (zzz) -- (zoz) -- (zoo) -- cycle;
   \draw[black,fill=#4!20, opacity=0.6] (ozo) -- (ooo) -- (ooz) -- (ozz) -- cycle;
   \draw[black,fill=#4!20, opacity=0.8] (zzo) -- (zoo) -- (ooo) -- (ozo) -- cycle;
   \draw[black,fill=#4!20, opacity=0.8] (zoo) -- (zoz) -- (ooz) -- (ooo) -- cycle;
}}}
\tikzset{pics/.cd,
linecube/.style args={#1/#2/#3}{code={
   \coordinate (zzz) at (0,0,0);
   \coordinate (zoz) at (0,#2,0);
   \coordinate (zoo) at (0,#2,#3);
   \coordinate (zzo) at (0,0,#3);
   \coordinate (ozz) at (#1,0,0);
   \coordinate (ooz) at (#1,#2,0);
   \coordinate (ooo) at (#1,#2,#3);
   \coordinate (ozo) at (#1,0,#3);
   \draw[black] (zoz) -- (zoo) -- (zzo) -- (ozo) -- (ozz) -- (ooz) -- cycle;
   \draw[dotted] (zoo) -- (ooo) -- (ozo);
   \draw[dotted] (ooz) -- (ooo);
}}}

\tikzset{pics/.cd,
lwing/.style args={#1/#2/#3/#4}{code={
\coordinate (O) at ( 0, 0, 0);
\coordinate (A) at ( 0, 0,#3);
\coordinate (B) at ( 0,#2,#3);
\coordinate (C) at ( 0,#2, 0);
\draw[black,fill=red!20] (O) -- (A) -- (B) -- (C) -- cycle;
\node at (0,0.5*#2,0.5*#3) {#4};
}}}
\tikzset{pics/.cd,
mwing/.style args={#1/#2/#3/#4}{code={
\coordinate (O) at ( 0, 0, 0);
\coordinate (A) at (#1, 0, 0);
\coordinate (B) at (#1,#2, 0);
\coordinate (C) at ( 0,#2, 0);
\draw[black,fill=green!20] (O) -- (A) -- (B) -- (C) -- cycle;
\node at (0.5*#1,0.5*#2,0) {#4};
}}}
\tikzset{pics/.cd,
rwing/.style args={#1/#2/#3/#4}{code={
\coordinate (O) at ( 0, 0, 0);
\coordinate (A) at (#1, 0, 0);
\coordinate (B) at (#1, 0,#3);
\coordinate (C) at ( 0, 0,#3);
\draw[black,fill=blue!20] (O) -- (A) -- (B) -- (C) -- cycle;
\node at (0.5*#1,0,0.5*#3) {#4};
}}}

\title{A spectral theory for transverse tensor operators}

\author{Uriya First}
\address{
	University of Haifa\\
	Department of Mathematics\\
	199 Abba Khoushy Avenue\\
	Haifa, Israel
}
\email{ufirst@univ.haifa.ac.il}

\author{Joshua Maglione}
\address{
	Fakult\"at f\"ur Mathematik\\
	Universit\"at Bielefeld\\
   D-33501 Bielefeld\\ 
   Germany
}
\email{jmaglione@math.uni-bielefeld.de}

\author{James B. Wilson}
\thanks{This work was partially supported by NSF grant DMS-1620454 and by the
	Simons Foundation \#636189, and EPSRC Grant Number EP/RO14604/1.}
\address{
	Department of Mathematics\\
	Colorado State University\\
   Fort Collins, CO 80523\\
   USA
}
\email{James.Wilson@ColoState.Edu}
\date{\today}

\begin{document}

%==============================================================================
%\input{abs.tex}
\begin{abstract}
    Tensors are multiway arrays of data, and transverse operators are the operators
    that change the frame of reference. We develop the spectral theory of
    transverse tensor operators and apply it to problems closely related to
    classifying quantum states of matter, isomorphism in algebra, clustering in
    data, and the design of high performance tensor type-systems.  
    We prove the existence and uniqueness of the optimally-compressed tensor
    product spaces over algebras, called \emph{densors}. This gives structural
    insights for tensors and improves how we recognize tensors in arbitrary
    reference frames. Using work of Eisenbud--Sturmfels on binomial ideals, we
    classify the maximal groups and categories of transverse operators, leading us
    to general tensor data types and categorical tensor decompositions, amenable
    to theorems like Jordan--H\"older and Krull--Schmidt. All categorical tensor
    substructure is detected by transverse operators whose spectra contain a
    Stanley--Reisner ideal, which can be analyzed with combinatorial and
    geometrical tools via their simplicial complexes. 
    Underpinning this is a ternary Galois correspondence between tensor
    spaces, multivariable polynomial ideals, and transverse operators.  
    This correspondence can be computed in polynomial time.  We give an
    implementation in the computer algebra system \textsf{Magma}.
\end{abstract}
%==============================================================================

\maketitle

\setcounter{tocdepth}{1}%
\tableofcontents

%==============================================================================
%\input{intro.tex}

\section{Introduction}\label{sec:intro}

Let $K$ be a commutative ring, and we say that a \emph{tensor} $t$ is any element of 
\begin{align}\label{def:T-1}
		T:=K^{d_0\times \cdots\times d_{\vav}}\cong K^{d_0}\otimes\cdots\otimes
K^{d_{\vav}}.
\end{align} 
So, $t$ is a $(d_0\times\cdots\times d_{\vav})$ multiway array of numbers, a.k.a. a
hypermatrix. Call $K^{d_a}$ the $a$-\emph{axis} of $t$,
$\{V_0,\ldots,V_{\vav}\}$ the frame, $\vav$ the \emph{valence}, but
reserve \emph{dimension} for the values $d_a$.  Elements $\omega$ of
\begin{align}\label{def:Omega-1}
	\Omega:=\M_{d_0}(K)\times\cdots\times\mathbb{M}_{d_{\vav}}(K)
\end{align}
are called \emph{transverse} operators. They act on $t\in T$, written $\omega\cdot t$, 
by operating on each axis.

Our investigation centers on the spectral properties of transverse operators
(Theorems~\ref{thm:correspondence} \& \ref{mainthm:construction}).
Aspects of this have appeared, see Section~\ref{sec:related}, but as far as we
know, this is the first general study. As with eigenvalues, our definitions here
are quite general, so they lend themselves to broad use.  We use this spectral theory to explore a
class of clustering problems in data sets (Section~\ref{sec:Singularities}),
to design the appropriate data types for tensors (Section~\ref{sec:algorithms}), and
to improve how one compares tensors effectively
(Section~\ref{sec:Small-densor}). Our solutions rest on resolving several
questions in multilinear algebra including recognizing universally smallest
tensor product spaces (Theorem~\ref{mainthm:Densor}), characterizing the largest groups
and algebras that can act transversely (Theorems~\ref{mainthm:Lie-Asc} \&
\ref{thm:group-intro}), and identifying what spectral properties signal the existence
of tensor substructure (Theorem~\ref{mainthm:singularity}).

A key step in the spectral theory of linear operators $\omega\in
\mathbb{M}_{d}(K)$ is to record the recurrence relation $\omega^{i+1}v=\lambda_0
v+\lambda_1 \omega^1 v+\cdots +\lambda_{i} \omega^i v$, for $v\in K^d$, as a
minimal polynomial $\min_{\omega}(x)=x^{i+1}-\lambda_{i+1}x^i-\cdots-\lambda_0$.
Transverse operators $\omega\in \Omega$
can be iterated by different amounts on each axis making a multivariable
recurrence \emph{ideal} in $K[X]:=K[x_0,\ldots,x_{\vav}]$.  For $S\subset T$ and
$\Delta \subset \Omega$ set
\begin{align}\label{def:id-intro}
	\Id{S}{\Delta} & =\left\{p(X)=\sum_e \lambda_e X^e
		~\middle|~\forall t\in S,\forall \omega\in\Delta,
          \quad \sum_e \lambda_e \omega_0^{e(0)}\cdots 
            \omega_{\vav}^{e(\vav)}\cdot t=0\right\}.
\end{align}
We study the zero loci of $\Id{S}{\Delta}$ as affine subsets of
$\mathbb{A}^{1+\vav}(K):=K^{1+\vav}$. If $\vav=0$, $T=K^{d_0}$, and
$\Delta=\{\omega\}$, then $\Id{T}{\Delta}=(\min_{\omega}(x_0))\subset
K[x_0]$, and the zeros are the eigenvalues of $\omega_0$.

To generalize eigenspaces take $P\subset K[X]$ and $\Delta \subset \Omega$ and define:
\begin{align}\label{def:ten-intro}
	\Ten{P}{\Delta} & =\left\{t\in T 
	~\middle|~
		P\subset \Id{t}{\Delta}
	\right\}.
\end{align}
Notice if $\vav=0$ and $P=(x-\lambda)$, then for $\omega\in
\mathbb{M}_{d_0}(K)$, $\Ten{P}{\omega}=\{t\in K^{d_0}\mid \omega t=\lambda t \}$
is the usual eigenspace (possibly trivial). In this way, the prime ideals $P$
containing $\Id{T}{\Delta}$ are candidates for transverse operator eigenvalues,
and the primary decomposition $Q_1\cap\cdots \cap Q_{\ell}=\Id{T}{\Delta}$, with
$\sqrt{Q}_i$ prime, further extends the concepts of spectrum and algebraic
multiplicity of eigenvalues to the general transverse operator setting.

One last set matters to transverse operators.  For $S\subset T$ and $P\subset
K[X]$, define
\begin{align}\label{def:op-intro}
	\Op{S}{P} & = \left\{
		\omega\in \Omega
		~\middle|~
		P\subset \Id{S}{\omega}
		\right\}.
\end{align}
If $\vav=0$, $T=K^{d_0}$, and $P=(p(x_0))$, then this is the set of matrices
$\omega\in \mathbb{M}_{d_0}(K)$, where $p(\omega)=0$. In other words, the operators
whose spectra has at least the roots of $p(x_0)$.

{\it We note that $\Op{S}{P}$ naturally has the structure of a
$K$-scheme, and \eqref{def:op-intro} will later be written $\Op{S}{P}(K)$, i.e.: the
$K$-rational points of the scheme.  Our results take this extra structure
into account. See Section~\ref{sec:Galois}.} Informally, this
means that whenever we say a property holds for $\Op{S}{P}$, we mean that it
also holds after base-changing $K$ to any commutative $K$-algebra.  

The definitions of $\Id{-}{-}$, $\Ten{-}{-}$, and $\Op{-}{-}$ are inclusion
reversing in each of the components. The three constructions are, moreover,
related by a ternary Galois connection:

\begin{mainthm}[Correspondence Theorem]
	\label{thm:correspondence}
	Let $P\subseteq K[X]$, $S\subseteq T:=K^{d_0}\otimes\cdots\otimes
	K^{d_{\vav}}$ and $\Delta\subseteq \Omega:=
	\mathbb{M}_{d_0}(K)\times\cdots\times \mathbb{M}_{d_{\vav}}(K)$. Then
	$\Id{S}{\Delta}$ is an ideal of $K[X]$, $\Ten{P}{\Delta}$ is a subspace of
	$T$, and $\Op{S}{P}$ is a closed subscheme of $\Omega$, satisfying the
	following Galois connection property:
	\begin{align*}
		S\subset \Ten{P}{\Delta} \quad \Longleftrightarrow\quad
		P\subset \Id{S}{\Delta} \quad\Longleftrightarrow\quad
		\Delta\subset \Op{S}{P}.
	\end{align*}
\end{mainthm}

To put these observations into use, we prove that the terms in this spectral
theory of transverse tensor operators are efficiently computable in reasonable
computational settings shown in Section~\ref{sec:algorithms}.

\begin{mainthm}\label{mainthm:construction}
    When the valence $\vav$ is fixed, there are polynomial-time algorithms that,
    given sets generating $S$, $P$ and $\Delta$, compute a basis for
    $\Ten{P}{\Delta}$, a Gr\"obner basis for $\Id{S}{\Delta}$ and its primary
    decomposition, and polynomials defining the scheme $\Op{S}{P}$. If $P$ is
    generated by linear polynomials, we may further compute a set of $2+\vav$ generic
    $K$-rational points of $\Op{S}{P}$.
\end{mainthm}

The algorithms of Theorem~\ref{mainthm:construction} and a suite of companion
data types and functions, including constructions used throughout, are detailed in
Section~\ref{sec:algorithms} and implemented by the second and third author as
an open source project \cite{TensorSpace}. This is also distributed as the multilinear
algebra package in the computer algebra system {\sf Magma}~\cite{magma}.

\subsection{Related works}\label{sec:related}

Several works have emphasized the study of tensors as a generalization of
linear operators.  There, notions of annihilating polynomials, eigenvalues,
singular value decomposition, and so on are introduced for a particular tensor.
Similarly, one can slice a tensor into matrices and consider the simultaneous
spectral theory of the ensemble.  Such approaches date back to work of Kronecker
and others, and recently, Belitskii, Landsberg, Lim, Qi, Sergeichuk, Sturmfels
and  collaborators have evolved this field intensely \citelist{\cite{Belitskii}
\cite{engineer} \cite{Landsberg} \cite{Lim:Spectral} \cite{Qi:det} \cite{Sergeichuk}
\cite{Sturmfels}}. Yet, the work of Drozd, H\aa stad, Lim, Raz and others show
these questions hide deep problems in algebraic geometry, wild representation
theory, and NP-completeness \citelist{\cite{Drozd} \cite{Haastad} \cite{Lim:NP}
\cite{Raz}}. 

An alternative perspective emphasizes tensors as generalization of distributive
products focusing on algebraic qualities.  Here, effort is placed on algebraic
structures associated to the tensor such as adjoint rings with involutions,
Lie and Jordan algebras, and groups of isometries. This has enabled projects
ranging across $\omega$-stability and finite Morley rank in algebra
\citelist{\cite{Malcev} \cite{Myasnikov}}, group isomorphism
\citelist{\cite{BW:autotopism} \cite{BMW:genus2} \cite{Ivanyos-Qiao}
\cite{LW:iso} \cite{Li-Qiao} \cite{Wilson:profiles}}, intersecting classical
groups \citelist{\cite{BW:isom} \cite{GG} \cite{BF}}, Krull--Schmidt type
theorems \citelist{\cite{Wilson:central} \cite{Wilson:Remak-I}}, obstructing
existence of characteristic subgroups \citelist{\cite{Maglione:efficient}
\cite{Maglione:filter} \cite{Glasby-Palfy-Schneider}}, and other algebraic
problems \citelist{\cite{BFM} \cite{First:general} \cite{Rossmann} 
\cite{Rossmann2} \cite{Tyburski} \cite{Wilson:division}}. While these
approaches have been fruitful, the techniques are all specialized by design and
mostly concern $3$-tensors and bilinear maps. 

In this work, we find a combination of both strategies.  In generalizing the
study of tensors as products, we continue to concentrate on operators that act
on individual tensors rather than tensors being treated as operators themselves.
Even so, we now rely on the algebro-geometric properties of the annihilators of
these operators to guide us towards selecting those operators that are most
informative of a tensor's structure.

\subsection{Preliminaries.}\label{sec:terminology}
Now let us give a more general context. We blend expositions of Dirac
\cite{Dirac}, Mal'cev \cite{Malcev}, and Landsberg \cite{Landsberg}, and the
many recent articles included in our bibliography.

Define $\zrange{\ell}:=\{0,\ldots,\ell\}$ and $\range{\ell}:=\{1,\ldots,\ell\}$.
For a set $R$, complements of elements $a$, resp.\ subsets $A$, in $R$ are
denoted $\bar{a}$, resp.\ $\bar{A}$. Elements $x\in X:=\prod_{a\in A}X_a$ are
functions $x:(a\in A)\to (x_a\in X_a)$. For $B\subset A$ we denote
restriction to $B$ by $x_B$ and $X_B$. 

We turn to a coordinate-free description of tensors of
f.g.~projective $K$-modules $V_a$ and abbreviate:
\[
	V_0\rversor V_1  =\hom_K(V_1,V_0) \qquad\text{and}\qquad V_0\rversor 
	\cdots \rversor V_{\vav}
	=(V_0\rversor \dots \rversor V_{\vav-1})\rversor V_{\vav}  
\]
($\rversor$ is pronounced ``versor''; cf. \cite{Wilson:division}).  Elements of
$V_0\oslash V_1$ are linear maps, e.g.\ matrices are lists of vectors.  Elements
of $V_0\oslash V_1\oslash V_2$ are \emph{bi-}linear, e.g.\ hypermatrices are
lists of matrices; and in general elements of $V_0\oslash\cdots \oslash
V_{\vav}$ are \emph{multi-}linear: lists of lower valence grids. 
More formally, we have a natural isomorphism $ \hom(V_2,\hom(V_1,V_0)) \cong
\hom(V_1\otimes V_2,V_0)$, and by iterating, we get 
$V_0\oslash V_1  \oslash \cdots  \oslash
V_{\vav}\cong \hom(V_1\otimes\cdots \otimes V_{\vav},V_0)$.

\begin{defn}
A \emph{tensor space} (over $K$) is a module $T$ together
with a $K$-linear injection
\begin{align}
\tag{Tensor Space}
	\bra{\cdot}:T\hookrightarrow V_0\oslash \ldots \oslash V_{\vav}.
\end{align}
Elements $t\in T$ are \emph{tensors}. The multilinear map $\bra{t}$ is its
\emph{interpretation} and $(V_{0},\ldots,V_{\vav})$ its \emph{frame}. 
\end{defn}

Given $t\in T$, the evaluation of $\bra{t}$ at $v\in \prod_{a\in \range{\vav}}
V_a$ is denoted by $\bra{t} v\ra$. For $a\in \range{\vav}$, write $\langle t |
\omega_av_a, v_{\bar{a}}\rangle := \langle t | v_1,\dots, \omega_a v_a, \dots,
v_{\vav}\rangle$, and likewise for subsets $A\subset \range{\vav}$. For example,
$T$ could consist of $(d_0\times\cdots \times d_{\vav})$ multiway arrays $t$ of
coefficients in $K$ and $V_a:=K^{d_a}$, then one interpretation $\bra{\cdot}$
assigns $t$ to the multilinear map $V_1\times\cdots\times V_{\vav}\bmto V_0$
($\bmto$ denotes multilinear) via
\begin{align*}
	\langle t|v_1,\ldots,v_{\vav}\rangle & :=
			\sum_{i_*} t^{i_0}_{i_1,\ldots,i_{\vav}} v_{i_1}\cdots v_{i_{\vav}} e_{i_0},
		& 
		e_{i} & = (0,\ldots,\underset{i}{0},\ldots,0)
\end{align*}

In the above example, we have used what is known as index-calculus: we
distinguished the $0$-axis as outputs and used it in exponents to indicate its
dual use.  More generally, we can partition $\zrange{\vav}=A\sqcup B\sqcup C$, where
$a\in A$ are the covariant axes, $b\in B$ are the contravariant axes, and $c\in C$
are held constant.  For example, the multiplication of a left $R$-module is a bilinear map
$*:R\times M\bmto M$, where $V_1:=R$ never changes, so $1\in C$ in this example.
Let us accordingly update our definition of transverse operators.
\begin{defn}
	Given covariant axes $A\subset \zrange{\vav}$ and contravariant axes
	$B\subset\bar{A}$, the \emph{transverse operators} associated with the \emph{frame}
	$(V_0,V_1,\dots,V_{\vav})$ are
	$\Omega_{A,B}:=\prod_{a\in A} \End(V_a)\times\prod_{b\in B}	\End(V_b)^{\op}$.
\end{defn} 
For convenience we assume $0\notin B$. We omit $B$ if $B=\emptyset$, and also
omit $A$ if $A=\zrange{\vav}$.

Given a transverse operator $\omega\in \Omega_{A,B}$, there is
an induced action of $K[X_{A\sqcup B}]\hookrightarrow K[X]$ on
$\End(V_0\oslash\cdots\oslash V_{\vav})$ where for $b\in B$, $\omega_b
v_b:=v_b(\omega_b^{\op})$, and for $c\in \zrange{\vav}-A-B$, all exponents $e$
in terms $\lambda_e X^e\in K[X_{A\sqcup B}]$ have $e(c)=0$; so, the $c$-axis is
held constant.  The action of $p(X)\in K[X_{A\sqcup B}]$ on $t\in T$ is
defined as follows.
\begin{align}\label{eq:p-omega-action}
	\bra{t} p(\omega) \ket{v} &:= 
		\sum_e \lambda_e \omega^{e(0)}_0 \bra{t} 
			\omega_1^{e(1)}v_1,\dots, \omega_{\vav}^{e(\vav)} v_{\vav}\rangle. 
\end{align}
\emph{From now on, treat $\Id{S}{\Delta}$, $\Ten{P}{\Delta}$, and $\Op{S}{P}$ in the
general setting of $S\subset T\hookrightarrow V_0\oslash\cdots\oslash V_{\vav}$,
$P\subset K[X]$, $\Delta\subset \Omega_{A,B}$ with $\bra{t}p(X)=0$ replacing
$\omega\cdot t=0$ from above.}

\subsection{The densor}
\label{subsec:densor}

We consider the \emph{Tensor Isomorphism Problem (TIP)} as studied in algebra
\citelist{\cite{BhargavaI} \cite{BhargavaII} \cite{BMW:genus2}
\cite{BOW:graded}}, physics \citelist{\cite{SLOC4}\cite{SLOCC:Classification}},
and computer science \citelist{\cite{Agrawal-Saxena} \cite{Grochow-Qiao}
\cite{Ivanyos-Qiao}}. Here one compares tensors $s,t\in V_0\oslash\cdots\oslash
V_{\vav}$ up to transverse equality, i.e. does there exists $\omega\in \prod_a
\Aut(V_a)$ such that $\bra{s}\omega=\bra{t}$?  This is a huge space and a huge
group to search. A better tensor product is one that is as small as possible but
also still computable.  We will see that the Correspondence Theorem can
manufacture this outcome.

If $\Delta\subseteq \Omega_{A,B}$ and $P\subseteq K[X_{A\sqcup B}]$ then
$\Ten{P}{\Delta}$ may be regarded as tensoring (or ``versoring'') the frame
$(V_{0},\dots,V_{\vav})$, or the tensor space $T$, over $\Delta$ relative to
$P$. For example, let $\vav=1$ and $T=V_0\oslash V_1$,   fix a subring $
\Delta\subset \Omega_{\{0,1\}}=\End(V_0)\times \End(V_1)$, and let
$P=(x_0-x_1)$.  Then
\[
	\Ten{P}{\Delta}=\{ 
		t\in T \mid (\forall \delta\in \Delta) (\la t |\delta_1 v\ra=\delta_0 \la t| v\ra )
	\}
	=
	V_0\oslash_{\Delta} V_1
\]  
Likewise if $\vav=2$, $T=V_0\oslash V_1\oslash V_2$, and $ \Delta$ is a subring
of $\Omega_{2,1}$, then
\(
	\Ten{x_1-x_2}{\Delta}=
	V_0\oslash (V_1\otimes_{\Delta} V_2).
\)  
This is detailed in Example~\ref{ex:nuclei}.

Fixing $P$, we see that the larger $\Delta$ is, the smaller $\Ten{P}{\Delta}$
is, and the largest possible $\Delta$ such that $S\subset \Ten{P}{\Delta}$ is
$\Op{S}{P}$ (Theorem~\ref{thm:correspondence}). We are therefore interested in
ideals $P\subseteq K[X]$ such that $\Ten{P}{\Op{S}{P}}$ is minimal.  Said
another way, $\Ten{P}{\Op{S}{P}}$ is a closure of the above Galois
connection with $P$ constant. In order that $\Op{S}{P}$ remains effectively
computable, we further restrict to the case where $P$ is a linear homogeneous
ideal---an ideal generated by linear homogeneous polynomials---so that
$\Op{S}{P}$ is a $K$-module.

The next theorem states that, under mild hypotheses, as $P$ ranges over linear
homogeneous ideals, $\Ten{P}{\Op{S}{P}}$ has a \emph{unique} minimal member independent
of $S$. To express this, we say that $S\subset T$ is \emph{nondegenerate} if
every nonzero $v_a$ in some $a$-axis admits $v_{\bar{a}}\in \prod_{b\neq a}V_b$ 
and $t\in S$ with $\bra{t} v\ra \neq 0$.  It is \emph{full} if $V_0$
is spanned as a $K$-module by all the $\bra{t} v\ra$. The support $\supp P$ of
$P$ is the set of $a\in \zrange{\vav}$ such that $K$ is generated by the
$\lambda_e$ where $\lambda_e X^e$ is a term of some $p(X)\in P$ and $e(a)\neq 0$.
Say $P$ has \emph{full support} if $\zrange{\vav}=\supp P$.

\begin{mainthm}\label{mainthm:Densor}
	Let $d:= x_0-x_1-\cdots-x_{\vav} \in K[X]$. If $P\subseteq K[X]$ is a
	linear homogeneous ideal of full support, then for each $S\subset T$,
	\[
		\Ten{P}{\Op{S}{P}}\supset \Ten{d}{\Op{S}{d}}.
	\]	
	If $S$ is fully nondegenerate, then this holds for every linear homogeneous
	ideal $P$.
\end{mainthm}

Because of the special distinction of $d$ in Theorem~\ref{mainthm:Densor}, we
denote by $\Der(S)$ the set $\Op{S}{d}$, calling its elements
\emph{derivations}, which is justified by the characterization: $\delta\in
\Der(S)$ if, and only if,
\begin{align}\label{def:derivation}
	\delta_0\bra{t} v\ra &=
	\bra{t} \delta_1 v_1,v_{\bar{1}}\ra +\dots+
	\bra{t} \delta_{\vav} v_{\vav},v_{\bar{\vav}}\ra
\end{align}
for all $v$ and $t\in S$, cf.\ \cite{LL:gen-der}. We
define the \emph{densor space} (a portmantaeu of derivation tensor) as 
\begin{align}
 \Den{S} &
 	= \{ t\in T\mid \Der(S)\subset \Der(t)\} 
	 = \Ten{d}{\Op{S}{d}}.
\end{align} 
The vector space $\Der(S)$ is a Lie subalgebra of
$\prod_{a\in\zerovav}\gl(V_a)$, where $\gl(V_a)=\End(V_a)$ with Lie bracket
$[\delta,\delta']=\delta\delta'-\delta'\delta$, and all the $V_a$ are Lie
modules over $\Der(S)$. The dimension of the densor space can be analyzed using
the representation theory of Lie algebras. This result is not generally a closed
form but rather a combinatorial algorithm appealing to the
Littlewood--Richardson rule and Clebsch--Gordan formulas. 

Returning to TIP, Theorem~\ref{mainthm:Densor}
says that the smallest possible compression of a tensor space $S$ using linear
homogeneous ideals is the densor space. We discuss implications of
this observation in Section~\ref{sec:Small-densor}, which are applied to the
group isomorphism problem in~\cite{BMW:Der-Densor} to achieve exponential
speedups for some classes of groups. In Section~\ref{sec:Small-densor}, we also
show tensors encoding quantum information, social networks, and a myriad of
algebraic structures that reside in densor spaces much smaller than the ambient
tensor space.

\subsection{Lie tensor products are canonical}

Tensor products $U\otimes_{\Delta} V$, as introduced by Whitney~\cite{Whitney},
have traditionally involved an associative ring $\Delta$.  In contrast, our
densor spaces are formed over the Lie algebra of derivations. This raises
the question of what families of subalgebras of $\Omega$ arise in a natural way
as the universal scalars $\Op{S}{P}$ of a closure $\Ten{P}{\Op{S}{P}}$.  
So we consider a range of products.  Choose $A\subset\zrange{\vav}$
and $(\lambda,\rho)\in (K^2)^A$, then for $\omega,\tau\in \Omega_A$, define
\begin{align}\label{eq:bullet-product}
	& \omega\bullet \tau=\omega \bullet_{(\lambda,\rho)} \tau
	=(\lambda_a\omega_a\tau_a+\rho_a\omega_a\tau_a)_{a\in A}.
\end{align}
The restriction of $\bullet$ to the $a$-axis is an associative subalgebra when the
projective point $(\lambda_a:\rho_a)\in\mathbb{P}^1(K)$ equals $(1:0)$ or
$(0:1)$, and it is a Lie (resp.\ Jordan) product if $(\lambda_a:\rho_a)=(1:-1)$
(resp.\ $(\lambda_a:\rho_a)=(\frac{1}{2}:\frac{1}{2})$). We show that Lie
algebras are natural and \emph{any} unital algebra is rare.

\begin{mainthm}\label{mainthm:Lie-Asc}
	Let $P=(p_1,\ldots,p_r)\subset K[X]$ be a linear homogeneous ideal of
	support $A$. If $S$ is full, or is nondegenerate in some axis, and there
	is a $(\lambda,\rho)\in (K^2)^A$ such that	$\Op[A]{S}{P}$ is closed to
	$\bullet_{(\lambda,\rho)}$, then the $a$-axis product is unital in at most $2r$
	axes $a\in A$. If $K$ is a field, then at least $|A|-2r$ of the
	axes products are a Lie product.
\end{mainthm}

We also show that there is always a Lie product of weighted derivations on
$\Op[A]{S}{p}$, where $A=\supp p$, (Proposition~\ref{prop:Lie-sufficient}) but
that associative algebras arise precisely in the case of binomials
(Corollary~\ref{cor:associative-rare}).  In fact, $\Op[a,b]{S}{\alpha x_a-\beta
x_b}$, with $\alpha,\beta$ possibly $0$, are examples found throughout the
literature. For example when $\vav=1$, $\Op{t}{x_1-x_0}=\{\omega\in\Omega \mid
\la t| \omega_1 v_1\ra = \omega_0\la t|v_1\ra\}$ is the centralizer subring of
the linear transformation $t$.  If $\vav=2$, we switch to $\Omega_{1,2}$
instead, so that
\begin{align}\label{def:adj}
	\Op[1,2]{t}{x_1-x_2} & = \{\omega\in \Omega_{1,2}
	 \mid \la t| v_2 \omega_2^{\op}, v_1\ra = \la t|v_2, \omega_1 v_1\ra\}
	= \mathrm{Adj}(t)
\end{align}
is the customary algebra of \emph{adjoints}, see Example~\ref{ex:nuclei}.

\subsection{Transverse Groups \& Categories}

Our next two results concern what happens when we replace $\Omega=\prod_{a}
\End(V_a)$ with $\prod_a V_a\oslash U_a$, i.e. if replace the $\Omega$ in
\eqref{def:Omega-1} with tuples of rectangular matrices.  In abstract terms we
are headed towards the transverse categories of tensors with the goal of
studying problems about clustering in data.  Later we also use this to develop
data types and algorithms of Section~\ref{sec:algorithms}.

To reach the categories, we first study groups of transverse operators.  Let
$\Aut(V_a)$ be the group of units of $\End(V_a)$, and for $A\subseteq \zrange{\vav}$,
$B\subset\comp{A}$, set $\Omega_{A,B}^{\times} = \prod_{a\in A}\Aut(V_a)\times 
\prod_{b\in B}\Aut(V_b)^{\op}$.   Define 
$\Op[A,B]{S}{P}^{\times}=\Op[A,B]{S}{P}\cap \Omega_{A,B}^{\times}$.
Applying work of Eisenbud--Sturmfels~\cite{ES} we prove:

\begin{mainthm}\label{thm:group-intro}
	Fix $p\in K[X]$.  There exists $A\subset \zrange{\vav}$ and $B\subset\comp{A}$ 
	such that for all tensor spaces $T$, $\Op[A,B]{T}{p}^{\times}$ is a
	subgroup of $\Omega_{A,B}^{\times}$ if, and only if, there exists
	$e,f:\zrange{\vav}\to \{0,1\}$ with $A\supset \supp e$ and $B
	\supset \supp f$ disjoint, such that for all tensor spaces $T$,
	$\Op[A,B]{T}{p}^{\times}=\Op[A,B]{T}{X^e-X^f}^{\times}$.
\end{mainthm}

In fact, we prove that for general $P\subset K[X]$, if $\Op{T}{P}^{\times}$ forms a
subgroup of $\Omega_{A,B}^{\times}$, then again there are
$e_i,f_i:\zrange{\vav}\to \{0,1\}$ where for all $T$,
$\Op{T}{P}^{\times}=\Op{T}{(X^{e_1}-X^{f_1},\ldots, X^{e_m}-X^{f_m})}^{\times}$.
With further conditions on the supports of the $(e_i,f_i)$, we can prove also the
converse.

In any category, the automorphisms of the objects form a group, and thus
Theorem~\ref{thm:group-intro} places severe restrictions the kinds of transverse
tensor categories.  Basically, the only options are the obvious options:
fix a partition $\zrange{\vav}=A\sqcup B\sqcup C$ of covariant axes $A$,
contravariant axes $B$, and constant axes $C$.  Then compose transverse operators
axis-by-axis according to the variance.  This is still an exponential number of
categories each having the same objects (tensors) which explains some of the
complexity and diverse uses of tensors.

\subsection{Restriction and Stanley--Reisner ideals}
\label{subsec:restriction}

We now consider another concrete problem from the literature.  In the
\emph{Block-Decomposition Problem (BDP)}, we are given a $(d_0\times\cdots \times
d_{\vav}$) multiway array $t$ as in \eqref{def:T-1}.  The task is to change the
bases of some subset $A$ of the axes, so that the new array is block diagonal (or
triangular) along the $ab$-faces, for $a,b\in A$.  Such decompositions appear as
the question of finding clusters or outliers in social network data
\citelist{\cite{chatroom}\cite{engineer}\cite{MFMc}}, and in computational
algebra, this is how to build $\oplus$-decompositions, composition series of
modules, and more \citelist{ \cite{Holt-Rees}\cite{Ivanyos-Lux}
\cite{Maglione:efficient}\cite{Wilson:central}}. Applying the formalism
developed by our previous result, BDP asks: given a tensor in some
transverse tensor category, find a proper nontrivial subtensor in that category.
Our next result uses the Correspondence Theorem to identify polynomials that signal
the existence of subtensors; it further connects to the study of
tensor singularities.

First we observe that the many categories of subtensors organize into a
simplicial complex. For $\bra{t}\in V_0\oslash\cdots\oslash V_{\vav}$ and submodules
$U_0<V_0$, and $0\neq U_a\leq V_a$ for $a>0$, define
\begin{align}\label{def:nabla}
	\nabla(t;U) & = \{A\subset\zerovav\mid U_A\not\perp V_{\bar{A}}\},
		& U_A\bot V_{\bar{A}} \Leftrightarrow \left\{\begin{array}{ll}
			\langle t|U_{B},V_{\range{\vav}-B}\rangle\not\leq U_0 & A=\{0\}\sqcup B,\\
			\langle t|U_A,V_{\range{\vav}-A}\rangle\neq 0 & \textnormal{otherwise.} 
		\end{array}\right.
\end{align}
For example, consider the top-cell of this simplicial complex missing, i.e.~$\bra{t}U_1,\ldots,U_{\vav}\rangle\leq U_0$.  Writing
$U_0^{\bot}=\{\pi:V_0\to K\mid \pi(U_0)=0\}$, this becomes $U_0^{\bot}\langle
t|U_1,\ldots,U_{\vav}\rangle=0$ which exposes how restricting $\bra{t}$ to
$U_0\oslash\cdots\oslash U_{\vav}$ implicitly requires the existence of a tensor singularity in the right configuration -- the top cell configuration in this case.
Many familiar concepts including left and right ideals and orthogonal subspaces
are captured in this complex.  See the examples in
Section~\ref{sec:Singularities}.

Now consider those operators $\omega$ that factor through a subtensor in the
following sense
\begin{align}\label{def:op-res}
	\Omega(U,V)=\{\omega\in
	\Omega\mid (\forall a)(\omega_a(V_a)\leq U_a)\}.
\end{align}
Notice $\Omega(U,V)$ is a right ideal of $\Omega$. Indeed, when $K$ is a field
every right ideal $\Delta$ of $\Omega$ has this form.  Then
$\Id{t}{\Omega(U,V)}$ is affected by both the tensor $t$ and the limitations
brought on by the $U$.  The following precisely calculates this ideal.

For an abstract simplicial complex $\nabla$ on $\zerovav$, the 
\emph{Stanley--Reisner ideal} is $(X^e\mid \supp e \notin \nabla)$.

\begin{mainthm}\label{mainthm:singularity}
	For fields $K$, if $U_0<V_0$ and for $a>0$, $0\neq U_a\leq V_a$, then
	$\nabla(t;U)$ is a simplicial complex and $\Id{t}{\Omega(U,V)} =(X^e \mid
	\supp e\notin \nabla(t;U))$ is its associated Stanley--Reisner ideal.
\end{mainthm}

This proffers a generic decomposition algorithm: sample transverse operators in
such a manner that favors the discovery of an operator $\omega$ where
$\Id{t}{\omega}$ contains a monomial -- this we can test effectively.  This
captures the high-level reasoning in algorithms of \citelist{\cite{Holt-Rees}
\cite{Ivanyos-Lux}\cite{Wilson:central}} and applies it to tensors in general. A
detailed decomposition algorithm is a subject for future work.

\subsection*{Acknowledgements}

The second and third authors are grateful for the support of the Hausdorff
Institute for Mathematics, during the trimester on \emph{Logic and Algorithms in
Group Theory}, and the Isaac Newton Institute for Mathematical Sciences, during
the program \emph{Groups representations and Applications}, where some of this
research was conducted.  The third author thanks Aner Shalev and Alex Lubotzky
for hosting him at the Hebrew University where this research began.  We also
thank Laurent Bartholdi, Peter Brooksbank, and  Bill Kantor for many answers.

%==============================================================================
%\input{examples.tex}
%==============================================================================

\section{Examples of traits of transverse operators}
\label{sec:foundations}

The following examples give useful intuition about our 
correspondence.  Throughout this paper we call the elements of $\Id{S}{\Delta}$
the traits of $\Delta$ over $S$.

\subsection{Traits over Matrices}

Within the literature, matrices are often the first example of
a tensor.  Such a statement usually assumes from context how this matrix should
be interpreted as a multilinear map (or form).  Specifically, letting
$T=\mathbb{M}_{m\times n}(K)$, there are at least the following three distinct
and natural ways in which matrices $M\in T$ can be regarded as multilinear maps:
\begin{enumerate}[(i)]
	\item $\bra{M}:K^n\to K^m$ where $\langle M|v\rangle=Mv$.  So
	$\bra{\cdot}:T\to K^m\rversor K^n$ and $\vav=1$.
	\item $\bra{M}:K^m\to K^n$ where $\langle M|u\rangle=u^{\dagger}M$.  So
	$\bra{\cdot}:T\to K^n\rversor K^m$ and $\vav=1$. 
	\item To treat $M$ as affording a bilinear form, use $\bra{\cdot}:T\to
	K\oslash K^m\oslash K^n$ where $\langle M|u,v\rangle=u^{\dagger}Mv$.
	Here $\vav=2$.\footnote{
		While infix notation like $\bra{u}M\ket{v}$, or $u*v$, is convenient, it
		only applies to valence $\vav=2$, so we use it sparingly.
	}
\end{enumerate}

Let us assume the interpretation (i), writing $V_1=K^b$, $V_0=K^a$ and regarding
$\bra{M}$ as an element of $\hom(K^b,K^a)=V_0\rversor V_1$. We shall write
$K[x,y]$ instead of $K[x_0,x_1]$ to save on subscripts. Then a transverse
operator is a pair of matrices $\omega = (X,Y)\in \End(V_0)\times\End(V_1)=
\mathbb{M}_{m}(K)\times \mathbb{M}_{n}(K)$, and the product  $\bra{M}\omega$
(see \ref{eq:p-omega-action}) is given by $\bra{X M Y}$. The induced  right
$K[x,y]$-module structure on $\mathbb{M}_{m\times n}(K)$ is determined by
\[
	M \cdot x =X  M \qquad\text{and}\qquad M\cdot y=MY.
\]
Thus,
\begin{align}
	\Id{M}{\omega}=
	\Ann_{K[x,y]}^{(X,Y)}(M)=\left\{ 
		\sum_{i,j\in\mathbb{N}}\lambda_{ij} x^i y^j\in K[x ,y]~\middle|~
		\sum_{i,j\in\mathbb{N}}\lambda_{ij} X^i M Y^j=0	
	\right\}.	
\end{align}
We call that the elements of  $\Id{M}{\omega}$ the \emph{traits} of
$\omega=(X ,Y)$ relative to $M$. Generators for $\Id{M}{\omega}$ can be found by
inspecting the actions of $x$ and $y$ on $M$. 

\begin{figure}[!htbp]
	\begin{subfigure}{0.44\textwidth}
        \centering
%==================================================================================
%        \input{ann-1.tex}
        \begin{tikzpicture}
            \node () at (0,-4) {};
            \node (A.i) at (0,2) {
                    $X=\left[\begin{smallmatrix} 0 & 0\\ 0 & 1 \end{smallmatrix}\right],
                    Y=\left[\begin{smallmatrix} 0 & 0 & 0\\ 0 & 0 & 0 \\ 0 & 0 & 1 \end{smallmatrix}\right]$
                };
                \node (A) at (0,0) {	\begin{tikzpicture}
                    \node (T) at (0,0) {
                        $\begin{bsmallmatrix} 
                            1 & 2 & 3 \\ 
                            2 & 3 & 0 
                        \end{bsmallmatrix}$};
                    \node (TY) at (3,0) {
                        $\begin{bsmallmatrix} 
                        0 & 0 & 3 \\ 
                        0 & 0 & 0 
                        \end{bsmallmatrix}$};
            
                    \node (XT) at (0,-2) {
                        $\begin{bsmallmatrix}
                        0 & 0 & 0 \\ 
                        2 & 3 & 0
                        \end{bsmallmatrix}$};
                    \node (XTY) at (3,-2) {$0$};
                    
                    \draw[thick,->] (T) -- (TY) node[midway,above] {$Y$};
                    \draw[thick,->] (TY) edge[out=-45, in=45, looseness=4] (TY);
            
                    \draw[thick,->] (T) -- (XT) node[midway,left] {$X$};
                    \draw[thick,->] (XT) edge[out=-45, in=-135, looseness=4] (XT);

                    \draw[thick,->] (XT) -- (XTY) node[midway,above] {$Y$};
                    \draw[thick,->] (TY) -- (XTY) node[midway,left] {$X$};
            
                    \draw[thick,->] (XTY) edge[out=-45, in=45, looseness=4] (XTY);
                    \draw[thick,->] (XTY) edge[out=-45, in=-135, looseness=4] (XTY);
                \end{tikzpicture}};
            
            \end{tikzpicture}
%==================================================================================
		\caption{$\Ann_{K[x,y]}^{(X,Y)}\left(\left[\begin{smallmatrix}
			1 & 2 & 3 \\ 
			2 & 3 & 0 \end{smallmatrix}\right]\right)=(x^2-x, y^2-y, xy)$}
		\label{fig:ex-ann-idempotent}
	\end{subfigure}\quad%
	\begin{subfigure}{0.53\textwidth}
		\centering
%==================================================================================
%    \input{ann-2.tex}
        \begin{tikzpicture}
        \node (A.i) at (0,3) {$X'=\left[\begin{smallmatrix} 0 & 1\\ 0 & 0 \end{smallmatrix}\right],Y'=\left[\begin{smallmatrix} 0 & 0 & 0\\ 1 & 0 & 0 \\ 0 & 1 & 0 \end{smallmatrix}\right]$};
        \node (a) at (0,0) {	\begin{tikzpicture}
        \node (T) at (0,0) {
            $\begin{bsmallmatrix} 
                1 & 2 & 3 \\ 
                2 & 3 & 0 
            \end{bsmallmatrix}$};
        \node (TY) at (2,0) {
            $\begin{bsmallmatrix} 
            2 & 3 & 0 \\ 
            3 & 0 & 0 
            \end{bsmallmatrix}$};

            \node (TYY) at (4,0) {
                $\begin{bsmallmatrix} 
                3 & 0 & 0 \\ 
                0 & 0 & 0 
            \end{bsmallmatrix}$};

            \node (TYYY) at (6,0) {$0$};

            \node (XT) at (0,-2) {
            $\begin{bsmallmatrix}
            2 & 3 & 0 \\ 
            0 & 0 & 0
            \end{bsmallmatrix}$};
        \node (XTY) at (2,-2) {
            $\begin{bsmallmatrix} 
            3 & 0 & 0 \\ 
            0 & 0 & 0 
            \end{bsmallmatrix}$};

        \node (XXT) at (0,-4) {$0$};
        \node (XXTY) at (2,-4) {$0$};
        \node (XTYY) at (4,-2) {$0$};

        \draw[thick,->] (T) -- (TY) node[midway,above] {$Y'$};
        \draw[thick,->] (TY) -- (TYY) node[midway,above] {$Y'$};
        \draw[thick,->] (TYY) -- (TYYY) node[midway,above] {$Y'$};
        \draw[thick,->] (TYYY) edge[out=-45, in=45, looseness=4] (TYYY);
        \draw[thick,->] (TYYY) edge[out=-45, in=-135, looseness=4] (TYYY);

        \draw[thick,->] (T) -- (XT) node[midway,left] {$X'$};
        \draw[thick,->] (TYY) -- (XTYY) node[midway,left] {$X'$};
        \draw[thick,->] (XTY) -- (XTYY) node[midway,above] {$Y'$};

        \draw[thick,->] (XT) -- (XXT) node[midway,left] {$X'$};
        \draw[thick,->] (XTY) -- (XXTY) node[midway,left] {$X'$};

        \draw[thick,->] (XT) -- (XTY) node[midway,above] {$Y'$};
        \draw[thick,->] (TY) -- (XTY) node[midway,left] {$X'$};

        \draw[thick,->] (XTYY) edge[out=-45, in=45, looseness=4] (XTYY);
        \draw[thick,->] (XTYY) edge[out=-45, in=-135, looseness=4] (XTYY);

        \draw[thick,->] (XXT) -- (XXTY) node[midway, above] {$Y'$};
        \draw[thick,->] (XXT) edge[out=-45, in=-135, looseness=4] (XXT);

        \draw[thick,->] (XXTY) edge[out=-45, in=45, looseness=4] (XXTY);
        \draw[thick,->] (XXTY) edge[out=-45, in=-135, looseness=4] (XXTY);
        \end{tikzpicture}};

        \end{tikzpicture}
%==================================================================================
    \caption{$\Ann_{K[x,y]}^{(X',Y')}\left(\left[\begin{smallmatrix}
			1 & 2 & 3 \\ 
			2 & 3 & 0 
		\end{smallmatrix}\right]\right)=(x^2, y^3, xy-y^2)$}
		\label{fig:ex-ann-nilpotent}
	\end{subfigure}
	\caption{Transverse tensor operators acting on a tensor revealing the annihilator.}
 	\label{fig:annihilator}
\end{figure}
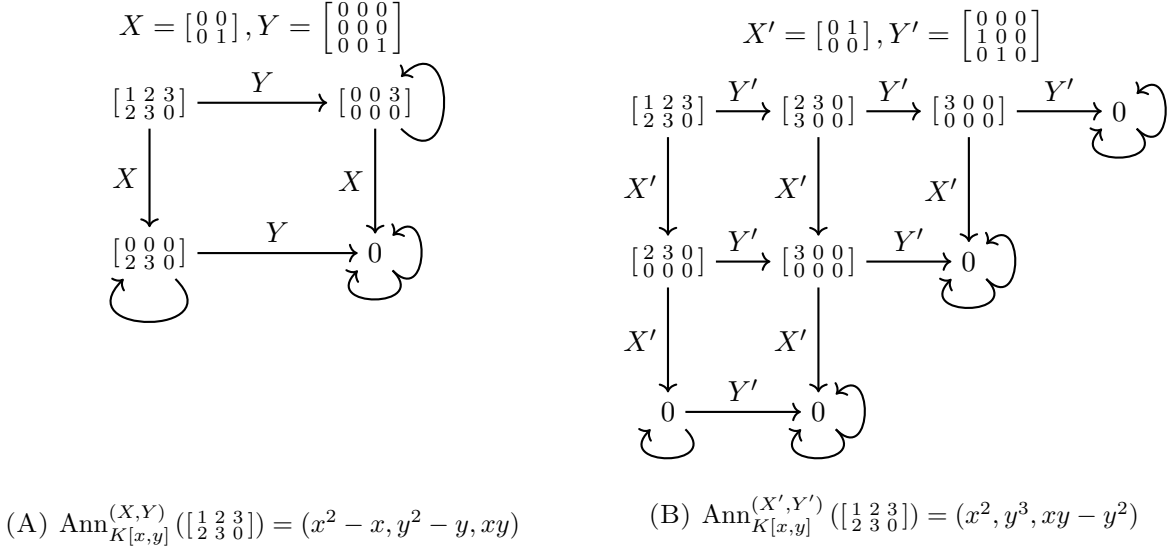

Figure~\ref{fig:annihilator} shows what is known as a Penrose tensor-network
diagram. We use it to compute annihilators.  Within $\Ann^{\omega}_{K[X]}(M)$ we
find the annihilators $\Ann_{K[x]}^{X}(M)$ and $\Ann_{K[y]}^Y(M)$, i.e.\ the
relations we see along the rows or columns.  So the usual spectral theory of the
individual operators $X$ and $Y$ persists, we also see new relations if we trace
paths within the grid. The right-hand example has a symmetry which is
responsible for binding two monomials into a binomial trait $xy-y^2\in
\Ann^{\omega}_{K[x,y]}(M)$. The left example has $xy$ as a trait because the
$(2,3)$ entry in the matrix $M$ is $0$---replacing this entry with any nonzero
scalar would remove $xy$ form the annihilator.  (Here we see a first indication
of a relation between traits and singularities of tensors, which discuss
extensively in Section~\ref{sec:Singularities}.)

In valence $\vav=3$, matrices are replaced by 3-dimensional hypermatrices
 and transverse operators $\omega\in \End(V_0)\times
\End(V_1)\times\End(V_2)$ act on the length, width, and height of the
hypermatrix by tensor contraction. Figure~\ref{fig:annihilator-2} illustrates
the Penrose diagrams of such a case and their annihilators. The variables
$(x_0,x_1,x_2)$ are written as $(x,y,z)$.

\begin{figure}[!htbp]
	\begin{center}
\begin{subfigure}{0.4\textwidth}
	%% A transverse op on GHZ state.
	\begin{tikzpicture}
		\pgfmathsetmacro{\xx}{0.3}
		\pgfmathsetmacro{\yy}{0.3}
		\pgfmathsetmacro{\zz}{0.3}

		%% Weird bug, using label (A) arrows go wonky, using label (A0) all good.
		\node (A0) at (0,0,0) {\begin{tikzpicture}
			% <000|
			\pic at (0*\xx,1*\yy,1*\zz) {ccube={\xx/\yy/\zz/gray}};
			% <111|
			\pic at (1*\xx,0*\yy,0*\zz) {ccube={\xx/\yy/\zz/gray}};
		
			\pic at (0*\xx,0*\yy,0*\zz) {linecube={2*\xx/2*\yy/2*\zz}};
		\end{tikzpicture}};

		\node (XA) at (4,0,0) {\begin{tikzpicture}
			% <011|
			\pic at (0*\xx,0*\yy,0*\zz) {ccube={\xx/\yy/\zz/gray}};
			% <100|
			\pic at (1*\xx,1*\yy,1*\zz) {ccube={\xx/\yy/\zz/gray}};
		
			\pic at (0*\xx,0*\yy,0*\zz) {linecube={2*\xx/2*\yy/2*\zz}};
		\end{tikzpicture}};

		\node (AY) at (0,-4,0) {\begin{tikzpicture}
			% <010|
			\pic at (0*\xx,0*\yy,1*\zz) {ccube={\xx/\yy/\zz/gray}};
			% <101|
			\pic at (1*\xx,1*\yy,0*\zz) {ccube={\xx/\yy/\zz/gray}};
		
			\pic at (0*\xx,0*\yy,0*\zz) {linecube={2*\xx/2*\yy/2*\zz}};
		\end{tikzpicture}};

		\node (XAY) at (4,-4,0) {\begin{tikzpicture}
			% <110|
			\pic at (1*\xx,0*\yy,1*\zz) {ccube={\xx/\yy/\zz/gray}};
			% <001|
			\pic at (0*\xx,1*\yy,0*\zz) {ccube={\xx/\yy/\zz/gray}};
		
			\pic at (0*\xx,0*\yy,0*\zz) {linecube={2*\xx/2*\yy/2*\zz}};
		\end{tikzpicture}};

		\draw[thick,<->] (A0) -- (AY) node[midway,left] {$Y$};
		\draw[thick,<->] (A0) -- (XA) node[midway,above] {$X$};
		
		\draw[thick,<->] (XA) -- (XAY) node[midway,right] {$Y$};
		\draw[thick,<->] (AY) -- (XAY) node[midway,above] {$X$};

		\draw[thick,<->] (A0) -- (XAY) node[midway,above] {$Z$};
		\draw[thick,<->] (XA) -- (AY) node[midway,right] {$Z$};

		\node () at (3,2.5,0) {$X=Y=Z=\begin{bmatrix} 0 & 1\\ 1 & 0 \end{bmatrix}$};

		\node () at (1.5,3.5,0) {$\bra{GHZ}=\bra{000}+\bra{111}=$};%%=e_0\otimes e_0\otimes e_0+e_1\otimes e_1\otimes e_1=$};
		\node () at (4.5,3.5,0) {\begin{tikzpicture}
		% <000|
		\pic at (0*\xx,1*\yy,1*\zz) {ccube={\xx/\yy/\zz/gray}};
		% <111|
		\pic at (1*\xx,0*\yy,0*\zz) {ccube={\xx/\yy/\zz/gray}};

		\pic at (0*\xx,0*\yy,0*\zz) {linecube={2*\xx/2*\yy/2*\zz}};
	\end{tikzpicture}};
	\end{tikzpicture}
    \caption{$\Id{GHZ}{(X,Y,Z)}=$\\$(x^2-1,y^2-1,z^2-1, xy-z, x-yz, y-xz)$.}\label{fig:annhilator-2a}
\end{subfigure}
\hspace{1in}
\begin{subfigure}{0.4\textwidth}
	%% Transverse Op on W state
	\begin{tikzpicture}
		\pgfmathsetmacro{\xx}{0.3}
		\pgfmathsetmacro{\yy}{0.3}
		\pgfmathsetmacro{\zz}{0.3}

		%% Weird bug, using label (A) arrows go wonky, using label (A0) all good.
		\node (A0) at (0,0,0) {\begin{tikzpicture}
			% <001|
			\pic at (0*\xx,1*\yy,0*\zz) {ccube={\xx/\yy/\zz/gray}};
			% <100|
			\pic at (0*\xx,0*\yy,1*\zz) {ccube={\xx/\yy/\zz/gray}};
			% <010|
			\pic at (1*\xx,1*\yy,1*\zz) {ccube={\xx/\yy/\zz/gray}};
		
			\pic at (0*\xx,0*\yy,0*\zz) {linecube={2*\xx/2*\yy/2*\zz}};
		\end{tikzpicture}};

		\node (A^Z) at (0,0,-4) {\begin{tikzpicture}
			% <101|
			\pic at (0*\xx,0*\yy,0*\zz) {ccube={\xx/\yy/\zz/gray}};
			% <011|
			\pic at (1*\xx,1*\yy,0*\zz) {ccube={\xx/\yy/\zz/gray}};
			% <000|
			\pic at (0*\xx,1*\yy,1*\zz) {ccube={\xx/\yy/\zz/gray}};
		
			\pic at (0*\xx,0*\yy,0*\zz) {linecube={2*\xx/2*\yy/2*\zz}};
		\end{tikzpicture}};

		\node (XA) at (4,0,0) {\begin{tikzpicture}
			% <001|
			\pic at (1*\xx,1*\yy,0*\zz) {ccube={\xx/\yy/\zz/gray}};
			% <100|
			\pic at (1*\xx,0*\yy,1*\zz) {ccube={\xx/\yy/\zz/gray}};
			% <010|
			\pic at (0*\xx,1*\yy,1*\zz) {ccube={\xx/\yy/\zz/gray}};
		
			\pic at (0*\xx,0*\yy,0*\zz) {linecube={2*\xx/2*\yy/2*\zz}};
		\end{tikzpicture}};

		\node (AY) at (0,-4,0) {\begin{tikzpicture}
			% <011|
			\pic at (0*\xx,0*\yy,0*\zz) {ccube={\xx/\yy/\zz/gray}};
			% <110|
			\pic at (0*\xx,1*\yy,1*\zz) {ccube={\xx/\yy/\zz/gray}};
			% <000|
			\pic at (1*\xx,0*\yy,1*\zz) {ccube={\xx/\yy/\zz/gray}};
		
			\pic at (0*\xx,0*\yy,0*\zz) {linecube={2*\xx/2*\yy/2*\zz}};
		\end{tikzpicture}};

		\node (XAY) at (4,-4,0) {\begin{tikzpicture}
			% <111|
			\pic at (1*\xx,0*\yy,0*\zz) {ccube={\xx/\yy/\zz/gray}};
			% <100|
			\pic at (0*\xx,0*\yy,1*\zz) {ccube={\xx/\yy/\zz/gray}};
			% <010|
			\pic at (1*\xx,1*\yy,1*\zz) {ccube={\xx/\yy/\zz/gray}};
		
			\pic at (0*\xx,0*\yy,0*\zz) {linecube={2*\xx/2*\yy/2*\zz}};
		\end{tikzpicture}};

		\node (AY^Z) at (0,-4,-4) {\begin{tikzpicture}
			% <111|
			\pic at (1*\xx,0*\yy,0*\zz) {ccube={\xx/\yy/\zz/gray}};
			% <010|
			\pic at (0*\xx,1*\yy,0*\zz) {ccube={\xx/\yy/\zz/gray}};
			% <001|
			\pic at (0*\xx,0*\yy,1*\zz) {ccube={\xx/\yy/\zz/gray}};
		
			\pic at (0*\xx,0*\yy,0*\zz) {linecube={2*\xx/2*\yy/2*\zz}};
		\end{tikzpicture}};

		\node (XA^Z) at (4,0,-4) {\begin{tikzpicture}
			% <001|
			\pic at (1*\xx,0*\yy,0*\zz) {ccube={\xx/\yy/\zz/gray}};
			% <111|
			\pic at (0*\xx,1*\yy,0*\zz) {ccube={\xx/\yy/\zz/gray}};
			% <100|
			\pic at (1*\xx,1*\yy,1*\zz) {ccube={\xx/\yy/\zz/gray}};
		
			\pic at (0*\xx,0*\yy,0*\zz) {linecube={2*\xx/2*\yy/2*\zz}};
		\end{tikzpicture}};

		\node (XAY^Z) at (4,-4,-4) {\begin{tikzpicture}
			% <101|
			\pic at (0*\xx,0*\yy,0*\zz) {ccube={\xx/\yy/\zz/gray}};
			% <110|
			\pic at (1*\xx,0*\yy,1*\zz) {ccube={\xx/\yy/\zz/gray}};
			% <011|
			\pic at (1*\xx,1*\yy,0*\zz) {ccube={\xx/\yy/\zz/gray}};
		
			\pic at (0*\xx,0*\yy,0*\zz) {linecube={2*\xx/2*\yy/2*\zz}};
		\end{tikzpicture}};

		\draw[thick,<->] (A0) -- (A^Z) node[midway,above] {$Z$};
		\draw[thick,<->] (A0) -- (AY) node[midway,left] {$Y$};
		\draw[thick,<->] (A0) -- (XA) node[midway,above] {$X$};

		\draw[thick,<->] (A^Z) -- (XA^Z) node[midway,above] {$X$};
		\draw[thick,<->] (A^Z) -- (AY^Z) node[midway,left] {$Y$};
		
		\draw[thick,<->] (XA) -- (XA^Z) node[midway,above] {$Z$};
		\draw[thick,<->] (XA) -- (XAY) node[midway,right] {$Y$};

		\draw[thick,<->] (AY) -- (AY^Z) node[midway,above] {$Z$};
		\draw[thick,<->] (AY) -- (XAY) node[midway,above] {$X$};

		\draw[thick,<->] (XAY) -- (XAY^Z) node[midway,above] {$Z$};
		\draw[thick,<->] (XA^Z) -- (XAY^Z) node[midway,right] {$Y$};
		\draw[thick,<->] (AY^Z) -- (XAY^Z) node[midway,above] {$X$};

	\node () at (2,2.5,0) {$X=Y=Z=\begin{bmatrix} 0 & 1\\ 1 & 0 \end{bmatrix}$};

	\node () at (2,3.5,0) {$\bra{W}=\bra{001}+\bra{010}+\bra{100}=$};
	\node () at (5,3.5,0) {\begin{tikzpicture}
		% <001|
		\pic at (0*\xx,1*\yy,0*\zz) {ccube={\xx/\yy/\zz/gray}};
		% <100|
		\pic at (0*\xx,0*\yy,1*\zz) {ccube={\xx/\yy/\zz/gray}};
		% <010|
		\pic at (1*\xx,1*\yy,1*\zz) {ccube={\xx/\yy/\zz/gray}};

		\pic at (0*\xx,0*\yy,0*\zz) {linecube={2*\xx/2*\yy/2*\zz}};
	\end{tikzpicture}};
	\end{tikzpicture}
    \caption{$\Id{W}{(X,Y,Z)}=(x^2-1,y^2-1,z^2-1)$.\\}\label{fig:annhilator-2b}
\end{subfigure}
\caption{Annihilators in larger valence.}\label{fig:annihilator-2}
\end{center}
\end{figure}
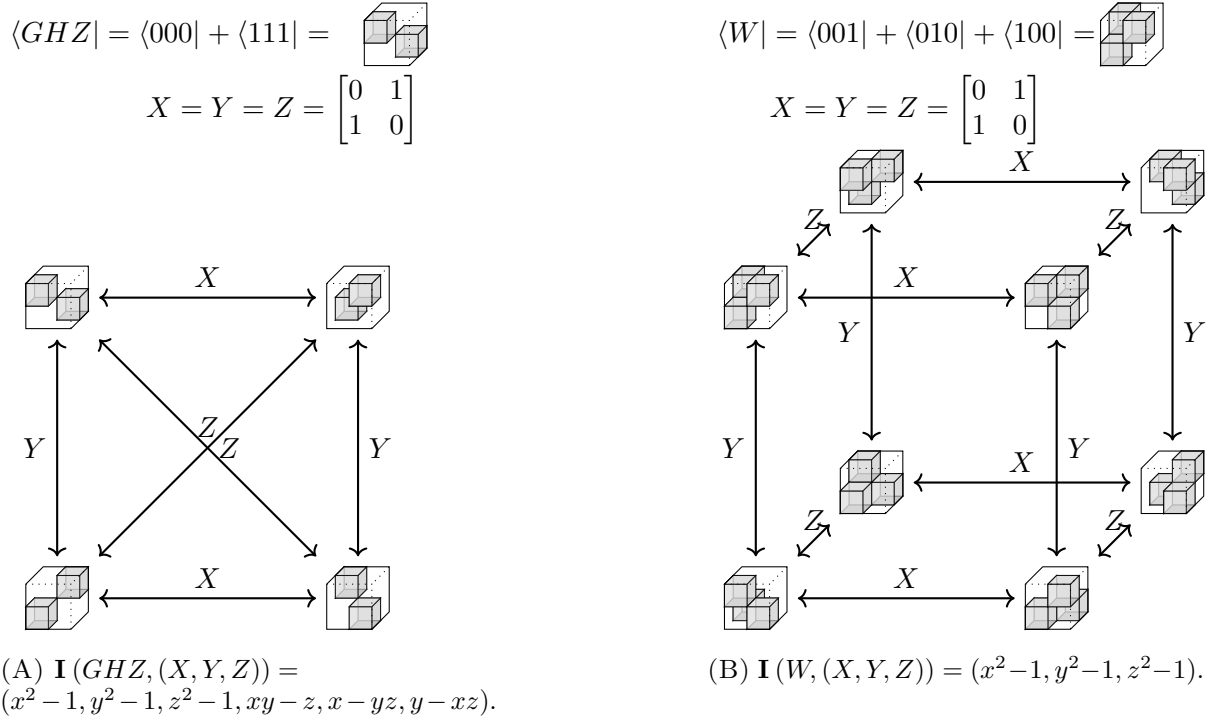

\subsection{Operators characterized by specific traits}

The remaining examples in this section illustrate how important families of
transverse operators arise with a common trait (or traits) relative to the given
tensor. Unless indicated otherwise, $\bra{\cdot}:T\to V_0\rversor\dots \rversor
V_{\vav}$ is a tensor space over a commutative (unital) ring $K$.

\begin{ex}\label{ex:autotopism}
	A transverse operator $\omega\in \Omega^\times$ is an \emph{autotopism} of
	$t\in T$ if for all $v\in \prod_{a=1}^{\vav}V_a$, $\omega_0\bra{t} v\ra
	=\bra{t} \omega_1 v_1,\ldots,\omega_{\vav}v_{\vav}\ra$. Using infix notation
	when $\vav = 2$, we write $v_1\ast v_2=\bra{t} v_1,v_2\ra$; this becomes
	$\omega_0(v_1\ast v_2)=\omega_1v_1\ast \omega_2v_2$. So if $V_0=V_1=V_2$,
	then we recover A.\ A.\ Albert's autotopisms of the non-associative algebra
	$(V_0,\ast)$. It is immediate from the definition that the autotopisms of
	$t$ are the invertible transverse operators having $x_0 - x_1\cdots
	x_{\vav}$ as a trait.
\end{ex} 

\begin{ex}
	An \emph{isometry} of $t\in T$ is an element $\omega\in
	\Omega_{\range{\vav}}$ such that $\bra{t} v\ra = \bra{t} \omega_1
	v_1,\ldots, \omega_{\vav}v_{\vav}\ra$, for all $v\in
	\prod_{a\in\range{\vav}}V_a$. Using infix notation when $\vav = 2$, we write
	$v_1\ast v_2=\bra{t} v_1,v_2\ra$; this becomes $ v_1\ast v_2
	=\omega_1v_1\ast \omega_2v_2$. The isometries of $t$ are
	$\Op[\range{\vav}]{t}{1-x_1\cdots x_{\vav}}^\times$. Said differently, the
	isometries of $t$ are the transverse operators having $1-x_1\cdots x_{\vav}$
	as a trait upon ignoring the $0$-axis.
\end{ex}

\begin{ex}\label{ex:nuclei}
	Let $S\subset T$. For every  $a,b\in \zerovav$ with $a<b$, define the
	$(a,b)$-\emph{nucleus} of $S$ by 
	\[
		\Nuc_{a,b}(S)=\Op[\{a,b\}]{S}{x_a-x_b}.
	\]
	The $(a,b)$-nucleus is also known as the $(a,b)$-\emph{scalar ring}. The
	nuclei have a structure of  associative $K$-algebras, and they play
	important role in the study of tensors, see \citelist{\cite{Wilson:division}
	\cite{Ivanyos-Qiao} \cite{Li-Qiao} \cite{BMW:exactseq}}.

	In more detail, for $a>0$, $\Nuc_{a,b}(S)$ consists of pairs
	$(\omega_a,\omega_b)\in \End(V_a)\times\End(V_b)$ such that
	\begin{align*}
		\bra{t}\omega_a v_a,v_{\bar{a}}\ra & = \bra{t}\omega_b v_b,v_{\bar{b}}\ra.
	\end{align*}
	Note that $\Nuc_{a,b}(S)$ is an associative unital $K$-subalgebra of
	$\End(V_a)^\op\times\End(V_b)$. Moreover, every associative $K$-algebra
	$\Delta$ acting on $V_a$ from the right and on $V_b$ from the left such that
	$\bra{S}\subset \hom(V_a\otimes_{\Delta} V_b\otimes (\bigotimes_{c\in
	\overline{\{a,b\}}} V_c),V_0)$ factors through $\Nuc_{a,b}(S)$ via its
	action on $V_a$ and $V_b$. Thus, the $(a,b)$-nucleus is universal for this
	property.

	When $a=0$, the set $\Nuc_{0,b}(S)$ consists of pairs
	$(\omega_0,\omega_b)\in \End(V_0)\times\End(V_b)$ such that
	\begin{align*}
		\omega_0\bra{t}v\ra & = \bra{t}\omega_b v_b,v_{\bar{b}}\ra.
	\end{align*}
	It is a unital $K$-subalgebra of $\End(V_0)\times\End(V_b)$ and universal
	among the associative $K$-algebras $\Delta$ acting on $V_0$ and $V_b$ such
	that $S\subset V_0\oslash_{\Delta} (V_{1}\otimes\cdots\otimes V_{\vav})$.
\end{ex}

\begin{ex}\label{ex:centroid}
	The \emph{centroid} of $t\in T$ is 
	\[
		\Cent(t):=\Op{t}{\{x_0-x_a\,|\,a=1,\dots,\vav\}}.
	\]
	Note that $\Cent(t)$ is a unital $K$-subalgebra of
	$\prod_{a\in\zrange{\vav}}\End(V_a)$ consisting of transverse operators
	$\omega\in\Omega$ such that for all $a\in \range{\vav}$ and all $v$, 
	$\omega_0\bra{t} v\ra =\bra{t} \omega_a	v_a,v_{\bar{a}}\ra$. Regarding each of
	the $V_a$ as a module over $\Cen(t)$, $\bra{t}$ is $\Cent(t)$-multilinear.
	Again, $\Cent(t)$ is universal for this property; namely, if $L$ is a ring
	acting on all the $V_a$ such that $\bra{t}$ is $L$-multilinear, then $L$
	factors uniquely through $\Cent(t)$ (in a way which is compatible with the
	module structures of the $V_a$). In valence $\vav=2$, this observation is a
	theorem of Myasnikov~\cite{Myasnikov}. If $\bra{t}$ is the product of a
	non-associative $K$-algebra, then $\Cent(t)$ is the algebra's centroid. In
	applications, one can usually replace $K$ with the centroid, which is a
	commutative ring if $t$ is fully nondegenerate.
\end{ex}

\begin{ex}\label{ex:derivative}
	We  observed in Section~\ref{subsec:densor} that derivations
	of a tensor space $ T$ are the transverse operators
	having $x_0-x_1-\cdots -x_{\vav}$ as a trait.
\end{ex}

Table~\ref{tab:common} summarizes
Examples~\ref{ex:autotopism}--\ref{ex:derivative} for valence $\vav=2$ and infix
notation $v_1\ast v_2=\bra{t} v_2,v_1\ra$.

\begin{table}[!htbp]
    \begin{tabular}{|l|l|l|l|}
        \hline
        Name & Property & Characterizing traits & Axes\\
        \hline
        Autotopism  $f$ & $f(u*v)=f(u)*f(v)$ & $x_2 x_1-x_0$ & $A=\{0,1,2\}$     \\
        Isometry $\omega$ & $( \omega u)\ast ( \omega v)= u\ast v$    &  $x_2 x_1-1$ & $A=\{1,2\}$\\
        Adjoint $a$ & $  (u a) \ast  v = u\ast (a v) $ & $x_2-x_1$ & $A=\{2\}, B=\{1\}$     \\
        Left scalar $\ell$ & $(\ell u) \ast v= \ell(u\ast v)$ & $x_2-x_0$ & $A=\{0,2\}$ \\  
        Scalar $\lambda$ & $(\lambda u)\ast v=\lambda(u\ast v)=u\ast (\lambda v)$ & $x_2-x_0$, $x_0-x_1$ & $A=\{0,1,2\}$     \\
        Derivation $\partial$ & $\partial (u\ast v)=(\partial u)\ast v + u\ast (\partial v)$ 
        & $x_0-x_1-x_2$ & $A=\{0,1,2\}$   \\
  		\hline
    \end{tabular}
    \caption{Common types of   operators associated to  bilinear maps
    $\ast: V_2\times V_1\bmto V_0$.}
    \label{tab:common}
\end{table}

\begin{ex}\label{ex:degeneration}
	When $K$ is a field, transverse operators detect degeneracy of the tensor
	space $T$. Extending the definition from Section~\ref{sec:terminology}, we
	say that $T$ is degenerate on the $a$-axis, for $a\neq 0$, if there is
	nonzero $v_a\in V_a$ such that $\bra{T} v_a,\prod_{b\in \bar{a}} V_b\ra=0$,
	and degenerate on the $0$-axis if $T$ is not full. 
	% For $a\neq 0$, $\bra{t} x_a(\omega)\ket{v} = \bra{t} \omega_a
	% v_a,v_{\bar{a}}\rangle$ and $\bra{t} x_0(\omega)\ket{v}=\omega_0\bra{t}
	% v\ra$.
	Since each $v_a\in V_a$ (resp.~$v_0\in V_0$) spans the image (resp.\ kernel)
	of some $\omega_a\in \End(V_a)$. Then the tensor space $T$ is nondegenerate
	in the $a$-axis if, and only if, $\Op[\{a\}]{T}{x_a}=0$. In particular, $T$
	is fully nondegenerate if, and only if, $\Op{T}{\{x_0,\dots,x_{\vav}\}}=0$.

	When $K$ is an arbitrary commutative ring, it is no longer true that every
	submodule of $V_a$ is the image of some $\omega_a\in\End(V_a)$. However, it
	is still true that if $T$ is nondegenerate in the axes $A\subset\zerovav$
	then $\Op[A]{T}{\{x_a\,|\,a\in A\}}=0$.
\end{ex}

%==============================================================================

%==============================================================================
%\input{correspondence.tex}

\section{Galois connection of tensors, polynomials, and transverse operators (Theorem~\ref{thm:correspondence})}
\label{sec:Galois}

Unless indicated otherwise, $K$ is a commutative ring, $\bra{\cdot}:T\to
V_0\rversor\dots\rversor V_{\vav}$ is a tensor space over $K$, $s,t$ are tensors
(of $T$ or of some other tensor space of context), $S$ ranges over subsets of
$T$, $P$ ranges over subsets of $K[X]$, and $\omega:(a\in \zrange{\vav})\to
(\omega_a\in \End(V_a))$ denote transverse operators. 

The material question embedded in our correspondence
Theorem~\ref{thm:correspondence} is how the points in the sets $\Op{S}{P}$ are
specified by a polynomial ideal, and how these ideals are independent of any
choices in coordinates that may be used to define them. This pushes us to enrich
transverse operators $\Op{S}{P}$ into an affine $K$-scheme. From that, we prove
the ternary Galois connection of Theorem~\ref{thm:correspondence} in this
context. In fact,  the scheme structure requires   the frame
$(V_0,\dots,V_{\vav})$ to consist of finitely generated (f.g.) projective
$K$-modules, so we will make this assumption throughout; we comment about the
non-projective case in Section~\ref{subsec:non-projective}. Readers who wish
to avoid this technicality may simply assume that $K$ is a field and the $V_a$
are finite-dimensional vector spaces.

Before we begin, we explain the necessity of schemes in the context of
transverse operators. The idea of $K[X]$-annihilators of tensor spaces relative
to transverse   operators demonstrates promise but quickly hits a few limits.
If we fix a transverse operator $\omega\in \Omega$, then the characteristic
polynomials of each $\omega_a$  (in the variable $x_a$) appear in the
annihilators $\Ann_{K[X]}^{\omega}(t)$.  Thus, the solution set of the
polynomials   $\Ann_{K[X]}^{\omega}(t)$   in $K^{\zrange{\vav}}$ is  a finite
set of points, or, more formally, $\Ann_{K[X]}^{\omega_a}(t)$ defines a
$K$-scheme of dimension $0$. Looking closely at Figures~\ref{fig:annihilator}
and \ref{fig:annihilator-2}, we witnessed the presence of polynomials like $xy$
and $xy-y^2$, which hint of richer geometry that might be  of a fundamental
quality of our tensors. To get solution sets corresponding to general varieties,
we need to intersect our single operator annihilators across sets
$\Delta\subseteq \prod_a \End(V_a)$ of transverse operators.  However, even that
could fail to notice phenomena that are exhibited only over extensions of $K$,
particularly when $K$ (and hence $\Omega$) is finite.  For this reason, we are
led to study annihilators of subfunctors $L\mapsto \Delta(L)$ of the functor
$L\mapsto \prod_a\End(L\otimes V_a)$  as $L$ ranges over the commutative
$K$-algebras. These subfunctors will typically be \emph{affine $K$-schemes}.

\subsection{A taste of schemes}

We briefly recall affine schemes, taking the approach of regarding them
as functors from the category of commutative algebras to sets (i.e.\
functors of points). An extensive source following this approach
is  \cite{Waterhouse}; a  full account can be found in 
\citelist{\cite{Hartshorne}}.

Let $\Comm{K}$ denote the category of commutative associative unital
$K$-algebras, and let $\Set$ denote the category of sets. For the purpose of
this work, an affine $K$-scheme (of finite type) is a functor ${\frak
X}:L\mapsto {\frak X}(L)$  from $\Comm{K}$ to $\Set$,   for which  there exist
$n\in\mathbb{N}$ and polynomials $\mathcal{F}\subset  K[x_{1},\dots,x_n]$ such
that ${\frak X}(L)$ is naturally isomorphic to the solution set of the equations
$\{f=0 \mid f\in \mathcal{F}\}$ in $\mathbb{A}^n(L)=L^n$.  (Note our use of $n$
variables here is general and not necessarily the same as $\vav+1$.) This is
equivalent to saying that ${\frak X}(L)$ is naturally isomorphic to the set of
$K$-algebra homomorphisms $K[x_1,\dots,x_n]/({\cal F})\to L$.  We call ${\frak
X}(L)$ the \emph{$L$-points} of ${\frak X}$.  In this setting, a subscheme of
${\frak X}$ is an affine $K$-scheme ${\frak Y}:\Comm{K}\to \Set$ such that for
all $L$, ${\frak Y}(L)\subset {\frak X}(L)$.\footnote{This notion of subscheme
is not standard.  What we have defined here amounts in the literature to saying
that the inclusion morphism ${\frak Y}\to {\frak X}$ is a monomorphism of
schemes.} We then write ${\frak Y}\subset {\frak X}$. We say that ${\frak Y}$ is
a \emph{closed} subscheme of ${\frak X}$ and write ${\frak
Y}\subset_{\mathrm{c}} {\frak X}$ if ${\frak Y}$ is  obtained by adding further
polynomials to the family ${\cal F}$ used  to define ${\frak X}$. This property
is  intrinsic to $\frak X$ and  independent of $n$ and the family $\cal F$
(which are not unique). Morphisms of affine $K$-schemes are natural
transformations. Yoneda's Lemma implies that any morphism $f$ from ${\frak X}$
to an affine $K$-scheme ${\frak Y}$ defined by  ${\cal G}\subset
K[y_1,\dots,y_m]$ is induced by precomposition with a unique $K$-algebra
homomorphism $f^{\#}:K[y_1,\dots,y_m]/({\cal G}) \to K[x_1,\dots,x_n]/({\cal
F})$. If $f^{\#}$ is surjective, then $f:{\frak X}\to {\frak Y}$ is called a
\emph{closed immersion}; this is equivalent to saying that $f_L:{\frak X}(L)\to
{\frak Y}(L)$ is injective  for all $L\in \Comm{K}$ and $\im f:L\mapsto \im f_L$
is a closed subscheme of ${\frak Y}$.

Throughout, $\mathbb{A}^n$ denotes the $n$-dimensional affine space over $K$,
which we can now also interpret as the $K$-scheme defined using the empty set of
polynomials in $K[x_1,\ldots,x_n]$, i.e., $\mathbb{A}^n(L)=L^n$ for all $L\in
\Comm{K}$.

In the sequel, we will often  define functors by specifying them only on objects. In all such
cases, the action on the morphisms will be evident from the context.

The following lemma gives examples of affine $K$-schemes  which
are not defined directly as the null set of some polynomials. It also highlights
why we require $V_0,\ldots,V_{\vav}$ to be projective modules.

\begin{lem}\label{lem:definition-of-calE}
	Let $U$ and $V$ be f.g.\ projective $K$-modules. Then the following functors
	from $\Comm{K}$ to $\Set$ are affine $K$-schemes.
	\begin{enumerate}[(i)]
		\item   $L\mapsto L\otimes V$,
		\item $L\mapsto (L\otimes V)\oslash_L (L\otimes U)=\hom_L(L\otimes U,L\otimes V)$,
		\item $L\mapsto \mathrm{Aut}_L(L\otimes V)$.
	\end{enumerate}
	Furthermore, if $V'$ is a summand of $V$,
	then $L\mapsto L\otimes V'$ defines a \emph{closed} subscheme
	of $L\mapsto L\otimes V$. 
\end{lem}

This is known, but we include a proof that can be made from our present
ingredients.  In particular it demonstrates the lift of the problem to free
modules in a natural way where the precise polynomial ideal defining $\Op{S}{P}$
will eventually emerge independent of choices.

\begin{proof}
	Since $V$ is f.g.\ projective, with rank $n$, there is a split exact
	sequence $0\to V\xrightarrow{g} K^n\xrightarrow{f} K^r$. The map $f$ in the
	sequence can be regarded as $r$ linear polynomials $f_1,\ldots,f_r\in
	K[x_1,\ldots,x_n]$, and $V$ is the solution set of $f_1=\cdots=f_r=0$ in
	$K^n$. Since the  sequence above is split, it remains exact after tensoring
	with arbitrary commutative $K$-algebras $L$. This means that $L\otimes V$ is
	the solution set of $f_1=\ldots=f_r=0$ in $L^n$, so $L\mapsto L\otimes V$ is
	an affine $K$-scheme.

	If $V'$ is a summand of $V$, we can find $f':K^n\to K^{r'}$ such that $0\to
	V'\xrightarrow{g} K^n\xrightarrow{f'} K^{r'}$ (same $g$ as in the previous
	paragraph) is exact. Writing $f'=(f'_1,\ldots,f'_{r'})$, we see that  
	$L\otimes V'$ is the solution set of $f_1=\cdots=f_r=f'_1=\cdots=f'_{r'}=0$ in
	$L^n$, so $L\mapsto  L\otimes V'$ is a closed subscheme of $L\mapsto
	L\otimes V$.

	To prove (ii), observe that the natural map $L\otimes (V\oslash U)\to
	(L\otimes V)\oslash_L (L\otimes U)$ determined by 
	$\gamma\otimes f\mapsto (\gamma'\otimes u\mapsto \gamma\gamma' \otimes f(u))$ is an
	isomorphism because $U$ is f.g.\ projective. Since $V\oslash U$ is a f.g.\
	projective $K$-module, $L\mapsto  (L\otimes V)\oslash_L (L\otimes U)\cong
	L\otimes (V\oslash U)$ is an affine $K$-scheme by (i).

	To prove (iii), choose a f.g.\ projective $K$-module $W$ such that, for some
	$r$, $V\oplus W\cong K^r$. Then $\varphi\mapsto \varphi\oplus 0_W$ embeds
	$\End(V)$ as a summand of $\End(V\oplus W)\cong \mathbb{M}_{r\times r}(K)$.
	By the previous paragraphs, the image of $\End_L(L\otimes V)\cong L\otimes
	\End(V)$ in $\mathbb{M}_{r\times r}(L)\cong L^{r^2}$ is the null set of some
	polynomials $h_1,\dots,h_t\in K[x_{ij}\,|\,i,j\in\{1,\dots,r\}]$. Define
	$\psi_L:\Aut_L(L\otimes V)\to L\times\mathbb{M}_{r\times r}(L)$ by
	$\psi_L(w)=(\det(w\oplus 1_W)^{-1}, w\oplus 1_W)$. Then $\psi_L$ is a
	bijection between its domain and the null set of the polynomials
	$h_1,\dots,h_t$ and $\det((x_{ij})_{i,j}+(0_V\oplus 1_W))y-1\in
	K[x_{11},x_{12},\dots,x_{rr},y]$ in $L\times \mathbb{M}_{r\times r}(L)$. As
	$\psi: L \mapsto \psi_L$ is a natural transformation, $L\mapsto
	\End_L(L\otimes V)$ is an affine $K$-scheme.
\end{proof}

\begin{ex}
When $V$ is a free $K$-module, say $V\cong K^n$, the scheme $L\mapsto L\otimes
V$ is isomorphic to $\mathbb{A}^n$, because $L\otimes V\cong L^n$ naturally.
Similarly, $L\mapsto \End_L(L\otimes V)$ is isomorphic to $\mathbb{A}^{n^2}$. In
this case, we also have a  natural isomorphism $\Aut_L(L\otimes V)\cong
\GL_n(L)$, and the latter is naturally isomorphic to the solution set of
$\det((x_{ij})_{i,j})y-1\in K[x_{11}, x_{12}, \dots, x_{nn},y]$ in
$\mathbb{M}_{n\times n}(L)\times L=L^{n^2+1}$.
\end{ex}

\begin{remark}
	In Lemma~\ref{lem:definition-of-calE}, $L\mapsto \Aut_L(L\otimes V)$ is a
	subscheme of $L\mapsto \End(L\otimes V)$ which is in general not a
	\emph{closed} subscheme. (Rather, it is an \emph{open} subscheme.)
\end{remark}

We finally note that if $\frak X$ and $\frak Y$ are affine $K$-schemes, then so
is their product ${\frak X}\times {\frak Y}:L\mapsto {\frak X}(L)\times {\frak
Y}(L)$---use the polynomials ${\cal F}\cup {\cal G}\subset
K[x_1,\dots,x_n,y_1,\dots,y_m]$, where ${\cal F}\subset K[x_1,\dots,x_n]$ and
${\cal G}\subset K[y_1,\dots,y_m]$ are polynomials defining  ${\frak X}$ and
${\frak Y}$, respectively.

\subsection{Operator families as schemes}

Recall that $\Op{S}{P}$ resides inside $\prod_{a\in \zrange{\vav}}\End(V_a)$,
which was denoted by $\Omega$ until this point. We now make $\Omega$ into an
affine $K$-scheme by defining
\[
	\Omega(L)=\prod_{a\in\zrange{\vav}} \End_L(L\otimes V_a).
\]
This is an affine $K$-scheme by Lemma~\ref{lem:definition-of-calE}, and its
$K$-points are $\Omega(K)=\prod_{a\in \zrange{\vav}}\End(V_a)$. Likewise, we
enrich $\Omega^\times$ into an affine $K$-scheme by setting
\[
	\Omega^\times(L)=\prod_{a\in\zrange{\vav}} \Aut_L(L\otimes V_a).
\]
One similarly regards $\Omega_{A,B}$, for disjoint $A,B\subset \zrange{\vav}$,
as an affine $K$-scheme.

To define $\Op{S}{P}$ as a subscheme $\Omega$, we first recall base-change of
tensors. Let $L\in\Comm{K}$. Then there a unique $L$-module homomorphism
$\bra{\cdot}_L: L\otimes T\to (L\otimes V_0)\oslash_L\dots\oslash_L (L\otimes
V_{\vav})$ where given $\gamma_0,\ldots,\gamma_{\vav}\in L$ and $\bra{t}\in
V_0\oslash\cdots\oslash V_{\vav}$,
\begin{align}\label{eq:base-change-of-tensor}
	\bra{\gamma_0\otimes t}_L ~\gamma_1\otimes v_{1},\ldots,\gamma_{\vav}\otimes v_{\vav}\ra
	& := \gamma_0\cdots\gamma_{\vav}\otimes \bra{ t} v\rangle.
\end{align}
We abbreviate $\bra{t}_L$ to $\bra{t_L}$, or just $\bra{t}$ when $L$ is clear
from the context. Given subsets $S\subset T$ and $P\subset
K[X]$, we now define the functor $\Op{S}{P}:\Comm{K}\to
\Set$ by
\[
	\Op{S}{P}(L)=\{\omega\in\Omega(L)\mid (\forall t\in S)(\forall p\in P)(\bra{t_L}p(\omega)=0)\}.
\]

The next lemma states that when concerned with the operator sets, we get
equivalent schemes whether the $V_a$ are all f.g.\ projective modules or f.g.\
free modules. 

\begin{lem}\label{lem:free-modules}
	Let $(V_0,\dots, V_{\vav})$ be $K$-modules such that $W_a = U_a\oplus V_a$,
	with $W_a$ f.g.\ and free and embedding $\iota_a:V_a\rightarrow W_a$. If
	$\bra{\cdot} : T \rightarrow V_0\oslash \cdots \oslash V_{\vav}$ is a tensor
	space, then for $\lla \cdot | = \iota\bra{\cdot} : T \rightarrow W_0\oslash
	\cdots \oslash W_{\vav}$ and for all $S\subset T$, $P\subset K[X]$, and
	$L\in\Comm{K}$, 
	\begin{align*}
		\Op{S}{P}\left(L\right) = \Op{\lla S|}{P}\left(L\right) \cap \prod_a \End_L(L\otimes V_a).
	\end{align*}
\end{lem}

\begin{proof}
	For all $L\in\Comm{K}$, $\Omega[V](L):=\prod_{a}\End_L(L\otimes V_a)$ embeds
	as a summand of $\Omega[W](L):=\prod_{a}\End_L(L\otimes W_a)$ via
	$(\omega_a)_a\mapsto (\omega_a\oplus 0_{U_a})_{a}$. Thus, we view
	$\Omega[V]$ as a subscheme of $\Omega[W]$. The lemma follows if we show that
	for all $t\in T$, $p\in P$, $L\in\Comm{K}$ and $\omega\in
	\Omega[V](L)\subset \Omega[W](L)$, $\lla t_L| p(\omega)=0$ if, and only if,
	$\bra{t_L}p(\omega)=0$. Writing $p=\sum_{e:\zrange{\vav}\to \mathbb{N}}
	\lambda_eX^e$, for all $v\in\prod_{a=1}^{\vav} V_a$, $u\in
	\prod_{a=1}^{\vav} U_a$, 
	\begin{align*}
		\lla t_L| p(\omega) |u+v\rra &=
			\sum_e \lambda_e \omega_0^{e(0)}\lla t_L| 
			\omega_1^{e(1)}(v_1+u_1),\ldots,
					\omega_{\vav}^{e(\vav)}(v_{\vav}+u_{\vav})\rra\\
		&=
			\sum_e \lambda_e \omega_0^{e(0)}\bra{t_L}  
			\omega_1^{e(1)}( v_1),\ldots,
				\omega_{\vav}^{e(\vav)}( v_{\vav}) \ra=\bra{t_L} p(\omega)\ket{v}. \qedhere
	\end{align*}
\end{proof}

\begin{prop}\label{prop:Z}
	For every subset $S\subset T$ and $P\subset K[X]$, the functor
	$\Op{S}{P}:\Comm{K}\to \Set$ is a closed subscheme of $\Omega$. In addition,
	for all subsets $S'\subset S$, $P'\subset P$, we have 
	$\Op{S}{P}\subset_{\rm c} \Op{S'}{P'}$. 
\end{prop}

\begin{proof}
	We shall make  use of the following fact: If ${\frak X}$ is an affine
	$K$-scheme, ${\frak X'}$ is a subscheme of ${\frak X}$, and ${\frak
	Z}\subset_{\rm c}{\frak X}$, we have ${\frak Z}\cap {\frak X'}\subset_{\rm
	c}{\frak X'}$, where ${\frak Z}\cap {\frak X'}$ is the functor $L\mapsto
	{\frak Z}(L)\cap {\frak X'}(L)$ (the intersection is taken in ${\frak
	X}(L)$); see \cite{Hartshorne}*{Exer.~II.3.11(a)}. Since each $V_a$ is f.g.\
	projective, there exists a $K$-module $U_a$ such that $W_a:=V_a\oplus U_a$
	is finitely generated and free. Consider the frame $(V'_0, \ldots,
	V'_{\vav})$ and its corresponding affine $K$-scheme $\Omega'$. By
	Lemma~\ref{lem:free-modules} and by the fact stated at the beginning of the
	proof, we may replace the tensor space $\bra{\cdot}:T\to V_0\rversor
	\dots\rversor V_{\vav}$ with $\lla \cdot|:T\to W_0\rversor\dots\rversor
	W_{\vav}$, defined in Lemma~\ref{lem:free-modules}, and assume that the
	$V_a$ are f.g.\ and free for the remainder of the proof.

	Now that the $V_a$ are free, say $V_a\cong K^{d_a}$, we have natural
	isomorphisms $L\otimes V_a\cong L^{d_a}$ and $\Omega(L)=\prod
	\End_L(L\otimes V_a)\cong L^{d }$, where $d=\sum_{a\in\zrange{\vav}}d_a^2$.
	Let $p\in P$, $t\in S$,  $\omega\in\Omega(L)$. There are polynomials $\{f_i
	\}_{i\in I_{t,p}}\subset K[x_1,\dots,x_d]$, depending on $t$ and $p$ but not
	the chosen bases of $V_a$, such that $\bra{t_L}p(\omega)=0$ if, and only
	if, for all $i\in I_{t,p}$, $f_i(\omega)=0$; we elaborate on their
	construction in Remark~\ref{rem:make-scheme}. Let ${\cal F}=\bigcup_{t\in
	T}\bigcup_{p\in P}\{f_i\,|\,i\in I_{t,p}\}$. Then $\Op{S}{P}(L)$ is
	naturally isomorphic to the null set of $\cal F$ in $L^d\cong \Omega(L)$;
	hence, $\Op{S}{P}$ is an affine $K$-scheme. Finally, note that the equations
	defining $\Op{S'}{P'}$ contain those defining $\Op{S}{P}$, so
	$\Op{S}{P}\subset_{\rm c}\Op{S'}{P'}$.
\end{proof}

\begin{remark}\label{rem:make-scheme}
	We elaborate on the construction of the polynomials defining $\Op{S}{P}$
	when each $V_a$ is f.g.\ and free, say  $V_a\cong K^{d_a}$. Let
	$\{f_{a,i}\mid i\in \range{d_a}\}$ be a $K$-basis for $V_a$ and let
	$\{f^0_{i}\mid i\in \range{d_a}\}\subset \hom(V_0,K)$ be the dual basis of 
	$\{f_{0,i}\mid i\in \range{d_a}\}$. We
	use these bases to identify $\End(V_a)$ with $\mathbb{M}_{d_a\times
	d_a}(K)$. For $\iota:(a\in \zrange{\vav})\to \range{d_a}$, now each tensor
	$t\in S$ can be prescribed by a hypermatrix \eqref{def:T-1},
	$t_{\iota}=f^0_{\iota(0)}\langle t|f_{1,
	\iota(1)},\dots,f_{\vav,\iota(\vav)}\rangle$. Given $p=\sum_e\lambda_e X^e$,
	$t\in S$, $L\in\Comm{K}$ and $\omega\in \Omega(L)$, a computation shows that
	$\bra{t_L}p(\omega)=0$  if, and only if, $\omega$ satisfies the equations 
	\begin{align} 
		\left(\forall \iota\in \prod_{a\in \zrange{\vav}}\range{d_a}\right) 
		\left(
		0  = \sum_{e,\kappa} \lambda_e t_{\kappa}
		(\Gamma_0^{e(0)})_{\kappa(0)\iota(0)}
			\cdots
				(\Gamma_{\vav}^{e(\vav)})_{\iota(\vav) \kappa(\vav )}
		\right),
	\end{align}
	where  each $\Gamma_{a}$ is a $(d_a\times d_a)$-matrix of indeterminates, and
	$(\cdot)_{ bc}$ means taking the $(b,c)$-entry.  To make this
	into a polynomial formula one still has to expand the matrix powers
	$\Gamma_a^{e_a}$. Such computation is best left to a computer, and we have
	implemented such routines in the computer algebra system \textsf{Magma}
	\cite{TensorSpace}.
\end{remark}

\begin{remark}
	It is at times convenient to restrict the ambient affine $K$-scheme $\Omega$
	to an affine subscheme $\Omega_0$ and consider $\Op{S}{P}\cap \Omega_0$
	instead of $\Op{S}{P}$. For example, if one is interested in invertible
	transverse operators, then take $\Omega_0=\Omega^\times$. Another useful
	example is when the frame is symmetric, i.e.\ $V_{\vav}=\cdots=V_0$, in
	which case we could take $\Omega_0(L)=\{\omega\in \Omega(L)\,:\,
	\omega_{\vav}=\dots=\omega_0\}$. In valence $\vav=2$, the affine $K$-scheme
	$\Op{S}{x_0-x_1x_2}\cap \Omega_0$ would be the automorphism scheme of the
	non-associate algebra $(V_0,*)$ (cf.\ Example~\ref{ex:autotopism}).
	Regardless of which $\Omega_0$ we choose, by the fact stated at the start of
	the proof of Proposition~\ref{prop:Z}, $\Op{S}{P}\cap \Omega_0$ will be a
	closed subscheme of $\Omega_0$.
\end{remark}

\begin{remark}
	For disjoint $A,B\subset \zrange{\vav}$, the proof of
	Proposition~\ref{prop:Z} can be modified to show that
	$\Op[A,B]{S}{P}\subset_{\rm c}\Omega_{A,B}$ whenever $P\subset K[X_{A\sqcup
	B}]$. If $P\nsubseteq K[X_{A\sqcup B}]$, then $\Op[A,B]{S}{P}$ is just a
	functor.
\end{remark}

Now that we have defined $\Op{S}{P}$ as a subscheme of $\Omega$, we turn to
define the ideal $\Id{S}{\Delta}$ and the tensor space $\Ten{P}{\Delta}$ when
$\Delta$ is a \emph{subscheme} of $\Omega$ (rather than a subset of
$\Omega(K)$). The definitions are in the spirit of
Section~\ref{sec:intro}, now incorporating extensions of $K$:
\begin{align*}
	\Id{S}{\Delta} & = \{ p\in K[X] \mid 
			\forall L\in \Comm{K}, 0=\bra{T_L}P(\Delta(L))\}, \\
	\Ten{P}{\Delta} & = \{ t\in T\mid 
			\forall L\in \Comm{K}, 0=\bra{T_L}P(\Delta(L))\}.
\end{align*}
Therefore, the following is a consequence of the above definitions.

\begin{prop}\label{prop:inclusion-reverse}
	The formation of $\Ten{P}{\Delta}$, $\Id{S}{\Delta}$ and $\Op{S}{P}$
	is inclusion-reversing relative to each of the inputs $P$, $\Delta$, $S$.
\end{prop}

\subsection{Proof of Theorem~\ref{thm:correspondence}}

Proposition~\ref{prop:Z} handles the claim that $\Op{S}{P}$ is a closed
subscheme of $\Omega$, and $\Id{S}{\Delta}$ is an ideal and $\Ten{P}{\Delta}$ is
a $K$-submodule of $T$, and Proposition~\ref{prop:inclusion-reverse} handles the
order-reversing inclusions. It remains to show the correspondence.  Fix
$S\subset T$, $P\subset K[X]$,  $\Delta \subset \Omega$, and let $L$ denote an
arbitrary commutative $K$-algebra.  Recall that we need to show that
\begin{align*}
S  \subseteq \Ten{P}{\Delta} \quad\Leftrightarrow\quad
		P  \subseteq \Id{S}{\Delta} \quad \Leftrightarrow\quad
			\Delta  \subseteq \Op{S}{P}.
\end{align*}
Unfolding the definitions, we see that each of these conditions is equivalent to
\begin{align*}
	(\forall t)(\forall p)(\forall L)(\forall\omega) \Big(
	\big((t\in S)\wedge (p\in P)\wedge(\omega\in \Delta(L)\big)\Rightarrow 
	(\bra{t}p(\omega)=0) \Big),
\end{align*}
and so the proof is complete. \qed

\subsection{Frames with non-projective modules.}\label{sec:general-module}
\label{subsec:non-projective}

When $V_0,\dots,V_{\vav}$ are not all f.g.\ projective $K$-modules, it can
happen that the functor $\Omega:L\mapsto  \prod_a\End_L(L\otimes V_a)$ is not an
affine $K$-scheme. In this case,  speaking about closed subschemes of $\Omega$
is meaningless, although one could still consider general subfunctors of
$\Omega$ and Theorem~\ref{thm:correspondence} remains correct if one interprets
$\subset $ as ``being a subfunctor.''  In fact, one can take a step further and
show that when the modules $V_a$ are finitely presented, the functor $\Omega$ is
a (set-valued) \emph{sheaf} on the \emph{large fppf site} of $\mathrm{Spec}\,K$,
and then Theorem~\ref{thm:correspondence} remains correct upon interpreting
$\subset$ and $\subset_{\mathrm{c}}$ as ``being a subsheaf.''

An alternative direction entirely is to work with rings $K$ having nice module
decomposition properties, for example, PIDs, uniserial rings, and general
K\"othe rings.  Here without assumption of projectivity, all f.g. $K$-modules
decompose as $V=Kv_1\oplus \cdots \oplus Kv_r$.  Hence, $\End(V)$ is isomorphic
to \emph{checkered matrices} $[M_{ij}]$ where $M_{ij}\in Kv_j\oslash Kv_i$.
The polynomials of our correspondence can be
recovered in terms of the matrix coordinates. Unfortunately such decompositions
now depend on the selected decompositions which makes the correspondence subject
to these choices, and the scheme theoretic implications need not hold. For some
applications, this undoubtedly should suffice, but that is a topic we do not
explore further here.

\subsection{More about base change}

We close this section by discussing the behavior of the tensor space
$\Ten{P}{\Delta}$ under base change, establishing several results that will be
important later.

While the definitions of $\Id{S}{\Delta}$ and $\Ten{P}{\Delta}$ are natural,
from a computational point of view they may seem discouraging. Even when
equations defining $\Delta$ are given, finding $\Id{S}{\Delta}$ and
$\Ten{P}{\Delta}$ ostensibly requires enumeration on all commutative
$K$-algebras. This hardship is addressed in the following two lemmas, which show
that it is enough to consider particular choices of $L$ and $\omega$. The second
lemma is sufficient for the applications considered in this work. 

\begin{lem}\label{lem:generic-element}
	Let $\Delta$ be a subscheme of $\Omega$, with defining polynomials ${\cal
	F}\subset K[y_1,\dots,y_n]$ inside some $\mathbb{A}^n$, and let
	$K[\Delta]=K[y_1,\dots,y_n]/({\cal F})$. Suppose $t\in T$, $p\in K[X]$, and
	$\hat{\omega}=(\hat{y}_1,\dots,\hat{y}_n)\in \Delta(K[\Delta])$, where $\hat{y_i}$ is
	the image of $y_i$ in $K[\Delta]$. Then for all $L\in \Comm{K}$ and all $\omega\in \Delta(L)$,
	$\bra{t_L}p(\omega)=0$ if, and only if, for all $\hat{\omega}\in \Delta(K[\Delta])$,
	$\bra{t_{K[\Delta]}}p(\hat{\omega})=0$.
\end{lem}

\begin{proof}
	We prove the converse. Let $L\in\Comm{K}$
	and $\omega\in \Delta(L)$. Regarding $\Delta(L)$ as a subset of
	$\mathbb{A}^n(L)$, we may identify  $\omega$ with a point
	$u=(u_1,\dots,u_n)\in L^n$. Since the polynomials $\cal F$ vanish at $u$, we
	have a unique $K$-algebra homomorphism $\phi:K[\Delta]=K[y_1,\dots,y_n]/({\cal
	F})\to L$ satisfying, for all $i$, $\phi(\hat{y_i})=u_i$. In particular,
	$\Delta(\phi)$ takes $\omega_0$ to $\omega$. Now, the base change map
	$(K[\Delta]\otimes V_0)\oslash_{K[\Delta]}\cdots
		\oslash_{K[\Delta]}(K[\Delta]\otimes V_{\vav})\to
			(L\otimes V_0)\oslash_{L}\cdots\oslash_{L}(L\otimes V_{\vav})$ induced by
	$\phi$ is a $K[X]$-module homomorphism (relative to the module structures
	induced by $\omega_0$ and $\omega$) taking $\bra{t_{K[\Delta]}} $ to $\bra{t_L} $.
	Since $\bra{t_{K[\Delta]}}p(\hat{\omega})=0$, this means that $\bra{t_L}p(\omega)=0$.
\end{proof}

\begin{lem}
	Let $p\in K[X]$ be a linear homogeneous polynomial, and let $U$ be a
	$K$-module summand of $\Omega(K)$ generated by $\{u_j\}_{j\in J}$. Let
	$\Delta=\Delta_U$ be the subscheme of $\Omega$ determined by
	$\Delta(L)=L\otimes U$. Then for all $L\in\Comm{K}$	and all $\omega\in
	\Delta(L)$ $\bra{t_L}p(\omega)=0$ if, and only if, for all $j\in J$,
	$\bra{t}p(u_j)=0$. When $K$ is a field and $P\subset K[X]$ is an ideal
	generated by linear homogeneous polynomials, then for all $S\subset T$ the
	scheme $\Op{S}{P}$ is of the form $\Delta_U$ for some $K$-module $U\leq
	\Omega(K)$.
\end{lem}

\begin{proof}
	We will prove the reverse direction. By assumption,  $\omega\in \Delta(L)$
	can be written as $\sum_j \gamma_j \otimes u_j$, for some
	$\{\gamma_j\}_{j\in J}$ in $L$,  all but finitely many being  $0$. Write
	$p=\sum_{a\in \zrange{\vav}} \lambda_a x_a$. Then for all $v\in
	\prod_{a=1}^{\vav}V_a$, we have that
	\begin{align*}
		\bra{t_L} p (\omega)\ket{1\otimes v}&=
		\sum_j \left( \lambda_0 (\gamma_j \otimes u_{j,0}) 
			\bra{t_L} 1\otimes v_{1},\ldots,1\otimes v_{\vav}\ra +
		\sum_{a=1}^{\vav} \lambda_a 
			\bra{t_L} (\gamma_j\otimes u_{j,a}) (1\otimes v_a),1\otimes v_{\bar{a}}\ra\right)\\
		&=
		\sum_j \gamma_j\otimes \left(\lambda_0 u_{j,0} 
			\bra{t} v\ra + \sum_{a=1}^{\vav} \lambda_a \bra{t} u_{j,a}v_a,v_{\bar{a}}\ra\right)=
		\sum_j \gamma_j \otimes \bra{t }  p(  u_j) \ket{  v}=
		0,
	\end{align*}
	so $\bra{t_L}p(\omega)=0$.

	The last assertion follows from Remark~\ref{rem:make-scheme} by noting that
	the polynomials defining $\Op{S}{P}$ in $\Omega\cong \mathbb{A}^d$ are
	linear homogeneous; the $K$-vector $U$ spaces is the intersection of their
	kernels.
\end{proof}

Given an affine $K$-scheme $\mathfrak X$ and $L\in\Comm{K}$, let $\mathfrak X_L$
denote the restriction of $\mathfrak X$ to $\Comm{L}$. Then $\mathfrak X_L$ is
an affine $L$-scheme---it is defined by the polynomials defining $\frak X$,
regarded as polynomials over $L$. Given $P\subset K[X]$ and $\Delta\subset
\Omega$, the next proposition shows that, under mild assumptions,
$\Ten{P}{\Delta(L)}$ (an $L$-submodule of $L\otimes T$) and $\Ten{P}{\Delta}$ (a
$K$-submodule of $T$) are the same after extending by $L$.

\begin{prop}\label{prop:flatness}
	Let $\Delta\subset_{\rm c} \Omega$, and let be $P$ a finitely generated
	ideal of $K[X]$. If $L\in \Comm{K}$ is flat over $K$, then
	$\Ten{P}{\Delta(L)} = L\otimes \Ten{P}{\Delta}$.
\end{prop}

The assumptions on $P$ and $L$ are always satisfied when $K$ is a field.

\begin{proof}
	Fix a polynomials $p_1,\dots,p_n$ generating $P$, and let $K[\Delta]$ and
	$\hat{\omega}$ be as in  Lemma~\ref{lem:generic-element}. By that lemma, $t\in
	\Ten{P}{\Delta}$ if, and only if, for all $i$, $\bra{t_{K[\Delta]}}p_i(\hat{\omega})=0$.
	 Equivalently, $\Ten{P}{\Delta}$ is the kernel of the map
	$\Phi:T\to [(K[\Delta]\otimes V_0)\oslash_{K[\Delta]} \cdots \oslash_{K[\Delta]} (K[\Delta]\otimes
	V_{\vav})]^{n}\cong [K[\Delta]\otimes (V_0\oslash\cdots\oslash V_{\vav})]^{n}$
	(see Lemma~\ref{lem:definition-of-calE} for the last isomorphism) given by
	$\Phi(t)=(\bra{t_{K[\Delta]}}p_i(\hat{\omega})\mid i\in \range{n} )$. Likewise, 
	$\Ten{P}{\Delta(L)}$ is the kernel of a similarly defined 
	$\Phi_L: L\otimes T\to [(L\otimes K[\Delta])\otimes (V_0\oslash \cdots\oslash V_{\vav})]^{n}$.
	Note that  $\Phi_L=\Phi\otimes 1_L$, and since   $L$ is
	flat over $K$, the natural map $\ker(\Phi)\otimes L\to \ker(\Phi_L)$ is an
	isomorphism. 
\end{proof}

\begin{coro}\label{cor:flatness}
	Let $S\subset T$ be a $K$-submodule, and let $P$ be a finitely generated
	ideal of $K[X]$. If $L\in \Comm{K}$ is flat over $K$, then
	$\Ten{P}{\Op{L\otimes S}{P}}=L\otimes \Ten{P}{\Op{S}{P}}$.
\end{coro}

\begin{proof}
	As $\Op{L\otimes S}{P}=\Op{S}{P}_L$, this is a special case of
	Proposition~\ref{prop:flatness}.
\end{proof}
%==============================================================================

%==============================================================================
%\input{linear.tex}

\section{Linear traits and Theorems~\ref{mainthm:Densor} and~\ref{mainthm:Lie-Asc}}\label{sec:universality}

As before, $\bra{\cdot}:T\hookrightarrow V_{0}\rversor \dots\rversor V_{\vav}$
denotes  a tensor space over a commutative ring $K$ where the frame
$(V_{0},\dots,V_{\vav})$ consists of f.g.\ projective $K$-modules.

For $S\subset T$ and a linear homogeneous ideal $P\subset K[X]$, we study
$\Op{S}{P}$ and its $P$-closure $\Ten{P}{\Op{S}{P}}$. In this case, for all $L\in \Comm{K}$,
$\Op{S}{P}(L)$ is a submodule of $\Omega(L)$. We prove
Theorem~\ref{mainthm:Densor}, showing that the densor of a tensor subspace $S$
is contained in all closures $\Ten{P}{\Op{S}{P}}$ when $P$ has full support. In
the second half, we study when $\Op{S}{P}$ is closed under natural associative
and non-associative products on $\Omega$, proving Theorem~\ref{mainthm:Lie-Asc}
and deriving consequences in valence $\vav=2$.

\subsection{A torus action}
\label{subsec:torus}

Let $\mathbb{T}$ denote the affine $K$-scheme $L\mapsto
(L^\times)^{\zrange{\vav}}$. Given $L\in \Comm{K}$, the $L$-points $\tau\in
\mathbb{T}(L)$ form a group that acts on $\Omega(L)$ by $\tau\omega =
(\tau_0\omega_0, \ldots, \tau_{\vav}\omega_{\vav})$. Further, the group of
$K$-points $\mathbb{T}(K)=(K^\times)^{\zrange{\vav}}$ acts on $K[X]$ via 
\begin{align}\label{def:torus-act}
	q^{\tau} (x_0,\ldots, x_{\vav})=q(\tau_0^{-1} x_0,\ldots,\tau_{\vav}^{-1} x_{\vav}).
\end{align}
If $L\in\Comm{K}$, then we have a similar action of $\mathbb{T}(L)$ on $L\otimes
K[X]\cong L[X]$.

If $\Delta$ is a subscheme of $\Omega$ and $\tau\in \mathbb{T}(K)$, then we
write $\tau\Delta$ for the subfunctor of $\Omega$ given by
$(\tau\Delta)(L)=\{\tau\omega\,|\,\omega\in \Delta(L)\}$. If $\Delta$ is a
closed subscheme of $\Omega$, then so is $\tau \Delta$. Indeed, multiplication
by $\tau$ defines an $K$-scheme isomorphism $\tau:\Omega \to \Omega$, and
therefore the composition $\Delta\to \Omega \xrightarrow{\tau}\Omega$ is also an
immersion with image $\tau\Delta$. The following proposition summarizes how the
$\mathbb{T}(K)$-action interacts with the formation of $\Op{-}{-}$, $\Id{-}{-}$
and $\Ten{-}{-}$.

\begin{prop}\label{prop:torus-redundancy}
	Let  $P\subset  K[X]$, $S\subset T$, $\Delta\subset   \Omega$ and
	$\tau\in \mathbb{T}(K)$. Then:
	\begin{enumerate}[(a)]
		\item $\Op{S}{P^\tau}=\tau\Op{S}{P}$.
		\item $\Ten{P^\tau}{\Delta}=\Ten{P}{\tau^{-1}\Delta)}$.
		\item $\Id{S}{\tau^{-1}\Delta}=\Id{S}{\Delta}^\tau$.
	\end{enumerate}
	Thus, $\Ten{P^{\tau}}{\Op{S}{P^\tau}}=\Ten{P}{\Op{S}{P}}$.
\end{prop}

\begin{proof}
	For all $t\in S$, $p\in P$, $L\in\Comm{K}$ and $\omega\in \Omega(L)$, we
	have $\bra{t_L}p^\tau(\omega)=\bra{t_L}p(\tau^{-1}\omega)$. The proposition
	follows from this observation.
\end{proof}

\subsection{Optimality of derivations and densors; Theorem~\ref{mainthm:Densor}}

Let $P\subset K[X]$ be a linear homogenous ideal. Recall the support of an ideal
$P\subset K[X]$; we say that $P$ has \emph{full support} if $\supp
P=\zrange{\vav}$. There is a subtlety we need to bring to light for
Theorem~\ref{mainthm:Densor} concerning linear homogeneous polynomial ideals
with full support and such ideals \emph{generated} by linear homogeneous
polynomials with full support. We illustrate this with some examples. 

\begin{ex}\label{ex:lin-hom-id}
\begin{enumerate}[(i)]
	\item Assume $K=\mathbb{F}_2$ and $\vav=2$. Then $P=(x_0-x_1,x_1-x_2)$ is a
	linear homogeneous ideal of full support, containing no linear homogeneous
	polynomials of full support.
	
	\item Assume $K=\mathbb{F}_3$ and $\vav=2$. Then $P=(x_0+x_1+x_2, x_1-x_2)$
	has full support and contains linear homogeneous polynomials of full
	support, but it cannot be \emph{generated} exclusively by linear homogeneous
	polynomials of full support.
	
	\item Assume $K=\mathbb{Z}$ and $\vav=2$. Then $(2x_0-x_1,2x_0-x_2)$ does
	not have full support while $(2x_0-x_1,5x_0-x_2)$ does have full support.
\end{enumerate}
\end{ex}

\addtocounter{submainthm}{2}
\begin{submainthm}
	\label{thm:densor-stronger}
	Let $P\subset K[X]$ be a f.g.\ linear homogeneous ideal and let $S\subset
	T$.
	\begin{enumerate}[(i)]
        \item If $P$ has full support, then $\Den{S}\subset
        \Ten{P}{\Op{S}{P}}$.

		\item If $S$ is fully nondegenerate and there is a subset $A\subset
		\zrange{\vav}$ such that $P$ is generated by linear polynomials in
		$K[X_A]$ and $P+(x_a\,|\,a\in\zrange{\vav}-A)$ has full support, then
		$\Den{S} \subseteq \Ten{P}{\Op{S}{P}}$.
	
		\item If $K$ is a field and is $S$ fully nondegenerate, then $\Den{S}
		\subseteq \Ten{P}{\Op{S}{P}}$.
	\end{enumerate} 
\end{submainthm}

We shall need several lemmas for the proof.

\begin{lem}\label{lem:lin-hom-poly-and-densor}
	If $p \in K[X]$ is linear homogeneous with full support, 
	then $\Ten{p}{\Op{S}{p}}=\Den{S}$.
\end{lem}

\begin{proof}
	Recall that $\Den{S}=\Ten{d}{\Op{S}{d}}$, where $d=x_0-x_1-\cdots
	-x_{\vav}$. Since $p$ has full support, we can find $\tau\in \mathbb{T}(K)$
	such that $p=d^{\tau}$. The lemma now follows from
	Proposition~\ref{prop:torus-redundancy}.
\end{proof}

We have seen in Example~\ref{ex:lin-hom-id} that there are homogeneous linear
ideals with full support that cannot be generated by homogeneous linear
polynomials with full support. The purpose of the following lemmas is to show
that we can remedy this situation by extending the base ring. Recall that a
commutative ring $K$ is \emph{semilocal} if it has finitely many maximal ideals,
e.g.\ when $K$ is a product fields or finite.

\begin{lemma}\label{lem:matrix}
	Assume $K$ is a commutative semilocal ring such that every quotient of $K$
	by a maximal ideal has more than $\vav+1$ elements. Then every f.g.\
	linear homogeneous ideal of full support in $K[X]$ is generated by linear
	homogeneous polynomials of full support.	
\end{lemma}

\begin{proof}
	The proof reduces to the case when $K$ is a field, so we first assume $K$ is
	a field. Fix linear homogeneous polynomials $p_1,\dots,p_n$ generating $P$,
	and let $V$ be the vector space spanned by the $p_i$. For each
	$a\in\zrange{\vav}$, write $U_a=\sum_{b\in \zrange{\vav}-\{ a\}}Kx_b$. Since
	$P$ has full support, $V$ is not contained in any of the $U_a$. We show that
	$V$ has a basis $q_1,\dots,q_m$ consisting of elements in $V-\bigcup_a U_a$
	by constructing it inductively and choosing $q_{j+1}$ to be a vector in $V-
	\mathrm{span}\{q_1,\dots,q_j\}-\bigcup_a (U_a\cap V)$, which is possible
	because $|K|>\vav +1$. So $V$ cannot be the union of $\vav+2$ proper
	subspaces; see~\cite{Clark}.
	
	Now for the general case, let $p_1,\dots,p_n$ be linear homogeneous
	polynomials generating $P$. Denoting by $J$ the Jacobson radical of $K$, set
	$\hat{K}=K/J$, and write $\hat{p}$ for the image of $p\in K[X]$ in
	$\hat{K}[X]$. Since $\hat{K}$ is a finite product of fields, applying  the
	above on each factor separately, there are polynomials $q_1,\ldots,q_m\in
	K[x]$ such that $\hat{q}_1,\ldots,\hat{q}_m\in \hat{K}[X]$ are linear
	homogeneous polynomials of full support generating $\hat{P}$. This implies that there exist corresponding linear homogeneous $q_j$ with full support.
	By
	construction, for each $i\in\{1,\dots,n\}$, there are
	$\alpha_{i1},\dots,\alpha_{im}\in K$ such that $p_i-\sum_{j=1}^m\alpha_{ij}
	q_j\in J[X]$. Set $f_i=q_1+p_i-\sum_{j=1}^m\alpha_{ij} q_j\in q_1+J[X]$.
	Then each $f_i$ has full support and $\{f_1,\dots,f_n,q_1,\dots,q_m\}$
	generate $P$. 
\end{proof}

\begin{lem}\label{lem:quotient-blowing-poly}
	For all $n\in\mathbb{N}$, there exists a monic polynomial $f\in
	\mathbb{Z}[x]$ such that for every commutative ring $K$  and every maximal
	ideal $M$ of $L=K[x\,|\,f(x)=0]$,  $|L/M|>n$.
\end{lem}

\begin{proof}
	Let $R$ denote the set of integer primes in $\{2,\dots,n\}$ and let $e$ be a natural
	number such that $(\forall q\in R)(q^e> n)$. For every $q\in R$, choose an
	irreducible monic polynomial $f_q\in \mathbb{F}_q[x]$ of degree $e$. By the
	Chinese Remainder Theorem, there exists a monic polynomial $f\in
	\mathbb{Z}[x]$ of degree $e$ such that for each $q\in R$, $f\bmod
	q\mathbb{Z}=f_q$. We claim that $f$ is the required polynomial. Indeed, let
	$M$ be a maximal ideal of  $L=K[x\,|\,f(x)=0]$ and let $Q=K\cap M$. Then $Q$
	is necessarily a maximal ideal of $K$ (because $M+Q'L$ is a proper ideal of
	$L$ for every proper ideal $Q'\subset K$ containing $Q$), and $L/M$ contains
	a copy of the  field $K/Q$. If $|K/Q|> n$, then we have $|L/M|> n$.
	Otherwise, $K/Q\cong \mathbb{F}_q$ for some $q\in R$, and by construction,
	$L/M$ contains a root of $f_q$. This means that $L/M$ contains a copy of
	$\mathbb{F}_{q^e}$, which has more than $n$ elements.
\end{proof}

\begin{lem}\label{lem:generation-by-full-sup-polys}
	Let $P\subset K[X]$ be a f.g.\ linear homogeneous ideal of full support.
	Then there exists a faithfully flat commutative $K$-algebra $L$ such that
	$PL[X]$ is generated by linear homogeneous polynomials of full support. If
	$K$ is a semilocal, then $L$ can be chosen so that its underlying $K$-module
	is f.g.\ and free.
\end{lem}

\begin{proof}
	Let $f$ be the polynomial given by Lemma~\ref{lem:quotient-blowing-poly} for
	$n=\vav+1$.  Put $L= K[x\,|,f(x)=0]$. Since $L$ is f.g.\ and free as a
	$K$-module, it is faithfully flat over $K$.  If $K$ is semilocal, then so is
	$L$ \cite{LamFC}*{Prop.~20.6}. Replace $K$ with $L$ so that every $K/M$,
	$M$ a maximal ideal, has $|K/M|>\vav+1$.

	In general, we proceed as follows. For every maximal ideal $M$ of $K$, the
	localization $K_M$ is a local ring with residue field containing more that
	$\vav+1$ elements. Thus, by Lemma~\ref{lem:matrix}, there are finitely many
	linear homogeneous polynomials of full support $p_{1,M},\dots,p_{n(M),M}\in
	K_M[X]$ generating $PK_M[X]$. Thus, there exists $u=u(M)\in K-M$ and linear
	homogeneous polynomials of full support $p'_{1,M},\dots,p'_{n(M),M}\in
	K_u[X]$ which generate $PK_u[X]$ (here $K_u$ is the localization of $K$
	at the multiplicative set $\{1,u,u^2,\dots\}$). The ideal generated by all
	the $u(M)$ in $K$ is not contained in any maximal ideal and therefore equals
	$K$. Thus, there are maximal ideals $M_1,\dots,M_r\subset K$ such that
	$u(M_1),\dots,u(M_r)$ generate the unit ideal in $K$. Write $u_i=u(M_i)$,
	$L=\prod_{i=1}^r K_{u_i}$ and $n=\max\{n(M_1),\dots,n(M_r)\}$. By
	duplicating some of the $p'_{j,M_i}$ if necessary, we may assume that
	$n=n(M_1)=\dots=n(M_r)$ and for each $j\in \range{n}$, define $p_j=
	(p'_{j,M_1},\dots,p'_{j,M_r})\in \prod_{i=1}^r (K_{x_i}[X])\cong L[X]$. Then
	$L$ is a faithfully flat extension of $K$, and $p_1,\dots,p_n$ are linear
	homogeneous polynomials of full support generating $PL[X]$.
\end{proof}

\begin{proof}[Proof of Theorem~\ref{thm:densor-stronger}]
	(i) Assume first that $P$ is generated
	by linear homogeneous polynomials of full support $\{p_i\mod i\in I\}$.
	Then, thanks to Lemma~\ref{lem:lin-hom-poly-and-densor},
	\begin{align*}
		\Ten{P}{\Op{S}{P}}&=\bigcap_{i\in I}\Ten{p_i}{\bigcap_{j\in I}\Op{S}{p_j}}
			\supset  \bigcap_{i\in I}\Ten{p_i}{ \Op{S}{p_i}}=\Den{S}.
	\end{align*}

	Next, assume that $P$ is a f.g.\ homogeneous linear ideal with full support.
	By Lemma~\ref{lem:generation-by-full-sup-polys}, there is a faithfully flat
	commutative $K$-algebra $L$ such that $PL[X]$ is generated by linear
	homogeneous polynomials of full support.  By the previous paragraph and
	Corollary~\ref{cor:flatness}, we have $L\otimes \Ten{P}{\Op{S}{P}}\supset
	L\otimes \Den{S}$ as $L$-submodules of $L\otimes T$. This means that the
	inclusion map $\iota:\Ten{P}{\Op{S}{P}}\to \Ten{P}{\Op{S}{P}}+\Den{S}$
	becomes an isomorphism after tensoring with $L$. As $L$ is faithfully flat
	over $K$, this means that  $\iota$ is an isomorphism, and so $\Den{S}\subset
	\Ten{P}{\Op{S}{P}}$.
	
	(ii) Assume $S$ is fully nondegenerate and let $A$ be a subset of
	$\zrange{\vav}$ such that $P$ is generated by polynomials in $ K[X_A]$ and
	$Q:=P+(x_a\,|\,a\in\zrange{\vav}-A)$ has full support. Write $P_0=P\cap
	K[X_A]$ and $B=\zrange{\vav}-A$. Then $Q$ is generated by
	$P_0\cup\{x_b\,|\,b\in B\}$. Observe that from Section~\ref{sec:terminology},
	$\Op{S}{P_0}=\Op[A]{S}{P_0}\times \Omega_B$ and $\Op{S}{\{x_b\,|\,b\in B\}}=
	\Omega_A\times \Op[B]{S}{\{x_b\,|\,b\in B\}} $. Since $S$ is fully
	nondegenerate, $\Op{S}{\{x_b\,|\,b\in B\}}= \Omega_A\times 0_B$
	(Example~\ref{ex:degeneration}). Thus,
	\begin{align*}
		\Ten{Q}{\Op{S}{Q}}
			&=\Ten{Q}{\Op{S}{P_0\cup \{x_b\,|\,b\in B\}}}\\
			&=\Ten{Q}{\Op{S}{P_0}\cap \Op{S}{ \{x_b\,|\,b\in B\}}}\\
			&=\Ten{P_0\cup \{x_b\,|\,b\in B\}}{\Op[A]{S}{P_0}\times 0_B}\\
			&=\Ten{P_0}{\Op[A]{S}{P_0}\times 0_B}\cap \Ten{\{x_b\,|\,b\in B\}}{\Op[A]{S}{P_0}\times 0_B}\\
			&=\Ten{P_0}{\Op{S}{P_0}}\cap \Ten{\{x_b\,|\,b\in B\}}{0_{\zrange{\vav}}}\\
			&=\Ten{P_0}{\Op{S}{P_0}}\cap T=\Ten{P}{\Op{S}{P}}.
	\end{align*}
	Now, by (i), $\Den{S}\subset \Ten{Q}{\Op{S}{Q}} = \Ten{P}{\Op{S}{P}}$.
	
	(iii) Let $\{p_i\mid i\in I\}$ be linear homogeneous polynomials generating
	$P$, and let $A$ be the set of $a\in \zrange{\vav}$ such that $x_a$ occurs
	(with nonzero coefficient) in one of the $p_i$. Then $\{p_i \mid i\in I\}
	\subset K[X_A]$, and $P+(x_a\,|\,a\in \zrange{\vav}-A)$ has full support.
	Here we need $K$ to be a field. By (ii), we conclude that $\Den{S}\subset
	\Ten{P}{\Op{S}{P}}$.
\end{proof}

We show that when $P\subset K[X]$ is a linear homogeneous ideal of full support,
the dimension (when defined) or the cardinality of the $K$-module
$\OpSet{S}{P}{K}$ cannot exceed that of $\Der(S)$. In fact, under mild
assumptions, a $\mathbb{T}(K)$-orbit of $\Der(S)$ contains $\OpSet{S}{P}{K}$. In
this sense, $\Der(S)$ is the largest of all the $\OpSet{S}{P}{K}$ as $P$ ranges
over the linear homogeneous ideals of full support.

\begin{lem}\label{lem:flat-change-of-Z}
	Let $P\subset K[X]$ be a f.g.\ linear homogeneous ideal, and let $S\subseteq
	T$ be a f.g.\ $K$-submodule. If $L\in\Comm{K}$is flat, then the natural map
	$L\otimes (\prod_a \End(V_a))\to L\otimes \prod_a \End_L(V_a)$ restricts to
	an isomorphism $L\otimes \OpSet{S}{P}{K}\cong \Op{S}{P}(L)=\Op{L\otimes
	S}{P}(K)$.
\end{lem}

\begin{proof}
	Let $p_1,\dots,p_r$ be linear homogeneous polynomials generating $P$, and
	let $t_1,\dots,t_\ell$ be generators of $S$. For every $i\in\range{r}$
	and $j\in \range{\ell}$, let $\phi  :\prod_a \End(V_a)\to
	\prod_{i,j}(V_0\rversor  V_1\rversor \cdots \rversor V_{\vav})$ be given by
	$\phi(\omega)=[\bra{t_j}p_i(\omega)]_{i,j}$. So $\phi$ is $K$-linear, and by
	definition, $\OpSet{S}{P}{K}$ is its  kernel. One similarly defines $\phi_L:
	\prod_a \End_L(L\otimes V_a)\to \prod_{i,j}(L\otimes V_0)\rversor_L \cdots
	\rversor_L (L\otimes V_{\vav})$, so that $\Op{S}{P}(L)$ is $\ker \phi_L$.
	Since $V_0,\dots,V_{\vav}$ are f.g.\ projective,  the natural maps $L\otimes
	(\prod_a \End(V_a))\to \prod_a \End_L(L\otimes V_a)$ and $L\otimes
	(V_0\rversor \ldots\rversor V_{\vav}) \to (L\otimes V_0)\rversor_L  \ldots
	\rversor_L (L\otimes V_{\vav})$ are isomorphisms, and under these
	isomorphisms, $\phi_L$ corresponds to $\phi\otimes \mathrm{id}_L$. Since $L$
	is flat, it follows that the isomorphism $L\otimes (\prod_a \End(V_a))\to
	\prod_a \End_L(L\otimes V_a)$ restricts to an isomorphism $L\otimes\ker
	\phi\to \ker \phi_L$, hence the lemma follows.
\end{proof}

\begin{prop}
	Let $P\subset K[X]$ be a linear homogeneous ideal of full support,
	and let $S\subset T$.
	\begin{enumerate}[(i)]
		\item If $K$ is finite, then $|\OpSet{S}{P}{K}|\leq |\Der(S)|$.
		\item If $K$ is a field, then $\dim_K \OpSet{S}{P}{K}\leq \dim_K
		\Der(S)$.
		\item If $K$ is semilocal such that every quotient of $K$ by a maximal
		ideal has more than $\vav+1$ elements, then there exists $\tau\in
		\mathbb{T}(K)$ such that $\OpSet{S}{P}{K}\subseteq \tau \Der(S)$.
	\end{enumerate}
\end{prop} 

\begin{proof}
	Without loss of generality, we replace $S$ with the $K$-submodule it
	generates. All three statements follow from similar reasoning.

	The assumptions imply that $K$ is Noetherian. Hence, $P$ is f.g.\ as an
	ideal, and $S$ is f.g.\ as a module. In cases (i) and (ii), use
	Lemma~\ref{lem:generation-by-full-sup-polys}, to show that there exists
	$L\in\Comm{K}$ which is nonzero f.g.\ and free as a $K$-module such that
	$PL[X]$ contains a linear homogeneous polynomial $p$ of full support. For
	(iii), set $L=K$ and use Lemma~\ref{lem:matrix}, so that $P$ contains a
	linear homogeneous polynomial of full support. For $d=x_0-x_1-\cdots
	-x_{\vav}$, there exists $\tau\in \mathbb{T}(L)$ such that $p=d^\tau$. By
	Proposition~\ref{prop:torus-redundancy} and
	Lemma~\ref{lem:flat-change-of-Z}, 
	\begin{align*}
		\OpSet{S}{P}{L}\subset \OpSet{S}{p}{L} = \OpSet{S}{d^\tau}{L} 
		= \tau^{-1} \OpSet{S}{d}{L} = \tau^{-1}\Der(L\otimes S).
	\end{align*}
	By Lemma~\ref{lem:flat-change-of-Z}, $L\otimes \OpSet{S}{P}{K}\subset
	L\otimes \Der(S)$ inside $\Omega(L)\cong L\otimes \Omega(K)$. Since $L$ is
	f.g.\ and free over $K$, statements (i) and (ii) hold.
\end{proof}

\subsection{Closure under associative and Lie products; Theorem~\ref{mainthm:Lie-Asc}}

Now we study which of the many distributive products described
in~\eqref{eq:bullet-product} are best suited for our generalized tensor spaces,
and we prove Theorem~\ref{mainthm:Lie-Asc}. Note that for an ideal $P\subset
K[X]$, with $A=\supp P$, $\Op{T}{P}=\Op[A]{T}{P}\times \Omega_{\bar{A}}$. We can
impose an arbitrary product on $\Omega_{\bar{A}}$, so the restriction to
$\Op[A]{T}{P}$ is merited. We begin with a lemma.

\begin{lem}\label{lem:diagonal-embedding} 
    Let $P\subset K[X]$. If $\mathbf{Z}(P)$ is the $K$-scheme of solutions to
    the equations $\{p=0\mid p\in P\}$ in $\mathbb{A}_K^{\zrange{\vav}}$, then for
    all $S\subset T$, the morphism $\mathbb{A}^{\zrange{\vav}}_K\to \Omega$
    given by $(\xi_a)_{a\in \zrange{\vav}}\mapsto (\xi_a
    1_{V_a})_{a\in\zrange{\vav}}$ restricts to a morphism $\mathbf{Z}(P)\to
    \Op{S}{P}$.
\end{lem}

\begin{proof}
	Let $L\in \Comm{K}$, $t\in S$, $\xi\in \mathbf{Z}(P)(L)$, and write $\omega=
	(\xi_a 1_{V_a})_{a\in\zrange{\vav}}\in \Omega(L)$. We need to check that
	$(\forall p\in P)(\bra{t_L}p(\omega)=0)$. Writing $p=\sum_e\lambda_e
	X^e$, for all $v\in \prod_{a=1}^{\vav} L\otimes V_a$, 
	\begin{align*}
		\bra{t_L} p(\omega) \ket{v} &= \sum_e \lambda_e \xi_0^{e(0)}\cdots \xi_{\vav}^{e(\vav)} \bra{t_L} v\ra =p(\xi)\bra{t_L} v\ra =0. \qedhere
	\end{align*}
\end{proof}

\begin{proof}[Proof of Theorem~\ref{mainthm:Lie-Asc}] 
	Recall $P=(p_1,\ldots,p_r)$ with $p_i$ homogeneous linear polynomials.
	If we establish the theorem over some flat extension $L/K$, then it also
	holds for $K$. By passing to a flat extension $L/K$ if necessary, we may
	assume that for every maximal ideal $M$ of $L$, $L/M$ has more than $\vav+1$
	elements. Set $A=\supp P$; by arguing as in the proof of
	Lemma~\ref{lem:matrix}, we see that there exists $p=\sum_{a\in A} {\alpha_a
	x_a}\in P$ such that $(\alpha_a)_{a\in A}\in \mathbb{T}^A(L)$.

	Define the following symmetric bilinear form on $L\otimes Z^A(P)$ via 
	\[
		(\sigma|\tau):=\sum_{a\in A}\alpha_a(\lambda_a+\rho_a )\sigma_a\tau_a.
	\]
	Let $L\otimes Z^A(P)=\{v\in L^{A}\mid p_1(v)=\cdots = p_r(v)=0\}$. By
	Lemma~\ref{lem:diagonal-embedding}, $L\otimes Z(P)$ embeds in $L\otimes
	Z^A(S,P)$ via $\sigma \mapsto (\sigma_a 1_{V_a})_{a\in \zrange{\vav}}$.
	Thus, for every $\sigma,\tau\in L\otimes Z^A(P)$ and every $t\in S$, we have
	\begin{align*}
			&( \forall v)&
			0 & = \bra{t_L}p (\sigma 1\bullet \tau 1)\ket{v}
				=\sum_{a\in A} \alpha_a(\lambda_a\sigma_a\tau_a+\rho_a\tau_a\sigma_a)\la t_L | v\ra
				=(\sigma|\tau)\la t_L|v\ra
				=\la t| (\sigma|\tau) v_b,v_{\bar{b}}\ra.
	\end{align*}
	As $L$ is flat, $\bra{t_L}$ remains either full or nondegenerate in some
	coordinate. Hence, for all $\sigma,\tau\in L\otimes Z^A(P)$,
	$(\sigma|\tau)=0$; that is, $L\otimes Z^A(P)$ is totally isotropic.
	
	Fix a maximal ideal $M$ of $K$ and $x\mapsto \hat{x}$ the projection $K\to
	K/M=:k$. By Witt's invariant on $(|)_k:k^A\times k^A\bmto k$, totally
	isotropic spaces have dimension at most $f+m$ where $f=f(M)$ is the
	dimension of the radical of $(|)_k$ and $m=m(M)$ satisfies $2m \leq
	|A|-f\leq 2m+2$. Recall by our assumption that $P=(p_1,\dots, p_r)$ with each $p_i$ 
	homogeneous linear. So $L/M\otimes Z^A(P)$ is a linear space of
	dimension $|A|-r$ at least.  In particular 
	\begin{align}\label{eqn:Witt}
		|A|-r & \leq \dim_{L/M} L/M\otimes Z^A(P)\leq f+\frac{|A|-f}{2}.
	\end{align}
	Now since $\alpha\in\mathbb{T}^A(L)$, the radical of $(|)_k$ is indexed by
	$R_M:=\{a\in A\mid \lambda_a+\rho_a\in M\}$; so $f(M)=|R_M|$. If $K$ is a
	field, then $M=0$ and $R_M:=\{a\in A\mid \lambda_a+\rho_a=0\}$, which
	by~\eqref{eqn:Witt} implies $|A|-2r\leq |\{a\in A\mid
	\lambda_a+\rho_a=0\}|$.  For general $K$, consider $U_A:=\{a\in
	A~|~\bullet_{(\lambda_a,\rho_a)} \text{ is unital}\}$. This means for $a\in
	U_A$ there is a $\tau_a\in L$ such that $(\lambda_a+\rho_a)\tau_a=1$, and so
	for every maximal ideal $M$, $\lambda_a+\rho_a\not\in M$. Hence
	from~\eqref{eqn:Witt}, $|U_A|\leq |A|-f(M)\leq 2r$.
\end{proof}

A Lie algebra covering all the axes is also 
attainable in the following sense.
\begin{prop}\label{prop:Lie-sufficient}
	Let $P=(p_1,\ldots,p_r)\subset K[X]$ be a homogeneous linear polynomial with
	support $A$.  Then there is a flat extension $L/K$ and $\alpha\in
	\mathbb{T}^A(L)$ such that $p:=\sum_{a\in A}\alpha_a x_a\in P$, and for
	arbitrary tensor spaces $S$, $\Op[A]{S}{p}$ admits a weighted Lie product,
	and $\Op[A]{S}{P}\subset \Op[A]{S}{p}$.
\end{prop}

\begin{proof}
	Use the $p=\sum_{a\in A}\alpha_a x_a$ as in the proof of
	Theorem~\ref{mainthm:Lie-Asc} but for convenience use $-\alpha_0$ if $0\in
	A$. Thus, $\Op[A]{S}{p}$ is closed to the product $[\delta,\xi]_a:=\alpha_a
	\delta_a\xi_a-\alpha_a \xi_a\delta_a$, for $a\in A$.
\end{proof}

The existence of unital associative algebras are limited considerably by
Theorem~\ref{mainthm:Lie-Asc}, but the following corollary shows that the nuclei
are always examples.

\begin{coro}\label{cor:associative-rare}
	Assume $K$ is a field and let $p\in K[X]$ be a homogeneous linear
	polynomial. There exists $(\lambda,\rho)\in (K^{2}-\{0\})^{\zrange{\vav}}$,
	with $(\forall a\in\zrange{\vav})((\lambda_a:\rho_a)\neq (1:-1))$ such that
	for all $S\subset T$, $\Op{S}{p}$ is closed under $\bullet_{(\lambda,\rho)}$
	(as defined in~\eqref{eq:bullet-product}) if, and only if, 
	for some $a,b\in \zrange{\vav}$ and $\alpha,\beta\in K$, $p=\alpha x_a-\beta x_b$. 
	In this case, $\bullet_{(\lambda,\rho)}$ is associative.
\end{coro}

\begin{proof}
	Observe that $\dim \mathrm{span}_K\{p\}\leq 1$, so the the ``only if''
	follows from Theorem~\ref{mainthm:Lie-Asc}. For the ``if'' part, observe
	that there is $\tau\in\mathbb{T}$ such that $p^\tau=x_a-x_b$, or
	$p^\tau=x_a$ for some $a,b\in\zrange{\vav}$, so by
	Proposition~\ref{prop:Lie-sufficient}, it is enough to just consider the
	cases $p=x_a-x_b$ and $p=x_a$. In the first case,
	$\Op{S}{p}=\Op[\{a,b\}]{S}{p}\times\Omega_{\zrange{\vav}-\{a,b\}}$, and this
	is an associative algebra by Example~\ref{ex:nuclei}. The case $p=x_a$ can
	be checked by hand---$\Op{S}{p}$ is a (non-unital) subalgebra of
	$\prod_{c\in\zrange{\vav}}\End(V_c)$.
\end{proof}

\subsection{Associative algebras of transverse operators in valence 2}

We explore the implications of Theorem~\ref{mainthm:Lie-Asc} in valence
$\vav=2$, namely, for $K$-bilinear maps $\bra{t}:V_1\times V_2\bmto V_0$. In this
context, several families of naturally associative transverse operators have
been identified in the literature, e.g.~\citelist{\cite{BMW:exactseq}
\cite{Wilson:Skolem-Noether}}. Those are the nuclei of \eqref{def:adj} and
Example~\ref{ex:nuclei}. A further example is the centroid
(Example~\ref{ex:centroid}),
\[
	\Cent(S)=\OpSet{S}{\{x_0-x_1,x_0-x_2\}}{K},
\]
which is a unital subalgebra of $\End(V_0)\times\End(V_1)\times\End(V_2)$. We
apply Theorem~\ref{mainthm:Lie-Asc} to show that these are essentially all the
examples where $\OpSet{S}{P}{K}$ is a \emph{unital} subalgebra of
$\Omega_{A,B}$.  

\begin{lem}\label{lem:zero-sum}
	Let $P\subset K[X]$ be a linear homogeneous ideal. If there exists $A\subset
	\zrange{\vav}$ such that for every $S\subset T$,
	$(1_{V_a})_{a\in\zrange{\vav}} \in \Op{S}{P}$, then for every linear
	homogeneous polynomial $\sum_a\lambda_a x_a\in P$, we have
	$\sum_a\lambda_a=0$.
\end{lem}

\begin{proof}
	Let $p=\sum_a\lambda_a x_a\in P$. Consider the unit tensor
	$\bra{\1}:K^{\vav}\to K$ given by $\vav$-fold product in $K$. Then
	$\omega:=(1_{V_a})_{a\in \zrange{\vav}}\in \OpSet{\1}{p}{K}$, hence
	$0=\bra{\1}p(\omega)\ket{1_K,\dots,1_K}=\sum_a\lambda_a$.
\end{proof}

\begin{prop}\label{prop:valence-two-unital-assoc-subalgs}
	Let $K$ be a field, $\vav=2$, and $P\subset K[x_0,x_1,x_2]$ be a linear
	homogeneous ideal. Then there is $A\subset \{0,1,2\}$ such that for every
	$S\subset T$ and $L\in\Comm{K}$, the set $\OpSet[A]{S}{P}{L}$ is a
	unital subalgebra of $\Omega_{A}(L)$ if, and only if, one of the following
	holds.
	\begin{enumerate}
		\item[(0)] $P=0$; in this case $\OpSet{S}{P}{K} =
		\prod_{a\in\{0,1,2\}}\End(V_a)$.
		\item[(1)] $P=(x_a-x_b)$ for distinct $a,b\in \{0,1,2\}$; in this case,
		$\OpSet[\{a,b\}]{S}{P}{K} = \Nuc_{a,b}(S)$.
		\item[(2)] $P=(x_0-x_1,x_0-x_2)$; in this case $\OpSet{S}{P}{K} = \tau
		\Cent(S)$.
	\end{enumerate}
\end{prop}

\begin{proof}
	The ``if'' implication was explained in the comment preceding
	Lemma~\ref{lem:zero-sum}, so we only prove the ``only if'' part. Let
	$p_1,\dots,p_r\in K[X]$ be linear homogeneous polynomials generating $P$.
	Write $p_i=\sum_a\lambda_{ia}x_a$ and consider the matrix
	$\Lambda=(\lambda_{ia})_{i,a}\in \mathbb{M}^{r\times (\vav+1)}$; we write
	$\Lambda$ so that columns $1,2,3$ correspond the variables $x_2,x_1,x_0$.
	Then the rank of $\Lambda$ is $\dim_K\mathrm{span}\{p_1,\dots,p_r\}$, and
	the number of its nonzero columns is the support of $P$. By applying row
	operations, we may assume that $\Lambda$ is in echelon form and has no zero
	rows. In addition, by Lemma~\ref{lem:zero-sum}, the sum of the columns of
	$\Lambda$ is zero. We now break into three cases.
	
	If $\rank \Lambda=0$, then $P=0$ and case (0) applies.

	If $\rank \Lambda=1$, then $\Lambda$ is a row matrix which, by
	Corollary~\ref{cor:associative-rare}, has at least one zero entry. Since the
	entries of $\Lambda$ add   to $0$, case (1) must apply.
	
	Finally, if $\rank \Lambda=2$, then $\Lambda$ has one of the following three
	forms: $[\begin{smallmatrix} 1 & 0 & *\\ 0 & 1 & *\end{smallmatrix}]$,
	$[\begin{smallmatrix} 1 & * & 0\\ 0 & 0 & 1\end{smallmatrix}]$,
	$[\begin{smallmatrix} 0 & 1 & 0\\ 0 & 0 & 1\end{smallmatrix}]$. As the
	columns of $\Lambda$ add to $0$, $\Lambda=[\begin{smallmatrix} 1 & 0 & -1\\
	0 & 1 & -1\end{smallmatrix}]$, so we are in case (2).
\end{proof}

\begin{remark}
	Assume $K$ is a field and $\vav=2$. There are further examples of linear
	homogeneous ideals $P\subset K[x_0,x_1,x_2]$ such that for all $S\subset T$
	and $L\in\Comm{K}$, $\OpSet{S}{P}{K}$ is an associative algebra
	relative to a product of the form \eqref{eq:bullet-product}. For example,
	take any $\tau\in \mathbb{T}(K)$ and apply it to the ideals in
	Proposition~\ref{prop:valence-two-unital-assoc-subalgs}---note that the
	resulting $\OpSet{S}{P^\tau}{L}$ is in general not a
	\emph{unital-subalgebra} (though it may be unital for some other unit) of
	$\Omega_{A}$ for any $A\subset\zrange{\vav}$. By analyzing echelon forms as
	in the proof of Proposition~\ref{prop:valence-two-unital-assoc-subalgs}, we
	see that the only nonzero ideals $P$ with this property are of the form
	$(x_a-x_b)^\tau$ or $(x_0-x_1,x_0-x_2)^\tau$ for some $\tau\in
	\mathbb{T}(K)$, or one of the following degenerate cases:
	\begin{enumerate} 
		\item[($1'$)] $P=(x_a)$ for $a\in\{0,1,2\}$; in this case,
		$\OpSet{S}{P}{K} = 0_{V_a}\times \prod_{b\neq a}\End(V_a)$ whenever $S$
		is fully nondegenerate.
		\item[($2'$)] 
		$P=(x_a-x_b,x_c)^{\tau}$ for distinct $a,b,c\in\{0,1,2\}$ and $\tau\in
		\mathbb{T}(K)$; in this case, $\OpSet{S}{P}{K} = \Nuc_{a,b}(S) \times
		0_{V_c}$ whenever $S$ is fully nondegenerate.
		\item[($2''$)] $P=(x_a,x_b)$ for distinct $a,b \in\{0,1,2\}$; in this
		case, $\OpSet{S}{P}{K} = 0_{V_a}\times 0_{V_b}\times\End(V_{4-a-b})$
		whenever $S$ is fully nondegenerate.
		\item[($3$)] $P=(x_1,x_2,x_3)$; in this case, $\OpSet{S}{P}{K} = 0$
		whenever $S$ is fully nondegenerate.
	\end{enumerate}
	The use of $\Op[A]{S}{P}$, $A=\sup P$, eliminates all these degenerate
	cases.
\end{remark}

%==============================================================================

%==============================================================================
%\input{groups.tex}
%==============================================================================

\section{Operator group schemes, Theorem~\ref{thm:group-intro}, 
and homotopism categories}\label{sec:groups}

This section concerns operator families whose invertible operators are
intrinsically subgroups. We prove Theorem~\ref{thm:group-intro} and a much
stronger form involving general ideals instead of only principal ideals as
stated in the introduction. In more detail, we study $A\subset\zrange{\vav}$,
$B\subset\comp{A}$, and ideals $P\subset K[X_{A\sqcup B}]$ where the scheme
$\Op[A,B]{S}{P}^\times :=\Op[A,B]{S}{P}\cap \Omega_{A,B}^\times$ is a
\emph{subgroup subscheme} independent of $S$.  We keep these notations fixed
throughout this section.

\begin{center}
	\emph{In this section $K$ is a field.}
\end{center}

\subsection{Group schemes}

For the purpose of this paper, an affine group scheme (of finite type) over $K$
is a functor $\mathfrak{G}$ from $\Comm{K}$ to $\mathsf{Group}$ such that
$\mathfrak{G}$ becomes an affine $K$-scheme when regarded as a functor from
$\Comm{K}$ to $\mathsf{Set}$, see Section~\ref{sec:Galois}. A (closed) subgroup
scheme of $\mathfrak{G}$ is a (closed) subscheme $\mathfrak{H}$ of
$\mathfrak{G}$ which is a group scheme relative to the product of
$\mathfrak{G}$, i.e.,  for each $L\in \Comm{K}$, $\mathfrak{H}(L)$ is a subgroup
of $\mathfrak{G}(L)$. We refer the reader to \cite{Waterhouse} for an extensive
treatment.

The most important example that we shall consider is $\mathcal{GL}(V)$, the
functor sending $L\in\Comm{K}$ to the group $\mathrm{Aut}_L(L\otimes V)$, where
$V$ is a finite dimensional $K$-vector space.  So $\Omega_{A,B}^{\times}(L)=
\prod_{a\in A}\mathrm{Aut}(L\otimes V_a)\times \prod_{b\in B}
\mathrm{Aut}(L\otimes V_b)^{\op}$.  Thus, $\Op[A,B]{S}{P}^{\times}$ formally 
means that
$
	\Op[A,B]{S}{P}^{\times}(L)
	=\Op[A,B]{S}{P}(L)\cap \Omega_{A,B}^{\times}(L).
$
By proposition~\ref{prop:Z}, $\Op{S}{P}^\times$ is a closed subscheme of
$\Omega^{\times}$, but {\it a priori} not a subgroup scheme.  We prove:

\addtocounter{submainthm}{1}
\begin{submainthm}\label{thm:group-full}
	Fix $K$ a field, $A\subset\zrange{\vav}$, and $B\subset \comp{A}$.
	\begin{enumerate}[(a)]
		\item For arbitrary tensor spaces $S$ of valence $\vav$,
		$\Op[A,B]{S}{X^{e_1}-X^{f_1},\ldots,X^{e_{r}}-X^{f_{r}}}^{\times}$ is a 
		subgroup scheme, whenever for every $i$, $e_i,f_i:A\sqcup B\to
		\{0,1\}$, $\supp e_i\subset A$, and $\supp f_i\subset B$.

		\item If $P\subset K[X_{A\sqcup B}]$ such that for arbitrary tensor
		spaces $S$ of valence $\vav$, $\Op[A,B]{S}{P}^{\times}$ is a subgroup
		scheme, then there is $Q=(X^{e_1}-X^{f_1},\ldots, X^{e_r}-X^{f_r})$ such
		that for all $i$, $e_i,f_i:A\sqcup B\to \{0,1\}$, and for all $S$, 
		$\Op[A,B]{S}{P}^{\times}=\Op[A,B]{S}{Q}^{\times}$.		
	\end{enumerate}
\end{submainthm}

Notice Theorem~\ref{thm:group-full}(b) lacks the requirement of disjoint
supports found in part (a). We can prove this in the special case that $Q$ is a
principal ideal which recovers Theorem~\ref{thm:group-intro}.  We are not aware
of any example preventing this condition nor have we found a means to prove it.
So we leave this as an open question in this work.  We later phrase the
precise obstacle in terms of lattices, Question~\ref{quest:full-converse}.

\subsection{A motivating example, proof of Theorem~\ref{thm:group-full}(a)}

In many ways the prototypical example of automorphisms of an algebra offers the
clues for a sufficient conditions to ensure that $\Op{S}{P}^\times$ is naturally
a group scheme. An isomorphism of algebras is a map $\omega$ such that
$\omega(x)\omega(y)=\omega(xy)$.  Generalizing slightly to tensors of valence
$\vav=2$ this becomes an expression of the form 
\begin{align*}
   \langle t| \omega_1 v_1,\omega_2 v_2\rangle=\omega_0\langle t|v_1,v_2\rangle.
\end{align*}  
From there, we see the operator $\omega$ is annihilated by $x_0-x_1 x_2$
relative to $\bra{t}$.  The mechanics of the proof that automorphisms of an
algebra form a group generalizes to demonstrate the following.

\begin{prop}
	$\Op{S}{x_0-x_1\cdots x_{\vav}}^{\times}$ is a group subscheme of
	$\Omega^{\times}=\prod_{a\in \zrange{\vav}}\mathcal{GL}(V_a)$.
\end{prop}

There are several important features in this example.  First, we have a binomial
(for us, binomials have the form $\alpha X^e-\beta X^f$, for nonzero
$\alpha,\beta$); second, its coefficients are $1$; third, the exponents of the
variables are all either $0$ or $1$. 

Suppose we generalize this situation even slightly and use $x_0x_1-x_2$.  This
polynomial has all the qualities of the first one we used.  However, we find a
natural closure under a different product rule. For all $\tau,\omega\in
\Op{S}{x_0x_1-x_2}^\times$, we have
\begin{align}\label{EQ:unusual-mult}
	\langle t| v_1, \omega_2\tau_2v_2\rangle
		& = \omega_0\langle t|\omega_1v_1,\tau_2v_2 \rangle
		  = \omega_0\tau_0\langle t| \tau_1\omega_1v_1, v_2 \rangle,
\end{align}
so $\Op{S}{x_0x_1-x_2}^\times$ is a group  scheme---not with the
operation in $\Omega^\times$---but with the action
\begin{align*}
	(\omega_0,\omega_1,\omega_2) \bullet (\tau_0,\tau_1,\tau_2) 
		& = ( \omega_0\tau_0 ,  \tau_1\omega_1 ,  \omega_2\tau_2 ).
\end{align*}
In other words, $\Op{S}{x_0x_1-x_2}$ is a group subscheme
of $\Omega_{A,B}^{\times}$ where $A=\{0,2\}$ and $B=\comp{A}=\{1\}$.

\begin{defn}
	For disjoint $A,B\subset \zrange{\vav}$, an ideal $P\subset K[X]$ is
	\emph{$(A,B)$-composable} if for every frame of finite-dimensional
	$K$-vector spaces $\{V_0,\ldots,V_{\vav}\}$ and every $S\subset V_0\oslash
	\cdots \oslash V_{\vav}$, $\Op[A,B]{S}{P}^\times$ is a group subscheme of
	$\Omega_{A,B}^\times$.
\end{defn}

We now prove Theorem~\ref{thm:group-full}(a).

\begin{prop}\label{prop:group-forward}
	Suppose that  $P=(X^{e_1} - X^{f_1},\ldots,X^{e_r}-X^{f_r})$,
	where for all $i$, $e_i,f_i:\zrange{\vav}\to \{0,1\}$,
	$\supp e_i\subset A$, and $\supp f_i\subset B$ with $A\cap B=\emptyset$.
	Then $P$ is $(A,B)$-composable. 
\end{prop}

\begin{proof}
	We may assume that $A=\cup_i \supp e_i$, $B=\cup_i \supp f_i$ and for
	convenience that $0\notin A\sqcup B$---if not modify to include $\omega_0$
	as in~\eqref{EQ:unusual-mult}.  Set $C=\range{\vav}-A-B$.  So for all $L\in
	\Comm{K}$, $\omega,\tau\in \Op{S}{P}^\times(L)$, and $t\in S$,  we have
	\begin{align}\label{eq:binomial-constraint}
		&\bra{t_L} \omega_A v_A,v_B,v_{C}\ra 
		- \bra{t_L} v_A,\omega_B v_B,v_{C}\ra=0.
	\end{align}
	This holds likewise for $\tau$. A  computation similar
	to~\eqref{EQ:unusual-mult} now shows that $\omega\bullet \tau\in
	\Op{S}{P}^\times(L)$, where $\bullet$ is   the multiplication in
	$\Omega_{A,B}^\times$. By substituting $ \omega^{-1}_{A\cup B}v_{A\cup B}$
	instead of $v_{A\cup B}$ in~\eqref{eq:binomial-constraint}, we get
	\[
		\bra{t_L}  v_A,\omega_B^{-1}v_B,v_C\ra 
		-   \bra{t_L} \omega_A^{-1}v_A,\ v_B,v_C\ra=0,
	\]
	so $\omega^{-1}\in \Op[A,B]{S}{P}^\times(L)$.
	Since $1\in \Op[A,B]{S}{P}^\times(L)$, the proposition follows.
\end{proof}

\subsection{Binomial generation}

In the remainder of this section, we establish a  partial converse to
Theorem~\ref{thm:group-intro}(a). In contrast with the situation considered in
Section~\ref{sec:universality}, it could happen that two distinct ideals
$P,Q\subseteq K[X]$ satisfy $\Op{S}{P}^\times=\Op{S}{Q}^\times$ for every frame
$V_0,\dots,V_{\vav}$ and tensor space $S\subseteq V_0\oslash \dots\oslash
V_{\vav}$. We resolve this redundancy in
Proposition~\ref{PR:redundancy-for-groups}, and then show that if $P$ is
$(A,B)$-composable for some $A\subseteq \zrange{\vav}$, then a particular ideal $Q$
satisfying $\Op[A,B]{S}{P}^\times=\Op[A,B]{S}{Q}^\times$ is generated by binomials.

% \medskip

We extend $K[X]$ to the Laurent polynomial ring $K[X^{\pm}]=K[x_0^{\pm 1},
\ldots, x^{\pm1}_{\vav}]$. Given an ideal $P\subset K[X]$, we let $(P:X^\infty)
= \{ p \in K[X] \mid (\exists e)(X^e p \in P)\}$; this is an ideal of $K[X]$.
Note that
\begin{align}\label{eq:P-over-Xinf}
	(P:X^\infty) & = \{ p \in K[X] \mid (\exists e)(X^e p \in P)\} =
	PK[X^{\pm}] \cap K[X],
\end{align}
where $PK[X^{\pm}]$ is the ideal generated by $P$ in $K[X^\pm]$. The notation
applies likewise to $K[X_{A\cup B}]$.

We shall use the \emph{unit tensor} in our proofs.  Let $V_0=K^{\vee}$ and for $a>0$, $V_a=K$
\begin{align}\label{def:unit-tensor}
	\bra{\1}\alpha_1,\ldots,\alpha_{\vav}\ra (\alpha_0)  & = \alpha_0(\alpha_1 \cdots \alpha_{\vav})
\end{align}
\begin{remark}
Note that if it is important to prove the the claims for fixed frame then 
observe that we can extend $V_a=K\oplus U_a$ and define $\bra{\1}$ to have $U_a$
as the $a$-axis radical and map only onto $K\hookrightarrow V_0$ and thus consider
tensors in a fixed frame $V_0\oslash\cdots\oslash V_{\vav}$.  However this means that
throughout the proof we have to separately account for this degeneracy.  Nevertheless,
Theorem~\ref{thm:group-full} applies in the context of a fixed $T\neq 0$ and $S\subset T$.
\end{remark}

\begin{prop}\label{PR:redundancy-for-groups}
	Let $P,Q\subset K[X_{A\cup B}]$ be ideals. The following conditions are equivalent.
	\begin{enumerate}[(a)]
		\item
		$(P:X_{A\cup B}^\infty)=(Q:X_{A\cup B}^\infty)$.
		\item
		$PK[X^{\pm}_{A\cup B}]=QK[X^{\pm}_{A\cup B}]$.
		\item
		$\Op[A,B]{S}{P}^\times =\Op[A,B]{S}{Q}^\times$ for every
		frame $V_0,\dots,V_{\vav}$ and tensor space
		$S\subseteq V_0\oslash \cdots\oslash V_{\vav}$.
		\item
		$\Op[A,B]{\1}{P}^\times =\Op[A,B]{\1}{Q}^\times$,
		where $\bra{\1}:K^{\vav}\bmto K$ is the unit tensor as defined in~\eqref{def:unit-tensor}.
	\end{enumerate}
	In particular, $(\forall S)(\Op[A,B]{S}{P}^\times =\Op[A,B]{S}{(P:X^\infty)}^\times)$.
\end{prop}

\begin{proof}
	It suffices to prove this for $\zrange{\vav}=A\cup B$.
	(c)$\implies$(d) is immediate.
	
	(b)$\implies$(a) follows from \eqref{eq:P-over-Xinf}.
	
	To prove (a)$\implies$(c), it is enough to prove that
	$\Op{S}{P}^\times=\Op{S}{(P:X^\infty)}^\times$. The inclusion
	$\Op{S}{P}^\times\supseteq \Op{S}{(P:X^\infty)}^\times$ follows from
	Theorem~\ref{thm:correspondence}, because $P\subset (P:X^\infty)$. To see
	the converse, let $p\in (P:X^\infty)$, so there is an $e:\zrange{\vav}\to
	\mathbb{N}$ such that $X^e p\in P$.  Let $t\in S$, let $L\in \Comm{K}$ and
	let $\omega\in \Op{S}{P}^\times(L)$. Since $\omega\in\Op{S}{P}^\times(L)$ and
	$p\cdot X^e=X^ep\in P$,  for all $\ket{v}$, 
   	\begin{align*}
      0=\bra{t_L} (p\cdot X^e)(\omega) \ket{v} = \omega_0^{e(0)} \bra{t_L} p(\omega) 
      \ket{\omega_1^{e(1)} v_1,\ldots, \omega_{\vav}^{e(\vav)} v_{\vav}}.
   	\end{align*}
   	As each $\omega_a$ is invertible,  it follows that
    $\omega\in\Op{S}{p}^\times(L)$. This holds for all $p\in P$, so
    $\omega\in \Op{S}{P}^\times(L)$, which is what we want. 
    
    Next, we show that (d)$\implies$(b).
    By assumption, $\Op{\1}{P}^\times(L)=\Op{\1}{Q}^\times(L)$ for all
    $L\in\Comm{K}$. Take $L=K[X^\pm]/PK[X^\pm]$, let $\hat{x}_i$ denote the
    image of $x_i$ in $L$, and let
    $\omega:=(\hat{x}_{0},\ldots,\hat{x}_{\vav})\in \Op{\1}{Q}^{\times}(L)$.
    For all $p\in K[X]$, we have 
	\begin{align*}
		\bra{\1_L} p(\omega)\ket{1,\dots,1}=  p(\hat{x}_{0},
			\dots,\hat{x}_{\vav}).
	\end{align*}
	If $p\in Q$, then the right hand side is $0$; so, $p(\hat{x}_{0},
    \ldots,\hat{x}_{\vav})=0$ in $L=K[X^\pm ]/PK[X^\pm]$. This means that $p\in
    PK[X^\pm]$, so we have shown $QK[X^\pm]\subset PK[X^\pm]$. The reverse
    inclusion is shown similarly.
    
    The last assertion of the proposition follows from the equivalence,
    because $((P:X^\infty):X^\infty)=(P:X^\infty)$.
\end{proof}

Proposition~\ref{PR:redundancy-for-groups} means  that we should look for
necessary conditions on   $(P:X^\infty)$ that will guarantee that $P$
(equivalently, $(P:X^\infty)$) is $(A,B)$-composable for disjoint $A,B\subseteq
\zrange{\vav}$.

\begin{prop}\label{prop:need-binomial}
   Suppose $A\subseteq \zrange{\vav}$ and $P\subseteq K[X]$ is an ideal. If $P$
   is $(A,B)$-composable, then $(P:X^\infty)$ is generated by binomials.
\end{prop}

The proof is based on results of Eisenbud and Sturmfels~\cite{ES}*{\S2}.

\begin{proof}
	Take $S$ to be the unit tensor $\bra{\1}:K\times\cdots\times K\bmto K$. Then
	$\Op{\1}{P}$ is the affine zero set $Z(P)$ and $\Op{\1}{P}^\times$ is
	$Z(P)\cap\mathbb{T}$, where $\mathbb{T}$ is as in
	Section~\ref{sec:universality}. Since $P$ is $(A, B)$-composable, and since
	$\mathbb{T}$ is commutative, $\Op{\1}{P}^\times$  is a subgroup scheme of
	$\mathbb{T}$. Let $M$ denote the ideal of Laurent polynomials in $f\in
	K[X^{\pm}]$ which vanish on $\Op{\1}{P}^\times(L)$, for every
	$L\in\Comm{K}$. 
	
	We claim that $M=PK[X^{\pm}]$. To see this, let $P_1=M\cap K[X]$ and note
	that $M=P_1K[X^{\pm 1}]$. Therefore, by
	Proposition~\ref{PR:redundancy-for-groups}, it is enough to show that
	$\Op{\1}{P}^\times(L)=\Op{\1}{P_1}^\times(L)$ for all $L\in \Comm{K}$. 
	Note that $\Op{\1}{P_1}^\times(L)$ coincides with the zero of $M$ in
	$\mathbb{T}(L)$. Furthermore, since $P\subseteq M$,   the zeroes of $M$ in
	$\mathbb{T}(L)$ are contained in the zeroes of $P$, and the converse holds
	by construction of $M$. Thus, $\Op{\1}{P}^\times(L)=\Op{\1}{P_1}^\times(L)$
	and the claim follows.
	
	Eisenbud--Sturmfels \cite{ES}*{Proposition~2.3} showed that $M$ is generated
	by binomials $b_1,\dots,b_r\in K[X]$ which generate $M$ in $K[X^\pm]$.  Let
	$Q$ be   the ideal  they generate  in $K[X]$. Then
	$QK[X^{\pm}]=M=PK[X^{\pm}]$. By \cite{ES}*{Corollary~2.5} (recalled below),
	$(Q:X^\infty)$ is generated by binomials. Since $(P:X^\infty)=(Q:X^\infty)$,
	we are done.
\end{proof}

In order to proceed, we need a  description of the binomial ideals in $K[X]$ of
the form $(P:X^\infty)$. This is given by a considerable work of Eisenbud and
Sturmfels \cite{ES}*{Theorem~2.1}, which we now recall.

A \emph{partial character} $\rho$ is a group homomorphism from a sublattice
$\mathcal{L}_\rho$ of $\mathbb{Z}^{\zrange{\vav}}$ into $K^\times$. By a
sublattice, we mean a subgroup of $\mathbb{Z}^{\zrange{\vav}}$ whose index is
not necessarily finite. We always denote the domain of a partial character
$\rho$ by $\mathcal{L}_\rho$. Given a partial character $\rho$, define the
binomial $K[X^{\pm}]$-ideal
\begin{align}\label{eq:I-lat}
   I(\rho) = \left(X^m - \rho(m) ~\middle|~ m\in \mathcal{L}_\rho\right) \subset K[X^{\pm}].
\end{align}
Eisenbud and Sturmfels showed that all proper binomial ideals in $K[X^{\pm}]$
are of the form $I(\rho)$, with $\rho$ uniquely determined.

\begin{thm}[{\cite{ES}*{Theorem~2.1}}]\label{thm:partial-char}
   Let $K$ be a field. 
   \begin{enumerate}
      \item For every proper   binomial ideal $I\subseteq K[X^{\pm}]$, there
      is a unique partial character $\rho$ of $\mathbb{Z}^{\zrange{\vav}}$ such that
      $I=I(\rho)$. 
      \item If $\mathcal{B}\subseteq \mathbb{Z}^{\zrange{\vav}}$ is a basis of the
      lattice $\mathcal{L}_\rho$, then
      $I(\rho) = \left(X^{m} - \rho(m) ~\middle|~ m\in\mathcal{B}\right)$.
   \end{enumerate}
\end{thm}

There is also a counterpart for binomial ideals in $K[X]$: Given
$m\in\mathbb{Z}^{\zrange{\vav}}$, let $m_+, m_-\in\mathbb{N}^{\zrange{\vav}}$ denote the
\emph{positive} and \emph{negative} parts of $m$, such that $m = m_+ - m_-$. For
a given partial character $\rho$ on $\mathbb{Z}^{\zrange{\vav}}$, let 
\begin{align}\label{eq:I+}
   I_+(\rho) = \left(X^{m_+} - \rho(m)X^{m_-} ~\middle|~ m\in \mathcal{L}_p \right) \subset K[X].
\end{align}

\begin{thm}[{\cite{ES}*{Corollary~2.5}}]\label{thm:partial-char-II}
	If $P\subset K[X]$ is a binomial   ideal not containing any
	monomial, then there exists a unique partial character
	$\rho$ such that $(P:X^\infty)=I_+(\rho)$. (In particular,
	if $P=(P:X^\infty)$, then $P=I_+(\rho)$.)
\end{thm}

Putting together all previous results, we get:

\begin{coro}\label{coro:need-binomials}
   Fix disjoint $A,B\subseteq \zrange{\vav}$, and let $P\subseteq K[X]$ be an ideal.
   If $P$ is $(A,B)$-composable, then for all frames of 
   $K$-vector spaces $\{V_0,\dots, V_{\vav}\}$ and for all $S\subset V_0\oslash
   \cdots \oslash V_{\vav}$, $\Op[A,B]{S}{P}^{\times}=\Op[A,B]{S}{(P:X^\infty)}^{\times}$,
   and there is a unique partial character $\rho$ of $\mathbb{Z}^{\zrange{\vav}}$ such
   that $(P:X^\infty) = I_+(\rho)$. 
\end{coro}

\begin{proof}
	This follows from Theorem~\ref{thm:partial-char-II}, Proposition~\ref{prop:need-binomial}
	and Proposition~\ref{PR:redundancy-for-groups}, provided
	we show that $P$ contains no monomials. 
	Consider the unit tensor $\1:K^{\vav}\to K$. By assumption,
	the identity $\omega:=(1_{V_{\vav}},\dots,1_{V_0})$ is in $\Op{S}{P}(K)$.
	Since $\bra{\1} X^e(\omega)\ket{1,\dots,1}=1$, for all $e\in {\mathbb{N}}^{\zrange{\vav}}$,
	there can be no monomials in $P$. 
\end{proof}

Before pursuing the proof of Theorem~\ref{thm:group-full}(b), we pause to
describe the conditions in part (a) in terms of the associated partial character $\rho$
and the lattice $\cal L_\rho$.  
Recall that we are considering $P=(X^{e_1}-X^{f_1},\cdots,X^{e_r}-X^{f_r})$
with $e_i,f_i:\zrange{\vav}\to \{0,1\}$ and $(\cup_i \supp e_i)\cap
(\cup_i \supp f_i)=\emptyset$. We write ${\cal L}(P)={\cal L}_\rho$; see \eqref{eq:I-lat}.

First, observe that the partial character $\rho$ associated to $(P:X^\infty)$ is
the same as the one associated to $PK[X^{\pm 1}]$. Using the theorems stated
above, $\cal L_\rho$ is spanned by $\{e_1-f_1,\dots,e_r-f_r\}$, and moreover,
$\rho:{\cal L}_\rho\to K^\times$ is trivial ($\rho({\cal L}_{\rho})=1$). We
prove that the triviality of $\rho$ is necessary for $(A, B)$-composable ideals
in Proposition~\ref{prop:need-trivial}.

The condition that $e_i,f_i:\zrange{\vav}\to \{0,1\}$ and $(\cup_i \supp
e_i)\cap (\cup_i \supp f_i)=\emptyset$ implies that $\mathcal{L}(P)$ has a very
particular generating set. We can articulate this by saying that there is
$\tau:\zrange{\vav}\to \{\pm 1\}$ such that $\tau  \mathcal{L}(P) $
(coordinate-wise multiplication) is generated by vectors in the hypercube
$[0,1]^{\zrange{\vav}}$, i.e. 
\begin{align}\label{eq:nec-cond-group}
	\tau\mathcal{L}(P) = \left\langle \tau \mathcal{L}(P) \cap [0,1]^{\zrange{\vav}}\right\rangle.
\end{align}
Specifically, choose $\tau$ to scale the union of the support of the $f_i$ by
$-1$, so that $(\forall i)(\tau(e_i-f_i):\zrange{\vav}\to\{0,1\})$.

\begin{quest}\label{quest:full-converse}
   Let $\rho$ be a trivial partial character on $\mathbb{Z}^{\zrange{\vav}}$. If there
   exists disjoint $A,B\subset \zrange{\vav}$ such that $I_+(\rho)$ is $(A,B)$-composable, then must
   there exist $\tau\in \{\pm 1\}^{\zrange{\vav}}$ such that \eqref{eq:nec-cond-group}
   holds?
\end{quest}

Here, we establish a weaker necessary condition, which in particular
affirmatively answers Question~\ref{quest:full-converse} when $I_+(\rho)$ is
principal. This is enough to establish Theorem~\ref{thm:group-full}(b).

\subsection{The need for unit coefficients}

In light of Corollary~\ref{coro:need-binomials}, we are left to expose what
properties of partial characters are necessary and sufficient to allow
$\Op[A,B]{S}{P}^{\times}$ to be a subgroup scheme of some
$\Omega_{A,B}^{\times}$.  Identities and inverses give some immediate
conditions. For disjoint $A,B\subseteq \zrange{\vav}$, the identity of
$\Omega_{A,B}^\times$ is independent of $(A,B)$; namely, it is the identity of
$\Omega_{\zrange{\vav}}^\times$, denoted by $1_\Omega$. 

\begin{prop}\label{prop:need-trivial}
   Suppose $\{V_0,\dots, V_{\vav}\}$ is a frame of finite-dimensional $K$-vector
   spaces, and $\rho$ a partial character on $\mathbb{Z}^{\zrange{\vav}}$, with $P =
   I_+(\rho)$. Then $\rho$ is trivial if, and only if, for every $S\subset
   V_0\oslash \cdots \oslash V_{\vav}$ and for all $L\in\Comm{K}$,
   $1_\Omega\in\Op[A,B]{S}{P}(L)$. Moreover, if $S\subset V_0\oslash \cdots \oslash
   V_{\vav}$ and $L\in\Comm{K}$ such that $1_\Omega\in\Op[A,B]{S}{P}^\times(L)$, then
   $\omega\in\Op[A,B]{S}{P}^\times(L)$ implies $\omega^{-1}\in\Op[A,B]{S}{P}^\times(L)$.
\end{prop}

\begin{proof}
	First, assume that for all $S$ and $L$, $1_\Omega\in \Op[A,B]{S}{P}(L)$. In
	particular, this holds for the unit tensor $\bra{\1}$, where $V_a=K$ for all
	$a\in\zrange{\vav}$, and $\bra{\1}$ is the unit tensor
	\eqref{def:unit-tensor}. In this case, given $m\in \mathcal{L}_\rho$, let $p
	= X^{m_+} - \rho(m)X^{m_-}\in I_+(\rho)=P$. Then,
	\begin{align*}
	   0 = \bra{\1} p(1_{\Omega} )\ket{1, \dots, 1} = 1 - \rho(m).
	\end{align*}
	Hence, $\rho(m)=1$. Conversely, because the sum of the coefficients of
	$X^{m_+} - X^{m_-}$ vanishes, it follows that for all $S$, $1_\Omega\in\Op[A,B]{S}{P}(L)$.

   Now suppose that for some $S$ and $L$, $1_{\Omega}\in\Op[A,B]{S}{I_+(\rho)}(L)$,
   so by the previous paragraph, $\rho$ (regarded
   as a character $\mathbb{Z}^{\zrange{\vav}}\to L^\times$) 
   is a trivial partial character on $\mathbb{Z}^{\zrange{\vav}}$.
   Let $e, f\in\mathbb{Z}^{\zrange{\vav}}$ such that $e-f\in \mathcal{L}_\rho$, so
   $p=X^e - X^f\in P=I_+(\rho)$. Let $\omega\in\Op[A,B]{S}{P}^\times(L)$, and for
   each $a\in\zrange{\vav}$, set $\tau_a = \omega_a^{-e(a) - f(a)}$, so that
   $\tau\in\Omega^\times_{A,B}(L)$. For all $t\in S$ and $\ket{v}$, 
   \begin{equation}\label{eqn:0-1}
      \begin{split}
      \bra{t}p(\omega^{-1})\ket{v}
      &= \omega_0^{-e(0)-f(0)}\omega_0^{f(0)} 
        \left\la t \middle| \omega_{\comp{0}}^{f(\comp{0})} 
            \omega_{\comp{0}}^{-e(\comp{0})-f(\comp{0})} v_{\comp{0}} \right\ra \\
      &\qquad - \omega_0^{-e(0)-f(0)}\omega_0^{e(0)} \left\la t \middle| 
            \omega_{\comp{0}}^{e(\comp{0})} \omega_{\comp{0}}^{-e(\comp{0})
                -f(\comp{0})} v_{\comp{0}} \right\ra \\ 
      &= -\tau_0 \bra{t} p(\omega) \ket{\tau_{\comp{0}} v_{\comp{0}}}=0.
      \end{split}
   \end{equation}
   It follows that $\omega^{-1}\in\Op[A,B]{S}{P}^\times(L)$. 
\end{proof}

\subsection{Necessary conditions on the lattice}
The remaining, delicate step is to consider when $\Op[A,B]{S}{P}^\times$ is closed
under composition. In light of the previous section, this falls exclusively to
conditions on the sublattice $\mathcal{L}_\rho$ associated to the (trivial) partial character
$\rho$.

We prove the next proposition by proving the contrapositive. For $a\in \zrange{\vav}-\{0\}$,
$m\in \mathbb{N}$, and $\pi : K^m \rightarrow K$ a $K$-linear map, define the
tensor $t(a, m, \pi)$ in the following way. Let $V_a = K^m$, and for $b\in
\zrange{\vav}-a$, set $V_b=K$. For all $v\in \prod_{b\ne 0} V_b$, let
\begin{align}\label{eqn:counter-tensor}
   \bra{t(a, m, \pi)} v \rangle = \pi(v_a)\prod_{b\ne a} v_b. 
\end{align}  

\begin{prop}\label{prop:need-split}
   Fix disjoint $A,B\subseteq \zrange{\vav}$, $a\in \zrange{\vav}$, $\rho$ a
   trivial partial character on $\mathbb{Z}^{\zrange{\vav}}$, and $P =
   I_+(\rho)$. If $P$ is $(A,B)$-composable, then the image of
   $\mathcal{L}_\rho$ under the $a$-th projection $\mathbb{Z}^{\zrange{\vav}}\to
   \mathbb{Z}$ is either $\mathbb{Z}$ or $0$.
\end{prop}

\begin{proof}
   Suppose that the image of  $\mathcal{L}_\rho$ under the $a$-th projection
    $\mathbb{Z}^{\zrange{\vav}}\to \mathbb{Z}$ is not $\mathbb{Z}$ or $0$. Then for all
    $ m\in\mathcal{L}_\rho$ with $m_a\neq 0$, we have $|m_a| > 1$. Thus, there exists $0\neq
    m\in\mathcal{L}_\rho$  such that $m_a > 1$ and $m_a$ is minimal. Set $p =
    X^{m_+} - X^{m_-}$, the binomial defined by $m$.

   Let $\Gamma$ be the permutation group on $\{1,\dots, m_a+1\}$ generated by
   all cycles of length $m_a$. Let $V=K^{m_a+1}$ be the permutation module of
   $\Gamma$, and suppose $U$ is a maximal subspace of $V$, not necessarily a
   $\Gamma$-submodule, to be specified later. For $\pi_U: V\rightarrow V/U$, set $t_U = t(a, m_a+1,
   \pi_U)$, as defined in~\eqref{eqn:counter-tensor}. Then for every
   $L\in\Comm{K}$, every $m_a$-cycle $\sigma\in\Gamma$ induces an operator
   $\omega^{\sigma}\in\Op[A,B]{t_U}{p}(L)$ defined by
   \begin{align*}
      \omega_b^\sigma &= \left\{ \begin{array}{ll} 
         \sigma & b=a, \\
         1 & b\ne a.
      \end{array}\right.
   \end{align*}
   Since $m_a$ is positive and minimal, for all $n\in\mathcal{L}_\rho$, $m_a$
   divides $n_a$, so $\omega^{\sigma}\in \Op[A,B]{t_U}{P}(L)$ as well. 
   
   Depending on the parity of $m_a$, the subgroup $\Gamma$ is either the full
   symmetric or alternating group on $\{1,\dots, m_a+1\}$. Therefore, there
   exists an integer $\ell > 1$, coprime to $m_a$, and  $\tau\in \Gamma$ with
   order $\ell$. Choose the maximal subspace $U\leq V$ such that $\tau^{m_a} U \ne
   U$. Then the operator $\omega^\tau$ is not contained in
   $\Op[A,B]{t_U}{P}^\times(L)$, and therefore,
   $\Op[A,B]{t_U}{P}^\times(L)$ is not a subgroup of $\Omega_{A,B}^\times(L)$.
   Hence, $P$ is not $(A,B)$-composable.
\end{proof}

\begin{ex}
	Propositions~\ref{prop:need-binomial},
	\ref{prop:need-trivial}
	and~\ref{prop:need-split} are not sufficient to decide whether, for
	$A\subset [2]$, the ideal $P=(x_0-x_1x_2,x_0x_1 - x_2)$, corresponding to
	the lattice $\langle (-1,1,1),(-1,-1,1) \rangle\leq \mathbb{Z}^{[2]}$, is
	$(A,\comp{A})$-composable. 
\end{ex}

\begin{proof}[Proof of Theorem~\ref{thm:group-full}(b)] 
	Suppose $P$ is $(A,B)$-composable. By Corollary~\ref{coro:need-binomials}
	and Proposition~\ref{prop:need-trivial}, there exists a unique trivial
	partial character $\rho\in\mathbb{Z}^{\zrange{\vav}}$ such that
	$(P:X^\infty) = I_+(\rho)$. Let $\cal B$ be a basis for $\cal L_\rho$. By
	Theorem~\ref{thm:partial-char}, $I(\rho)$ is generated by $\{X^{m_+} -
	X^{m_-}\mid m\in{\cal B}\}$ (however, this set need not generate $I_+(\rho)$
	in $K[X]$).  Then we have $e,f\in\{0,1\}^{\zrange{\vav}}$ by
	Proposition~\ref{prop:need-split}.
\end{proof}

\subsection{Proof of Theorem~\ref{thm:group-intro}}
In the forward direction apply Theorem~\ref{thm:group-full}(a) and for the
converse apply Theorem~\ref{thm:group-full}(b).  But observe furthermore that
when $\cal B$ is a singleton $\{m\}$, then $I(\rho)=(X^{m_+}-X^{m_-})$. Setting
$e=m_+$ and $f=m_-$, we see that $I_+(\rho)=(X^e-X^f)$.
\qed

\subsection{Homotopism categories}\label{sec:cats}

Theorem~\ref{thm:group-full} lets us identify those transverse operators, such
as isometries and adjoints \cite{Wilson:division}, that can be used to build
categories tensors, compared with operators like derivations that do not enjoy
this property. Beyond its mathematical interest, this inquiry allows for the
design of general, yet light-weight, data types for tensors; see
Section~\ref{sec:algorithms}. First, we explain the categorical developments. 

Fix a valence $\vav$ and a partition $\zrange{\vav}=A\sqcup B\sqcup C$. The objects
in our categories $\mathcal{C}$ are tensors --- though tensor spaces can also be
used as objects. The morphisms are \emph{transverse-linear} meaning that, given tensors
$\bra{t}\in V_0\oslash \cdots\oslash V_{\vav}$ and $\lla s| \in
U_0\oslash\cdots\oslash U_{\vav}$, the subset
of homomorphisms from $s$ to $t$ satisfy
\begin{align*}
	\hom_{\mathcal{C}}(s,t)\subset 
		\prod_{a\in A\sqcup B} \overline{V_a \oslash U_a} \times 
				\prod_{c\in C} 1_C
\end{align*}
where $\overline{V_a \oslash U_a}=V_a\oslash U_a$ if the $a$-axis is covariant (i.e.
$a\in A$ or $0=a\in B$) and $U_a\oslash V_a$ if the $a$-axis is contravariant.
If there exists $c\in C$ such that $U_c\neq V_c$, then the hom-set is empty. 
We have not said what conditions are imposed on members of
$\hom_{\mathcal{C}}(s,t)$ so the restriction of function composition axis-by-axis
(respecting the variance) need not be well-defined.  In one case it is. 
Let $0\in B$,
\begin{align}\label{def:common}
	\hom(s,t) & = \{ (\omega_A,\tau_B,1_C) \mid 
		\bra{s}\tau_B = \bra{t}\omega_A\}
	\}
\end{align}
defines a hom-set where axis-by-axis function composition is well-defined. (Note
that here it matters to observe that ) We now claim that these are essentially
the only choices that arise subject to any polynomial identity. 

For $P\subset K[X]$, say that $\mathcal{C}$ is a \emph{$P$-transverse tensor
category} if
\begin{align}\label{eqn:P-transverse-cat}
	\hom_{\mathcal{C}}(s,t)\circ \hom_{\mathcal{C}}(t,s) & \subset \Op[A,B]{S}{P}(K).
\end{align}
So for $\vav = 2$, the adjoint category would be the $(x_1-x_2)$-transverse
tensor category (with $A=\{1\}$ and $B=\{1\}$), and the isometry category would
be the $(x_1x_2-1)$-transverse tensor category (with $A=\{1, 2\}$ and
$B=\emptyset$).  See Figure~\ref{fig:common}.
\begin{figure}[!htbp]
	\begin{center}
		\begin{subfigure}{0.3\textwidth}
			\begin{tikzcd}
				\bra{s}:U_{2} \arrow[d,"\omega_{2}"] \arrow[r] & (U_1\arrow[d,"\omega_1"]\arrow[r] & U_0)\arrow[d,"\omega_0"]\\ 
				\bra{t}:V_{2} \arrow[r] & (V_1\arrow[r] & V_0)
			\end{tikzcd}
			\caption{$(x_0-x_1 x_2)$-transverse map.}\label{fig:cat-ddd}
		\end{subfigure}
		\hspace{3mm}
		\begin{subfigure}{0.3\textwidth}
			\begin{tikzcd}
			\bra{s}:U_{2} \arrow[d,"\omega_{2}"] \arrow[r]& (U_1\arrow[d,"\omega_1"] \arrow[r] & U_0)\arrow[d,equal]\\ 
			\bra{t}:V_{2} \arrow[r] & (V_1\arrow[r] & V_0)
		\end{tikzcd}
	\caption{$(x_1 x_2-1)$-transverse map.}\label{fig:cat-dde}
	\end{subfigure}
	\hspace{3mm}
	\begin{subfigure}{0.3\textwidth}
		\begin{tikzcd}
			\bra{s}:U_{2} \arrow[d,"\omega_{2}"] \arrow[r] & (U_1\arrow[r] & U_0)\arrow[d,equal]\\ 
			\bra{t}:V_{2} \arrow[r] & (V_1\arrow[u,"\omega_1"]\arrow[r] & V_0)
		\end{tikzcd}
		\caption{$(x_1-x_2)$-transverse map.}\label{fig:cat-due}
	\end{subfigure}
\end{center}
\caption{Morphisms in three transverse categories shown with interpretation 
maps $\bra{s}:U_0\oslash U_1\oslash U_2$ and $\bra{t}\in V_0\oslash V_1\oslash V_2$.}
\label{fig:common}
\end{figure}

The condition in~\eqref{eqn:P-transverse-cat} underpins many stated questions on
tensor categories, though often explored in different contexts such as Kronecker
modules and quiver representations; see \citelist{\cite{BFM} \cite{Wilson:division}}
and references there in.  Abstracting to the polynomials makes
certain questions of categories natural to resolve. For example, if $P$ is
linearly generated, then the $P$-transverse tensor category is abelian. Functors
are also easily described in terms of polynomials.  For instance the invertible
morphisms of a nonabelian $(x_a x_b-1)$-transverse tensor category embed into
the abelian $(x_a-x_b)$-transverse tensor category by observing that $x_a
x_b-1=x_a-x_b\in K[X^{\pm}]$.  Such tricks are essential to works like
\cite{BW:isom}. Theorem~\ref{thm:group-full} makes the following strong
constraint on $P$.

\begin{coro}\label{coro:cats}
	If $\mathcal{C}$ is a $P$-transverse tensor category, then its core groupoid
	(its subcategory of invertible morphisms) is a
	$(X^{e_1}-X^{f_1},\ldots,X^{e_r}-X^{f_r})$-transverse tensor category, where
	for each $i\in \range{r}$, $e_i,f_i:\zrange{\vav}\to \{0,1\}$. If $r=1$,
	then $\supp e_1 \cap \supp f_1=\emptyset$.
\end{coro}

\begin{proof}
	By the conditions of a category endomorphisms $\hom_{\mathcal{C}}(s,s)$ form
	a monoid. The assumptions of $P$-transverse imply that
	$\hom_{\mathcal{C}}(s,s)$ is a submonoid of some $\Op[A,B]{s}{P}$.  Passing
	to the groupoid and applying Theorem~\ref{thm:group-full} the claim follows.
\end{proof}

In Section~\ref{sec:algorithms}, we explore the implication to data types, and
in future work, we explore the categorical implications generally.   We point
out that $(X^e-X^f)$-transverse tensor categories always exist and are the
defined as in \eqref{def:common}.
These we call \emph{$(X^e-X^f)$-homotopism categories}. 

%==============================================================================

%==============================================================================
%\input{singular.tex}
\section{Monomial traits, singularity complexes, 
and Theorem~\ref{mainthm:singularity}}\label{sec:Singularities}

We detail the relationship between the spectral theory of a tensor and its
various subtensors.  Our focus concerns decomposing multiway arrays and general
tensors into block-diagonal and block-triangular forms.  Such algorithms are
used in numerous ways ranging from clustering problems in data \citelist{
\cite{chatroom} \cite{MFMc}}, to detecting radical and semisimple structure in
abstract algebraic systems \citelist{\cite{Holt-Rees} \cite{Ivanyos-Lux}}.  In
the context of Section~\ref{sec:cats} we can describe these as decomposing a
tensor, or even a tensor space, into a categorical product or coproduct; see
\citelist{\cite{Wilson:central} \cite{Wilson:division}}.  

\begin{remark}\label{rem:decompositions}
There are at least two other uses of the phrase \emph{decomposing tensors}. One
is to write $t=\sum_i \lambda_i u_1\otimes \cdots \otimes u_{\vav}$ with few
non-zero coefficients; see \cite{Lim:Spectral}.  These problems tend to be
NP-complete to solve, see the work of H\aa stad \cite{Haastad}.  A second tensor
decomposition notion is to write $t$ as a sequence of tensor contractions of
lower rank tensors, e.g. writing a matrix $M=AB$ where $A$ is ``tall-skinny''
and $B$ is ``short-fat''. Tucker and many others developed algorithms for that
sort of decomposition, but these too have difficult complexity \cite{engineer}.

By comparison, decomposing tensors into their categorical products and
coproducts has so far proved to be polynomial-time efficient and for
many classes almost linear time.  Implementations of these algorithms scale roughly
linearly with the speed of solving linear systems of equations. Partly
explaining the difference, decompositions into the first two families seek a
combinatorial optimality inside highly non-convex feasible regions.  Meanwhile
decompositions in a categorical context usually have Jordan--H\"older or
Krull--Schmidt type theorems asserting the effective uniqueness of the
decompositions (though as \cite{Wilson:central} shows this is not always
necessary).
\end{remark}

\subsection{Subtensors}
Sub-structure in algebraic contexts concerns monomorphisms in a category. As we
saw in Section~\ref{sec:cats}, we observed how tensors are contained in multiple
categories, and we focus on the homotopism categories.

We begin by pointing out that subtensors are characterized by tensor
singularities (i.e.~zeros) and their ``neighborhood'' of singularities.  Recall
the simplicial complex $\nabla$ from the \eqref{def:nabla}
\begin{align*} 
	\nabla(t;U) & = \{A\subset\zerovav\mid U_A\not\perp V_{\bar{A}}\},
	& U_A\bot V_{\bar{A}} \Leftrightarrow \left\{\begin{array}{cc}
		\langle t|U_{B},V_{\range{\vav}-B}\rangle\not\leq U_0 & A=\{0\}\sqcup B,\\
		\langle t|U_A,V_{\range{\vav}-A}\rangle\neq 0 & \textnormal{otherwise.} 
	\end{array}\right.
\end{align*}

\begin{ex}
	Consider the following $\mathbb{R}$-tensors with $\vav=2$.
	\begin{enumerate}[(a)]
	\item
		Consider the product $*:\mathbb{C}\times \mathbb{C}\bmto\mathbb{C}$ of
		$\mathbb{C}$ as an $\mathbb{R}$-algebra. Since $\mathbb{R}$ is a
		subalgebra $\mathbb{R}*\mathbb{R}\leq \mathbb{R}$, it follows that
		$\{0,1,2\}\notin \nabla:=\nabla(*;\mathbb{R})$. However,
		$\mathbb{R}*\mathbb{C},\mathbb{C}*\mathbb{R}\not\leq \mathbb{R}$, and
		$\mathbb{R}*\mathbb{R}\neq 0$, so $\{0,1\},\{0,2\},\{1,2\}\in \nabla$.
		Thus $\nabla$ is the edges of a triangle.

	\item 
		Consider $R=\mathbb{M}_2(\mathbb{R})$ as the $\mathbb{R}$-tensor
		$\bra{t}:R\times R \bmto R$. Take $I=\left\{ \left[\begin{smallmatrix} a
		& b\\ 0 & 0 \end{smallmatrix}\right]\right\}\leq R$. Then $I*I \leq I$
		and $I*R\leq I$, so that $\{0,1,2\},\{0,1\}\not\in \nabla:=\nabla(*;I)$.
		But $R*I\not\leq I$ and $I*I\neq 0$ so $\{0,2\},\{1,2\}\in \nabla$.
		Therefore, $\nabla$ is two sides of a triangle.
	
	\item 
		Consider the product tensor of $R=\left\{\left[\begin{smallmatrix} a & b
		\\ 0 & c \end{smallmatrix}\right]\right\}$, and take $I=\left\{
		\left[\begin{smallmatrix} 0 & b\\ 0 & 0
		\end{smallmatrix}\right]\right\}\leq R$. Then $I*R,R*I\leq I$ and
		$I*I=0$ so $\{0,1\},\{0,2\},\{1,2\}\not\in \nabla:=\nabla(*;I)$.
		However, $I*R\neq 0$, $R*I\neq 0$, and $R*R \not\leq I$ so
		$\{0\},\{1\},\{2\}\in \nabla$. So $\nabla$ is three isolated vertices.
	\end{enumerate}
\end{ex}

For tensors modeled as arrays one can illustrate the simplicial complex
$\nabla(S;U)$ on top of the array as done in Figure~\ref{fig:singular-2} for
$2$-tensors and Figure~\ref{fig:sing-val3} for $3$-tensors.  To explain the
2-tensor example, take $\bra{t}:K^2\times K^2\bmto K$ with
$\bra{t}(v_1,v_2),(v'_1,v'_2)\rangle=(v_1,v_2)\left[\begin{smallmatrix} a & b \\
c & d\end{smallmatrix}\right](v'_1,v'_2)^{\top}$. Fix  $U=[ K(1,0), K(1,0) ]$,
so that $\nabla(Kt;U)$ depends on the values of $a,b,c,d$ in the way shown in
Figure~\ref{fig:singular-2}. 
\begin{figure}[!htbp]
	\begin{tikzpicture}
        \tikzstyle{point}=[circle,thick,draw=black,fill=black,inner sep=0pt,
        minimum width=4pt,minimum height=4pt]
        \tikzstyle{point2}=[inner sep=0pt,minimum width=0pt,minimum height=0pt]

		\pgfmathsetmacro{\xx}{0.5}
		\pgfmathsetmacro{\yy}{-0.5}

		\node (A) at (0,0) {\begin{tikzpicture}
			\node at (0,0) {$(0)$};
			\node at (0,1.5) {\begin{tikzpicture}
				\draw[black,thick, fill=gray!50] 
				  (0*\xx,0*\yy) -- (0*\xx,3*\yy) -- (3*\xx,3*\yy) -- (3*\xx,0*\yy)--cycle;
				\fill[pattern=dots] 
				  (1*\xx,1*\yy) -- (1*\xx,3*\yy) -- (3*\xx,3*\yy) -- (3*\xx,1*\yy)--cycle;
				\node (0) [point] at (0.5*\xx,2.5*\yy) {};
				\node (01) [point2] at (0.5*\xx,0.5*\yy) {};
				\node (1) [point] at (2.5*\xx,0.5*\yy) {};
				\draw[thick] (0) -- (01) -- (1);
				\node at (-0.75*\xx,0.75*\yy) {{\tiny $U_1$}};
				\node at (0.75*\xx,-0.75*\yy) {{\tiny $U_0$}};
			\end{tikzpicture}};
		\end{tikzpicture}};

		\node (B) at (3,0) {\begin{tikzpicture}
			\node at (0,0) {$(x_0x_1)$};
			\node at (0,1.5) {\begin{tikzpicture}
				\draw[black,thick, fill=gray!50] 
				  (0*\xx,0*\yy) -- (0*\xx,3*\yy) -- (3*\xx,3*\yy) -- (3*\xx,0*\yy)--cycle;
				\draw[thick,fill=white] (0*\xx,0*\yy) rectangle (\xx,\yy);
				\fill[pattern=dots] 
				  (1*\xx,1*\yy) -- (1*\xx,3*\yy) -- (3*\xx,3*\yy) -- (3*\xx,1*\yy)--cycle;
				\node (0) [point] at (0.5*\xx,2.5*\yy) {};
				%\node (01) [point2] at (0.5*\xx,0.5*\yy) {};
				\node (1) [point] at (2.5*\xx,0.5*\yy) {};
				\node at (-0.75*\xx,0.75*\yy) {{\tiny $U_1$}};
				\node at (0.75*\xx,-0.75*\yy) {{\tiny $U_0$}};
			\end{tikzpicture}};
		\end{tikzpicture}};

		\node (C) at (6,0) {\begin{tikzpicture}
			\node at (0,0) {$(x_1)$};
			\node at (0,1.5) {\begin{tikzpicture}
				\draw[black,thick, fill=gray!50] 
				  (0*\xx,0*\yy) -- (0*\xx,3*\yy) -- (3*\xx,3*\yy) -- (3*\xx,0*\yy)--cycle;
				\draw[thick,fill=white] (0*\xx,0*\yy) rectangle (1*\xx,3*\yy);
				\fill[color=white, pattern=dots] 
				  (1*\xx,1*\yy) -- (1*\xx,3*\yy) -- (3*\xx,3*\yy) -- (3*\xx,1*\yy)--cycle;
				%\node (0) [point] at (0.5*\xx,2.5*\yy) {};
				%\node (01) [point2] at (0.5*\xx,0.5*\yy) {};
				\node (1) [point] at (2.5*\xx,0.5*\yy) {};
				\node at (-0.75*\xx,0.75*\yy) {{\tiny $U_1$}};
				\node at (0.75*\xx,-0.75*\yy) {{\tiny $U_0$}};
			\end{tikzpicture}};
		\end{tikzpicture}};

		\node (D) at (9,0) {\begin{tikzpicture}
			\node at (0,0) {$(x_0,x_1)$};
			\node at (0,1.5) {\begin{tikzpicture}
				\draw[black,thick, fill=white] 
				  (0*\xx,0*\yy) -- (0*\xx,3*\yy) -- (3*\xx,3*\yy) -- (3*\xx,0*\yy)--cycle;
				\fill[white, pattern=dots] 
				  (1*\xx,1*\yy) -- (1*\xx,3*\yy) -- (3*\xx,3*\yy) -- (3*\xx,1*\yy)--cycle;
				%\node (0) [point] at (0.5*\xx,2.5*\yy) {};
				%\node (01) [point2] at (0.5*\xx,0.5*\yy) {};
				%\node (1) [point] at (2.5*\xx,0.5*\yy) {};
				\node at (-0.75*\xx,0.75*\yy) {{\tiny $U_1$}};
				\node at (0.75*\xx,-0.75*\yy) {{\tiny $U_0$}};
			\end{tikzpicture}};
		\end{tikzpicture}};

	\end{tikzpicture}
	\caption{Some simplicial complexes $\nabla(Kt;U)$, represented
	graphically on top of a $2$-tensor.
	Shaded regions are known to contain
	a nonzero, white regions are $0$, and the dotted region is arbitrary.
	Beneath we include the associated Stanley--Reisner ideal which agrees with
	$\Id{Kt}{\Omega_{U}}$.  The missing example of $(x_0)$ is the transpose 
	of matrix for $(x_1)$.}
	\label{fig:singular-2}
\end{figure}
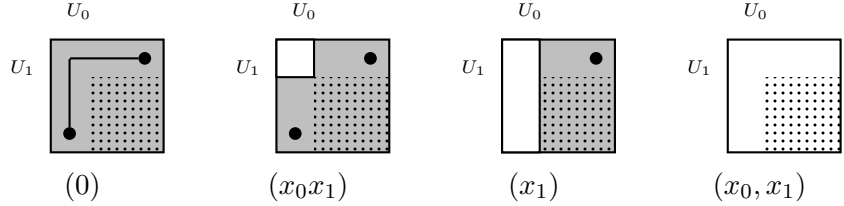
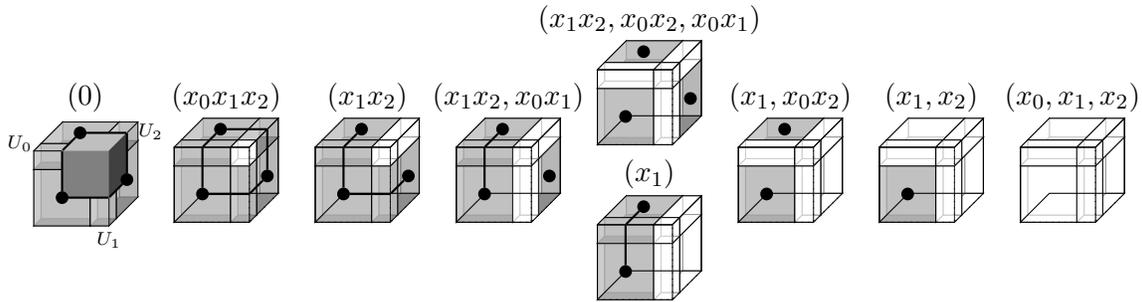
\begin{figure}[!htbp]
	
    \begin{tikzpicture}[scale=0.75]
        \tikzstyle{point}=[circle,thick,draw=black,fill=black,inner sep=0pt,
        minimum width=4pt,minimum height=4pt]
        \tikzstyle{point2}=[inner sep=0pt,minimum width=0pt,minimum height=0pt]

		\pgfmathsetmacro{\xx}{0.5}
		\pgfmathsetmacro{\yy}{.5}
        \pgfmathsetmacro{\zz}{0.5}
        
        \pgfmathsetmacro{\dd}{0.33}
        \pgfmathsetmacro{\ddd}{1.33}
		\pgfmathsetmacro{\xxx}{0.75}
		\pgfmathsetmacro{\yyy}{0.75}
		\pgfmathsetmacro{\zzz}{0.75}
		\pgfmathsetmacro{\xscale}{2.5}
		\pgfmathsetmacro{\yscale}{-2.75}
		\pgfmathsetmacro{\boxscale}{0.5}
		
%		\node (T) at (-9,0) {\begin{tikzpicture}[scale=0.94]
		
			\node at (7*\xscale,0*\yscale+1.30) {$(x_0,x_1,x_2)$};
			\node (x2-x1-x0) at (7*\xscale,0*\yscale) {\begin{tikzpicture}
                %000=0
                \pic at (0*\xxx,0*\yyy,0*\zzz) {ccube={1*\xxx/1*\yyy/1*\zzz/gray}};
                %010=1
                \pic at (0*\xxx,1*\yyy,0*\zzz) {ccube={1*\xxx/\dd*\yyy/1*\zzz/white}};
                %001=2
                \pic at (0*\xxx,0*\yyy,1*\zzz) {ccube={1*\xxx/1*\yyy/\dd*\zzz/white}};
                %011=3
                \pic at (0*\xxx,1*\yyy,1*\zzz) {ccube={1*\xxx/\dd*\yyy/\dd*\zzz/white}};
                %100=4
                \pic at (1*\xxx,0*\yyy,0*\zzz) {ccube={\dd*\xxx/1*\yyy/1*\zzz/white}};
                %110=5
                \pic at (1*\xxx,1*\yyy,0*\zzz) {ccube={\dd*\xxx/\dd*\yyy/1*\zzz/white}};
                %101=6
                \pic at (1*\xxx,0*\yyy,1*\zzz) {ccube={\dd*\xxx/1*\yyy/\dd*\zzz/white}};
                %111=7
                \pic at (1*\xxx,1*\yyy,1*\zzz) {ccube={\dd*\xxx/\dd*\yyy/\dd*\zzz/white}};

                \pic at (0,\ddd*\yyy,0) {linecube={\ddd*\xxx/-\ddd*\yyy/\ddd*\zzz}};	
			\end{tikzpicture}}; % (x2,x1,x0)
		
			\node at (6*\xscale,0*\yscale+1.30) {$(x_1,x_2)$};
			\node (x2-x1) at (6*\xscale,0*\yscale) {\begin{tikzpicture}
                %000=0
                \pic at (0*\xxx,0*\yyy,0*\zzz) {ccube={1*\xxx/1*\yyy/1*\zzz/black}};
                %010=1
                \pic at (0*\xxx,1*\yyy,0*\zzz) {ccube={1*\xxx/\dd*\yyy/1*\zzz/white}};
                %001=2
                \pic at (0*\xxx,0*\yyy,1*\zzz) {ccube={1*\xxx/1*\yyy/\dd*\zzz/black}};
                %011=3
                \pic at (0*\xxx,1*\yyy,1*\zzz) {ccube={1*\xxx/\dd*\yyy/\dd*\zzz/white}};
                %100=4
                \pic at (1*\xxx,0*\yyy,0*\zzz) {ccube={\dd*\xxx/1*\yyy/1*\zzz/white}};
                %110=5
                \pic at (1*\xxx,1*\yyy,0*\zzz) {ccube={\dd*\xxx/\dd*\yyy/1*\zzz/white}};
                %101=6
                \pic at (1*\xxx,0*\yyy,1*\zzz) {ccube={\dd*\xxx/1*\yyy/\dd*\zzz/white}};
                %111=7
                \pic at (1*\xxx,1*\yyy,1*\zzz) {ccube={\dd*\xxx/\dd*\yyy/\dd*\zzz/white}};

                \pic at (0,\ddd*\yyy,0) {linecube={\ddd*\xxx/-\ddd*\yyy/\ddd*\zzz}};	
                \node (21) [point] at (0.5*\xxx,0.5*\yyy,\ddd*\zzz) {};
                %\node (1) [point2] at (0.5*\xxx,\ddd*\yyy,\ddd*\zzz) {};
                %\node (10) [point] at (0.5*\xxx,\ddd*\yyy,0.5*\zzz) {};
                %\node (0) [point2] at (\ddd*\xxx,\ddd*\yyy,0.5*\zzz) {};
                %\node (20) [point] at (\ddd*\xxx,0.5*\yyy,0.5*\zzz) {};
                %\node (2) [point2] at (\ddd*\xxx,0.5*\yyy,\ddd*\zzz) {};
			\end{tikzpicture}}; % (x2,x1)

			\node at (5*\xscale,0*\yscale+1.30) {$(x_1,x_0x_2)$};
			\node (x2-x1x0) at (5*\xscale,0*\yscale) {\begin{tikzpicture}	
                %000=0
                \pic at (0*\xxx,0*\yyy,0*\zzz) {ccube={1*\xxx/1*\yyy/1*\zzz/gray}};
                %010=1
                \pic at (0*\xxx,1*\yyy,0*\zzz) {ccube={1*\xxx/\dd*\yyy/1*\zzz/black}};
                %001=2
                \pic at (0*\xxx,0*\yyy,1*\zzz) {ccube={1*\xxx/1*\yyy/\dd*\zzz/black}};
                %011=3
                \pic at (0*\xxx,1*\yyy,1*\zzz) {ccube={1*\xxx/\dd*\yyy/\dd*\zzz/white}};
                %100=4
                \pic at (1*\xxx,0*\yyy,0*\zzz) {ccube={\dd*\xxx/1*\yyy/1*\zzz/white}};
                %110=5
                \pic at (1*\xxx,1*\yyy,0*\zzz) {ccube={\dd*\xxx/\dd*\yyy/1*\zzz/white}};
                %101=6
                \pic at (1*\xxx,0*\yyy,1*\zzz) {ccube={\dd*\xxx/1*\yyy/\dd*\zzz/white}};
                %111=7
                \pic at (1*\xxx,1*\yyy,1*\zzz) {ccube={\dd*\xxx/\dd*\yyy/\dd*\zzz/white}};

                \pic at (0,\ddd*\yyy,0) {linecube={\ddd*\xxx/-\ddd*\yyy/\ddd*\zzz}};	

                \node (21) [point] at (0.5*\xxx,0.5*\yyy,\ddd*\zzz) {};
                %\node (1) [point2] at (0.5*\xxx,\ddd*\yyy,\ddd*\zzz) {};
                \node (10) [point] at (0.5*\xxx,\ddd*\yyy,0.5*\zzz) {};
                %\node (0) [point2] at (\ddd*\xxx,\ddd*\yyy,0.5*\zzz) {};
                %\node (20) [point] at (\ddd*\xxx,0.5*\yyy,0.5*\zzz) {};
                %\node (2) [point2] at (\ddd*\xxx,0.5*\yyy,\ddd*\zzz) {};
			\end{tikzpicture}}; % (x1-x2x0)

			\node at (4*\xscale,0.5*\yscale+1.30) {$(x_1)$};
			\node (x2) at (4*\xscale,0.5*\yscale) {\begin{tikzpicture}
                %000=0
                \pic at (0*\xxx,0*\yyy,0*\zzz) {ccube={1*\xxx/1*\yyy/1*\zzz/gray}};
                %010=1
                \pic at (0*\xxx,1*\yyy,0*\zzz) {ccube={1*\xxx/\dd*\yyy/1*\zzz/black}};
                %001=2
                \pic at (0*\xxx,0*\yyy,1*\zzz) {ccube={1*\xxx/1*\yyy/\dd*\zzz/black}};
                %011=3
                \pic at (0*\xxx,1*\yyy,1*\zzz) {ccube={1*\xxx/\dd*\yyy/\dd*\zzz/black}};
                %100=4
                \pic at (1*\xxx,0*\yyy,0*\zzz) {ccube={\dd*\xxx/1*\yyy/1*\zzz/white}};
                %110=5
                \pic at (1*\xxx,1*\yyy,0*\zzz) {ccube={\dd*\xxx/\dd*\yyy/1*\zzz/white}};
                %101=6
                \pic at (1*\xxx,0*\yyy,1*\zzz) {ccube={\dd*\xxx/1*\yyy/\dd*\zzz/white}};
                %111=7
                \pic at (1*\xxx,1*\yyy,1*\zzz) {ccube={\dd*\xxx/\dd*\yyy/\dd*\zzz/white}};

                \pic at (0,\ddd*\yyy,0) {linecube={\ddd*\xxx/-\ddd*\yyy/\ddd*\zzz}};	

                \node (21) [point] at (0.5*\xxx,0.5*\yyy,\ddd*\zzz) {};
                \node (1) [point2] at (0.5*\xxx,\ddd*\yyy,\ddd*\zzz) {};
                \node (10) [point] at (0.5*\xxx,\ddd*\yyy,0.5*\zzz) {};
                %\node (0) [point2] at (\ddd*\xxx,\ddd*\yyy,0.5*\zzz) {};
                %\node (20) [point] at (\ddd*\xxx,0.5*\yyy,0.5*\zzz) {};
                %\node (2) [point2] at (\ddd*\xxx,0.5*\yyy,\ddd*\zzz) {};

                \draw[thick] (21) -- (1) -- (10);
			\end{tikzpicture}}; % (x1)
		
			\node at (4*\xscale,-0.5*\yscale+1.30) {$(x_1x_2,x_0x_2,x_0x_1)$};
			\node (x2x1-x2x0-x1x0) at (4*\xscale,-0.5*\yscale) {\begin{tikzpicture}
                %000=0
                \pic at (0*\xxx,0*\yyy,0*\zzz) {ccube={1*\xxx/1*\yyy/1*\zzz/gray}};
                %010=1
                \pic at (0*\xxx,1*\yyy,0*\zzz) {ccube={1*\xxx/\dd*\yyy/1*\zzz/black}};
                %001=2
                \pic at (0*\xxx,0*\yyy,1*\zzz) {ccube={1*\xxx/1*\yyy/\dd*\zzz/black}};
                %011=3
                \pic at (0*\xxx,1*\yyy,1*\zzz) {ccube={1*\xxx/\dd*\yyy/\dd*\zzz/white}};
                %100=4
                \pic at (1*\xxx,0*\yyy,0*\zzz) {ccube={\dd*\xxx/1*\yyy/1*\zzz/black}};
                %110=5
                \pic at (1*\xxx,1*\yyy,0*\zzz) {ccube={\dd*\xxx/\dd*\yyy/1*\zzz/white}};
                %101=6
                \pic at (1*\xxx,0*\yyy,1*\zzz) {ccube={\dd*\xxx/1*\yyy/\dd*\zzz/white}};
                %111=7
                \pic at (1*\xxx,1*\yyy,1*\zzz) {ccube={\dd*\xxx/\dd*\yyy/\dd*\zzz/white}};

                \pic at (0,\ddd*\yyy,0) {linecube={\ddd*\xxx/-\ddd*\yyy/\ddd*\zzz}};	

                \node (21) [point] at (0.5*\xxx,0.5*\yyy,\ddd*\zzz) {};
                \node (1) [point2] at (0.5*\xxx,\ddd*\yyy,\ddd*\zzz) {};
                \node (10) [point] at (0.5*\xxx,\ddd*\yyy,0.5*\zzz) {};
                \node (0) [point2] at (\ddd*\xxx,\ddd*\yyy,0.5*\zzz) {};
                \node (20) [point] at (\ddd*\xxx,0.5*\yyy,0.5*\zzz) {};
                \node (2) [point2] at (\ddd*\xxx,0.5*\yyy,\ddd*\zzz) {};

			\end{tikzpicture}}; % (x2x1,x2x0,x1x0)

			\node at (3*\xscale,0*\yscale+1.30) {$(x_1x_2,x_0x_1)$};
			\node (x2x1-x1x0) at (3*\xscale,0*\yscale) {\begin{tikzpicture}
                %000=0
                \pic at (0*\xxx,0*\yyy,0*\zzz) {ccube={1*\xxx/1*\yyy/1*\zzz/gray}};
                %010=1
                \pic at (0*\xxx,1*\yyy,0*\zzz) {ccube={1*\xxx/\dd*\yyy/1*\zzz/black}};
                %001=2
                \pic at (0*\xxx,0*\yyy,1*\zzz) {ccube={1*\xxx/1*\yyy/\dd*\zzz/black}};
                %011=3
                \pic at (0*\xxx,1*\yyy,1*\zzz) {ccube={1*\xxx/\dd*\yyy/\dd*\zzz/black}};
                %100=4
                \pic at (1*\xxx,0*\yyy,0*\zzz) {ccube={\dd*\xxx/1*\yyy/1*\zzz/black}};
                %110=5
                \pic at (1*\xxx,1*\yyy,0*\zzz) {ccube={\dd*\xxx/\dd*\yyy/1*\zzz/white}};
                %101=6
                \pic at (1*\xxx,0*\yyy,1*\zzz) {ccube={\dd*\xxx/1*\yyy/\dd*\zzz/white}};
                %111=7
                \pic at (1*\xxx,1*\yyy,1*\zzz) {ccube={\dd*\xxx/\dd*\yyy/\dd*\zzz/white}};

                \pic at (0,\ddd*\yyy,0) {linecube={\ddd*\xxx/-\ddd*\yyy/\ddd*\zzz}};	

                \node (21) [point] at (0.5*\xxx,0.5*\yyy,\ddd*\zzz) {};
                \node (1) [point2] at (0.5*\xxx,\ddd*\yyy,\ddd*\zzz) {};
                \node (10) [point] at (0.5*\xxx,\ddd*\yyy,0.5*\zzz) {};
                \node (0) [point2] at (\ddd*\xxx,\ddd*\yyy,0.5*\zzz) {};
                \node (20) [point] at (\ddd*\xxx,0.5*\yyy,0.5*\zzz) {};
                \node (2) [point2] at (\ddd*\xxx,0.5*\yyy,\ddd*\zzz) {};

                \draw[thick] (10) -- (1) -- (21);
			\end{tikzpicture}}; % (x2x1,x1x0)

			\node at (2*\xscale,0*\yscale+1.30) {$(x_1x_2)$};
			\node (x2x1) at (2*\xscale,0*\yscale) {\begin{tikzpicture}
                %000=0
                \pic at (0*\xxx,0*\yyy,0*\zzz) {ccube={1*\xxx/1*\yyy/1*\zzz/gray}};
                %010=1
                \pic at (0*\xxx,1*\yyy,0*\zzz) {ccube={1*\xxx/\dd*\yyy/1*\zzz/black}};
                %001=2
                \pic at (0*\xxx,0*\yyy,1*\zzz) {ccube={1*\xxx/1*\yyy/\dd*\zzz/black}};
                %011=3
                \pic at (0*\xxx,1*\yyy,1*\zzz) {ccube={1*\xxx/\dd*\yyy/\dd*\zzz/black}};
                %100=4
                \pic at (1*\xxx,0*\yyy,0*\zzz) {ccube={\dd*\xxx/1*\yyy/1*\zzz/black}};
                %110=5
                \pic at (1*\xxx,1*\yyy,0*\zzz) {ccube={\dd*\xxx/\dd*\yyy/1*\zzz/white}};
                %101=6
                \pic at (1*\xxx,0*\yyy,1*\zzz) {ccube={\dd*\xxx/1*\yyy/\dd*\zzz/black}};
                %111=7
                \pic at (1*\xxx,1*\yyy,1*\zzz) {ccube={\dd*\xxx/\dd*\yyy/\dd*\zzz/white}};

                \pic at (0,\ddd*\yyy,0) {linecube={\ddd*\xxx/-\ddd*\yyy/\ddd*\zzz}};	

                \node (21) [point] at (0.5*\xxx,0.5*\yyy,\ddd*\zzz) {};
                \node (1) [point2] at (0.5*\xxx,\ddd*\yyy,\ddd*\zzz) {};
                \node (10) [point] at (0.5*\xxx,\ddd*\yyy,0.5*\zzz) {};
                \node (0) [point2] at (\ddd*\xxx,\ddd*\yyy,0.5*\zzz) {};
                \node (20) [point] at (\ddd*\xxx,0.5*\yyy,0.5*\zzz) {};
                \node (2) [point2] at (\ddd*\xxx,0.5*\yyy,\ddd*\zzz) {};

                \draw[thick] (10) -- (1) -- (21) -- (2) -- (20);
            \end{tikzpicture}}; % (x2x1)
		
			\node at (1*\xscale,0*\yscale+1.30) {$(x_0x_1x_2)$};
            \node (x2x1x0) at (1*\xscale,0*\yscale) {\begin{tikzpicture}
                %000=0
                \pic at (0*\xxx,0*\yyy,0*\zzz) {ccube={1*\xxx/1*\yyy/1*\zzz/gray}};
                %010=1
                \pic at (0*\xxx,1*\yyy,0*\zzz) {ccube={1*\xxx/\dd*\yyy/1*\zzz/black}};
                %001=2
                \pic at (0*\xxx,0*\yyy,1*\zzz) {ccube={1*\xxx/1*\yyy/\dd*\zzz/black}};
                %011=3
                \pic at (0*\xxx,1*\yyy,1*\zzz) {ccube={1*\xxx/\dd*\yyy/\dd*\zzz/black}};
                %100=4
                \pic at (1*\xxx,0*\yyy,0*\zzz) {ccube={\dd*\xxx/1*\yyy/1*\zzz/black}};
                %110=5
                \pic at (1*\xxx,1*\yyy,0*\zzz) {ccube={\dd*\xxx/\dd*\yyy/1*\zzz/black}};
                %101=6
                \pic at (1*\xxx,0*\yyy,1*\zzz) {ccube={\dd*\xxx/1*\yyy/\dd*\zzz/black}};
                %111=7
                \pic at (1*\xxx,1*\yyy,1*\zzz) {ccube={\dd*\xxx/\dd*\yyy/\dd*\zzz/white}};

                \pic at (0,\ddd*\yyy,0) {linecube={\ddd*\xxx/-\ddd*\yyy/\ddd*\zzz}};

                \node (21) [point] at (0.5*\xxx,0.5*\yyy,\ddd*\zzz) {};
                \node (1) [point2] at (0.5*\xxx,\ddd*\yyy,\ddd*\zzz) {};
                \node (10) [point] at (0.5*\xxx,\ddd*\yyy,0.5*\zzz) {};
                \node (0) [point2] at (\ddd*\xxx,\ddd*\yyy,0.5*\zzz) {};
                \node (20) [point] at (\ddd*\xxx,0.5*\yyy,0.5*\zzz) {};
                \node (2) [point2] at (\ddd*\xxx,0.5*\yyy,\ddd*\zzz) {};

                \draw[thick] (10) -- (1) -- (21) -- (2) -- (20) -- (0) -- (10);
			\end{tikzpicture}}; % (x2x1x0)

		\node at (0*\xscale,0*\yscale+1.3) {$(0)$};
		\node (0) at (0*\xscale,0.1*\yscale) {\begin{tikzpicture}
			%000=0
			\pic at (0*\xxx,0*\yyy,0*\zzz) {ccube={1*\xxx/1*\yyy/1*\zzz/gray}};
			%010=1
			\pic at (0*\xxx,1*\yyy,0*\zzz) {ccube={1*\xxx/\dd*\yyy/1*\zzz/black}};
			%001=2
			\pic at (0*\xxx,0*\yyy,1*\zzz) {ccube={1*\xxx/1*\yyy/\dd*\zzz/black}};
			%011=3
			\pic at (0*\xxx,1*\yyy,1*\zzz) {ccube={1*\xxx/\dd*\yyy/\dd*\zzz/black}};
			%100=4
			\pic at (1*\xxx,0*\yyy,0*\zzz) {ccube={\dd*\xxx/1*\yyy/1*\zzz/black}};
			%110=5
			\pic at (1*\xxx,1*\yyy,0*\zzz) {ccube={\dd*\xxx/\dd*\yyy/1*\zzz/black}};
			%101=6
			\pic at (1*\xxx,0*\yyy,1*\zzz) {ccube={\dd*\xxx/1*\yyy/\dd*\zzz/black}};
			%111=7
			\pic at (1*\xxx,1*\yyy,1*\zzz) {ccube={\dd*\xxx/\dd*\yyy/\dd*\zzz/black}};

			%\pic at (0,\ddd*\yyy,0) {linecube={\ddd*\xxx/-\ddd*\yyy/\ddd*\zzz}};	

			\fill[black!50] (0.5*\xxx,0.5*\yyy,\ddd*\zzz) -- (0.5*\xxx,\ddd*\yyy,\ddd*\zzz)
				-- (\ddd*\xxx,\ddd*\yyy,\ddd*\zzz) -- (\ddd*\xxx,0.5*\yyy,\ddd*\zzz) -- cycle;
			\fill[black!25] (0.5*\xxx,\ddd*\yyy,\ddd*\zzz) -- (\ddd*\xxx,\ddd*\yyy,\ddd*\zzz) 
				-- (\ddd*\xxx,\ddd*\yyy,0.5*\zzz) -- (0.5*\xxx,\ddd*\yyy,0.5*\zzz) -- cycle;
			\fill[black!75]  (\ddd*\xxx,0.5*\yyy,\ddd*\zzz) -- (\ddd*\xxx,\ddd*\yyy,\ddd*\zzz)
				-- (\ddd*\xxx,\ddd*\yyy,0.5*\zzz) -- (\ddd*\xxx,0.5*\yyy,0.5*\zzz) -- cycle;

			\node (21) [point] at (0.5*\xxx,0.5*\yyy,\ddd*\zzz) {};
			\node (1) [point2] at (0.5*\xxx,\ddd*\yyy,\ddd*\zzz) {};
			\node (10) [point] at (0.5*\xxx,\ddd*\yyy,0.5*\zzz) {};
			\node (0) [point2] at (\ddd*\xxx,\ddd*\yyy,0.5*\zzz) {};
			\node (20) [point] at (\ddd*\xxx,0.5*\yyy,0.5*\zzz) {};
			\node (2) [point2] at (\ddd*\xxx,0.5*\yyy,\ddd*\zzz) {};

			\draw[thick] (10) -- (1) -- (21) -- (2) -- (20) -- (0) -- (10);

			\node at (-0.25*\xxx,1.1*\ddd*\yyy,\ddd*\zzz) {{\tiny $U_0$}};
			\node at (\ddd*\xxx,-0.25*\yyy,\ddd*\zzz) {{\tiny $U_1$}};
			\node at (1.1*\ddd*\xxx,0.8*\ddd*\yyy,-0.25*\zzz) {{\tiny $U_2$}};
		\end{tikzpicture}}; % (0)
	%\end{tikzpicture}}; % (T)

	\end{tikzpicture}
	\caption{The simplicial complexes $\nabla(Kt;U)$ (up to permutation of
	 coordinates) drawn atop a $3$-tensor shown never to cross a singularity (in
	 white).  In the middle, there are two qualitatively different families of
	 singularities.  The type $(x_1)$ can be realized by zero-divisors such as
	 idempotents.  The type $(x_1x_2,x_0x_1,x_0x_2)$ occurs, for example, with
	 nilpotent zero divisors.}
	\label{fig:sing-val3}
\end{figure}

As is well-known that subrings $S$ of a ring $R$ are not in general submodules,
but proper left ideals $I\subset S$ are both nonunital subrings and left
$R$-submodules. This means ideals can be studied both within the nonunital ring
category as well as the $R$-module category.  This situation is repeated in
greater generality for tensors, and the singularity complex explains how.
Summarizing we have:

\begin{prop}
	Fix $\langle \cdot |:T\to V_0\oslash \cdots \oslash V_{\vav}$, and
	$\iota:(a\in \zrange{\vav})\to (\iota_a:U_a\hookrightarrow V_a)$. Let
	$A\subset\zrange{\vav}$ and $X_a:=U_a$, if $a\in A$, and $V_a$ otherwise. It
	follows that
	\begin{align*}
		A\not\in \nabla(T;U) \qquad \Longleftrightarrow \qquad 
			\langle\cdot |\iota_A :T\to X_0\oslash\cdots \oslash X_{\mathtt{v}}.
	\end{align*}
	Thus, $A\not\in \nabla(T;U)$ exactly when $\iota_A$ is a monomorphism in the $(X^e-1)$-homotopism 
	category, where $e\in  \{0,1\}^{\zrange{\vav}}$, $\supp e=A$.
\end{prop}

Therefore, subspaces $U$ for which $\nabla(t;U)=\{\emptyset\}$ implies that $U$
induces a subtensor in every transverse tensor category---this is the situation
of radicals. On the other hand, if $\nabla(t;U)$ is the full simplex, then $U$
does not induce a subtensor in \emph{any} transverse tensor category. In between
these two extremes, we recover many familiar concepts including subrings,
submodules, left and right ideals, and radicals.

\begin{remark}
	Fields are typically regarded as maximally distinct from singularities, but
	the example of $\mathbb{R}$ as a subalgebra of $\mathbb{C}$ clearly requires
	a singularity.  The confusion is understandable since singularities in the
	$0$-axes occur in the dual space and so appears non-singular when compared
	to singularity in other axes. Fully non-singular tensors are for example the
	Whitney tensor products $K^a\times K^b\bmto K^a\otimes K^b$.
\end{remark}

\subsection{Traits of subtensors}
As subtensors emerge across a mix of categories, when hunting for
decompositions into subtensors we can look for traits that appear as a result of
any intermediate categories thus breaking up the problem. So we need to
understand what traits appear within the transverse operators when they hit upon
a subtensor, that is, a singularity.  Indeed, in algorithms that locate such
decompositions the general scheme is to sample from transverse operators and
factorize the operators looking for splittings that have been shown to relate to
subtensors.  This has so far been a case-by-case approach but here we give a
complete characterization of the traits that signal subtensors.  We now recall
and slightly extend the definition of \eqref{def:op-res}. For $L\in \Comm{K}$,
\begin{align*}
	\Omega(U,V)(L) & = \left\{ \omega\in L\otimes \Omega ~\middle|~ \forall a,\; 
	\omega_a(L\otimes V_a)\leq L\otimes U_a \right\}.
\end{align*}
Since $K$ is a field, whenever $\langle t|U_A,V_{\bar{A}}\rangle\neq 0$
and   $L\in \Comm{K}$ is nonzero, we have
$\langle t_L|L\otimes U_A,L\otimes V_{\bar{A}}\rangle\neq 0$, so
$\nabla( Kt;U)=\nabla(L\otimes Kt;L\otimes U)$.  We prove Theorem~\ref{mainthm:singularity}
which requires that we demonstrate 
\begin{align*}
	\Id{S}{\Omega(U,V)}
		& =(X^e \mid {\rm supp}(e)\notin \nabla(S;U)).
\end{align*}

\begin{proof}[Proof of Theorem~\ref{mainthm:singularity}] Since this theorem
	applies to the case of fields $K$, the condition $\langle
	t|U_A,V_{\bar{A}}\rangle \leq U_0$ is interchanged with $U_0^{\bot}\langle
	t|U_A,V_{\bar{A}}\rangle=0$. So we may form a new tensor
	$\lla\cdot|:V_0^{\vee}\times V_1\times \cdots\times V_{\vav}\bmto K$ where
	$\lla t | \nu:V_0\to K, v_{\bar{0}}\rra := \nu\la t|v_{\bar{0}}\ra$.  We
	thereby take $A\subset \range{\vav}$ and $U_A\perp V_A$ if, and only if,
	$\lla t|U_A,V_{\bar{A}}\rra =0$.  So we assume without loss of generality
	that $V_0=K$.
	
	We begin by showing that $(X^e \mid \mathrm{supp}(e)\notin
	\nabla(S;U)) \subseteq \Id{S}{\Omega(U,V)}$.  
	Consider $A\not\in \nabla(S,U)$,
	and set $p=X^e$ where $e_a=1$ if $a\in A$ and $0$ otherwise. Then
	$\bra{t} U_A,V_{\bar{A}}\ra =0$. Take
	$L\in\Comm{K}$ and
	$\omega\in \Omega(U,V)(L)$.  For all $t\in S$,
	\begin{align*}
		\bra{t} p(\omega)\ket{v} & = 
		\bra{t} \omega_{A} v_{A},v_{\bar{A}}\rangle  \in
		\bra{t} U_{A}, V_{\bar{A}}\rangle=0.
	\end{align*}
	So $X^e\in \Id{S}{\Omega_{U}}$.

	Next we show that $\Id{S}{\Omega_{U}}\subset (X^e \mid \supp e\notin \nabla(S;U))$.
	Fix $p=\sum_e \lambda_e X^e\in \Id{S}{\Omega_{U}}$, and write 
	$p=p_{\nabla} + p_{\bar{\nabla}}$ such that 
	\begin{align*}
		p_{\nabla} &= \sum_{\mathrm{supp}(e)\in \nabla(S;U)} \lambda_e X^e, & 
		p_{\bar{\nabla}} &= \sum_{\mathrm{supp}(e)\notin \nabla(S;U)} \lambda_e X^e.
	\end{align*}	
	We show that $p_{\nabla} = 0$, so that $\lambda_e\neq 0$ implies that
    $\mathrm{supp}(e)\notin \nabla(S;U)$ and $p = p_{\bar{\nabla}}\in
    (X^e \mid \mathrm{supp}(e) \notin \nabla(S;U))$. If
    $\nabla(S;U) = \{\emptyset\}$, then we are done, so we assume that
    $\nabla(S;U)$ is not the empty complex.

	Fix a top cell $A\in \nabla(S;U)$. Then $\langle t|U_A,V_{\bar{A}}\rangle\neq 0$
	but for any $b\notin A$, $\langle t|U_A,U_b,V_{\overline{A\cup \{b\}}}\rangle = 0$.
	Hence, there is  $u\in U_A\times \prod_{b\notin
	A}(V_{b}-U_b)$ with $\langle t|u\rangle\neq 0$. Having fixed $u$,
	for every $a\in \range{\vav}$, we choose  an idempotent $\pi_a\in
	\mathrm{End}(V_a)$ with $\pi_a(V_a)=U_a$ such that $\pi_b u_b=0$
	if $b\notin A$. Thus for $a\in A$, we must have $\pi_a u_a=u_a$ because
	$u_a\in U_a$ (here we are specifically using the field assumption on $K$ to
	claim that $V_a=U_a\oplus X_a$ with $u_b\in X_a$.).

	Now, for all $\alpha\in K^{\range{\vav}}$ and $c\in \range{\vav}$, we have $\alpha\pi\in
	\Omega_{U}$ and
	\begin{align*}	
		(\alpha\pi)_c^{e_c}|u_c\rangle 
			& = \left\{\begin{array}{rl}
			\alpha_c^{e_c} |u_c\rangle & c\in A,\\
			\ket{u_c} & c\notin A \text{ and } e_c=0,\\
			0 & c\notin A \text{ and } e_c>0.
			\end{array}\right.
	\end{align*}
	Let $q_e:=\lambda_e X^e$ be a term of $p$ with $B:=\supp(e)$. Then
	\begin{align*}	
		\bra{t} q_e(\alpha\pi) \ket{u} & = \lambda_e \alpha^e
			\langle t| \pi_{A\cap B}u_{A\cap B},\; \pi_{B-A} u_{B-A},\; u_{\bar{B}}\rangle
			 = \left\{
			\begin{array}{cc}
				\lambda_e \alpha^e \langle t|u\rangle & B\subset A, \\
				0 & B-A\neq \emptyset.
			\end{array}
			\right.
	\end{align*}
	In what follows, we sum over all $e$ such that $\supp(e)=B$, abbreviated by $e:\supp(e)=B$. 
	Hence,
	\begin{align*}
		0 & = \langle t|p(\alpha\pi)\ket{u} 
		= \sum_{B\subset \range{\vav}}\sum_{e:\supp(e)=B}\lambda_e \alpha^e 
			\langle t| \pi_{A\cap B}u_{A\cap B}, \pi_{B-A} u_{B-A}, u_{\bar{B}}\rangle\\
		& = \left(\sum_{B\subset A}\sum_{e:\supp(e)=B}\lambda_e \alpha^e\right)\langle t|u\rangle.
	\end{align*}
	As $\langle t|u\rangle\neq 0$, $\sum_{B\subset
	A}\sum_{e:\supp(e)=B}\lambda_e \alpha^e=0$, for every $\alpha\in K^{\range{\vav}}$.
	In fact, since this argument  applies over all extensions    $L\in \Comm{K}$
	(by replacing $t$ with $t_L$), we see that $\sum_{B\subset
	A}\sum_{e:\supp(e)=B}\lambda_e X^e=0$ vanishes on an algebraic closure of
	$K$. Therefore $\sum_{B\subset A}\sum_{e:\supp(e)=B}\lambda_e X^e=0$.  So
	each term $\lambda_e X^e$ of $p$ for which $\supp(e)\subset A$ for some top
	cell $A\in \nabla(S;U)$ has $\lambda_e=0$. In particular
	$p_{\nabla}=0$.
\end{proof}

%==============================================================================

%==============================================================================
%\input{algorithms.tex}
\section{Data types, Algorithms, and Theorem~\ref{mainthm:construction}}
\label{sec:algorithms}

Now we show the effects of our results on the design of data types and
algorithms for tensors.  In \cite{TensorSpace}, the second and third author have
tested and implemented these design patterns which now comprise the multi-linear
algebra module of the computer algebra system in \textsf{Magma}~\cite{magma}.
Several projects have developed in parallel to this, which have added many
further complementary algorithms to \cite{TensorSpace}, and we wish to
especially thank P.A. Brooksbank and E.A. O'Brien for this added functionality
and testing.

First in Section~\ref{sec:types}, we detail the data types that serve to
facilitate computations with tensors and transverse operators.  The objective is
to demonstrate how to manage the many higher-level abstract manipulations of data
required in solving Tensor Isomorphism Problems (TIP) (Section~\ref{subsec:densor}) 
and Block-Decomposition Problems (BDP) (Section~\ref{subsec:restriction})
without losing the benefit of low-level optimizations.  This is where the
characterizations of homotopism categories of Section~\ref{sec:cats}
comes into play.

Second in Section~\ref{sec:sylver}, we detail the work-horse behind efficient
computing with our correspondence.  Given the work in
Section~\ref{sec:universality}, we have a specific set of linear equations to
solve which we demonstrate can be reduced to solving families of Sylvester
equations.  

Finally, we close with Section~\ref{sec:calc-corres} where we fill in the
missing algorithms to compute with our correspondence in general and prove
Theorem~\ref{mainthm:construction}.

\subsection{A feather-weight tensor type-system}\label{sec:types}

Operations on tensors divide into three levels: actions with frames, transverse
actions, and tensor arithmetic.  From our experiments and theory, we suggest
that tensor systems clearly articulate these levels.  We do so by introducing a
type for homotopism categories (as defined in Section~\ref{sec:cats}), a
type for tensor spaces, and the ability to import many existing tools to deal
with data in individual contexts.

\subsubsection{Background.} 

Most contexts supply a number of natural data types for representing tensorial
information.  These include dense multi-way arrays, sparse representations,
black-box and query based models; see~\cite{engineer}.  A closer inspection
spots numerous technical design choices each concerned with specific situations
that a general system will be incapable of adequately replacing.

We provide a framework to support access and control of tensors, grounded in
type theory.  In type theory notation, all data $x$ is labeled by its
\emph{type} $X$, written $x:X$, which guards that we make and access $x$
according to clearly articulated rules on $X$. \emph{Introduction} rules
describe how to create an instance $x$ of type $X$, sometimes called
constructors. \emph{Elimination} rules produce new data $y:Y$ from $x:X$, e.g.
through a function $f:X\to Y$. The remaining syntax of fractions separates the
data (and types) that precede a rule from its result, like input/output.  For example,
\begin{align}\label{eqn:example-type-notation}
	& \frac{x:X,\quad f:X\to Y}{f(x):Y}
	& & \frac{a:A,\quad f:\prod_{a:A} Y_a}{f(a):Y_a}
\end{align}

Mind that this notation does not declare a program to convert data, it is simply
the signature for asserting under what conditions the new data exist.  On the
left in~\eqref{eqn:example-type-notation}, we have represented the elimination
of both the type $X$, and the type $X\to Y$. The notation $f:X\to Y$
purposefully evokes a (mathematical) function, but because computers do not know
set theory, $X\to Y$ is just the name of a data type and $f:X\to Y$ is data of
that type.  The rule is what clarifies that we can eliminate the data $x$
together with $f$ to produce an output $f(x):Y$ and thus recover the experience
of a set-function.  The rule on the right in~\eqref{eqn:example-type-notation},
is a dependent-function, also called a heterogenous or H-maps.  We shall use these
for brevity here even though they are much more subtle data types.

\subsubsection{Tensors \& tensor spaces.}  

We assume that our context models a type ${\sf Abel}$ for abelian groups
together with a type $U\oslash V$ that equips the type $V\to U$ of linear maps
with the properties of an abelian group.  We introduce a tensor space type ${\sf
TenSpc}$ almost identically to its definition, which affords a uniform way to
interpret tensors as multilinear maps.
\begin{align}\label{eqn:tensor-space-sig}
	&
	\frac{
		\vav:\mathbb{N},\quad T:{\sf Abel},\quad V:\zrange{\vav}\to {\sf Abel}, 
		\quad \bra{\cdot}:V_0\oslash\cdots\oslash V_{\vav}\oslash T
	}{
		{\sf ts}(T,V,\bra{\cdot}):{\sf TenSpc}
	}
\end{align}
In~\eqref{eqn:tensor-space-sig}, ${\sf ts}$ just labels how we introduced the
type to distinguish it from, say, other introductions such as $K$-tensor spaces
where $K$ is also part of the input data.  

The main elimination rule is the most essential ingredient of our design, used
to evaluate tensors on potentially partial input. 
\begin{align*}
	\frac{
		{\sf ts}(T,V,\bra{\cdot}):{\sf TenSpc},
		\quad t:T,
		\quad \iota:A\hookrightarrow \range{\vav},
		\quad v_A:\prod_{a:A} V_{\iota(a)}, 
	}{
		\langle t|v_A\rangle:V_0\oslash \left(\oslash_{b:\range{\vav}-A}V_b\right)
	}
\end{align*}
True to our promise, tensors in our model remain as they were before: terms
$t:T$. (Note that we leave unspecified many essential, but mundane, elimination
rules such as ones to retrieve the data defining a tensor space.)

Now we trace the effect of our abstractions on the lowest level tensor
operations such as the evaluation $\la t|v\ra$.  In this model, both $v_a:V_a$
and $t:T$ can be represented by any data. The evaluation $\langle t|v\rangle$ is
assigned by the prescribed interpretation.  \emph{While our model treats all
such evaluations as interchangeable, every call is directly in the hands of the
backend with no type conversions or data wrappers imposed by our
design.}\footnote{One should be mindful that the programming language does not
insert such indirection as this is the most used operation of any tensor system.
Optimization here is warranted.}

\subsubsection{Homotopism Categories.} 

Homotopism categories capture combinatorial aspects of tensors like which axes are
covariant, contra-variant, and constant, which are described by integers $+1$,
$-1$, and $0$ respectively. In a more complex model, such
as~\cite{TensorSpace}, one adds further data like base rings and symmetry.  We
use a type {\sf TenCat} to capture the $(X^e-X^f)$-homotopism categories
which we introduce as follows.
\begin{align*}
	&
	\frac{
		\vav:\mathbb{N},\quad \sigma:\zrange{\vav}\to \{-1,0,1\}
	}{
		{\sf tc}(\vav,\sigma):{\sf TenCat}
	}
\end{align*}
Specifically, $\supp e:=\sigma^{-1}(1)$ and $\supp f:=\sigma^{-1}(-1)$ play
their usual roles of covariant and contravariant axes.

The objects of transverse tensor categories ${\sf tc}(\vav,\sigma):{\sf TenCat}$
are tensor spaces (but easily adapted to use tensors as objects).  Tensor spaces
are oblivious to tensor categories as they occur identically in every category.
Thus, the categories are distinguished by their morphisms.  These we form with a
type ${\sf Hmtp}$, short for \emph{homotopism}. Recall, that for $c:\sigma^{-1}(0)$, $U_c=V_c$;
otherwise no morphisms exist.   We eliminate transverse tensor categories when
we create homotopisms (and likewise with functors).
\begin{align}\label{eqn:tensor-cat-sig}
	\frac{
		\begin{array}{c}
			{\sf ts}(S,U,\lla \cdot|):{\sf TenSpc},\quad {\sf ts}(T,V,\bra{\cdot}):{\sf TenSpc},\\		
			{\sf tc}(\vav,\sigma):{\sf TenCat},\\
		\omega:\prod_{a:\sigma^{-1}(1)} \overline{V_a\oslash U_a},\quad
		\tau:\prod_{b:\sigma^{-1}(-1)} \overline{U_a\oslash V_a}
		\end{array}
	}{
		\textsf{ht}(\omega,\tau):{\sf Hmtp}
	}
\end{align}
Once more, we keep the encoding of operators because these arise as elements of
$U\oslash V$ which is encoded by the backend of our system.  What the rule
in~\eqref{eqn:tensor-cat-sig} says is that in order to describe a morphism we
need two tensor spaces (the domain $S$ and codomain $T$) a tensor category, and
transverse morphisms.  Because we are defining these homotopisms on the level of
tensor spaces, given $s:S$, there is an image $t:T$ under
$\textsf{ht}(\omega,\tau):{\sf Hmtp}$ and the restriction of $(\omega,\tau)$ to
$s\to t$ satisfies
\begin{align*}
	\langle s| \tau_B = \langle t| \omega_A.
\end{align*}
These types effect the conditions we require on an abstract level without
changing any of the underlying structures affording tensors and linear maps. So
for example $(\omega,\tau)$ may simply be a list of matrices.  Indeed often
tensor networks are assembled by placing matrices between tensors, as in
Figure~\ref{fig:annihilator-2}.  This however completely obscures the implied
interpretation and can lead to confused or incorrect application.  For instance,
using a transpose incorrectly or neglecting the effect of a change of basis. The
homotopism type exists, in part, to protect against such common errors but
without altering underlying structure which may be carefully optimized in
isolation.

\subsubsection{Transverse operators \& tensor networks.} One implication of our
homotopism type is the ability to perform and control lazy evaluation. While
some programming languages make eager/lazy evaluation part of the language
specification, we note that \emph{we do not have actual functions}. Instead,
transverse operators are whatever list of data our system supplies, e.g.~a list
of matrices. So the decisions about evaluation are now ours to make. For a lazy
evaluation, we simply make an elimination rule that adapts the interpretation
instead of applying the operators to the tensor space.
\begin{align*}
	\frac{
		{\sf ts}(T,V,\bra{\cdot}),\quad \omega:(a:\zrange{\vav})\to (\omega_a:V_a\oslash V_a)
	}{
		{\sf ts}(T,V,\bra{\cdot}\omega):{\sf TenSpc}
	}
\end{align*}
Such a model can be used repeatedly with multiple tensors (even of different
valence) to compose tensor networks.  For eager evaluation we do the opposite:
\begin{align*}
	\frac{
		{\sf ts}(T,V,\bra{\cdot}),\quad \omega:(a:\zrange{\vav})\to (\omega_a:V_a\oslash V_a)
	}{
		{\sf ts}(\omega T,V,\bra{\cdot}):{\sf TenSpc}
	}
\end{align*}
This is an intentionally elementary example, but a tour of
\cite{TensorSpace}*{Chapter~3} shows how to effect many more complex routines
such as restricting to subtensors, taking quotients, computing images, lifting
to free and projective tensor spaces and more.   Such constructions are well
outside the ergonomic use of multiway arrays and sparse tensors.  This level of
abstraction is better suited for such tasks as block decomposing tensors (BDP) and
deciding on isomorphism invariants (TIP).

\subsubsection{Moving data with functors of tensor categories.}\label{sec:moving-data}

Some of the most natural mathematical tasks with tensors concern re-ordering
data, like transposing, raising or lowering an index, and slicing out subsets of
the data. Moving data has a cost which scales non-linearly, and in contemporary
hardware and system designs these effects are pronounced. There are several
solutions to this problem which ought to be considered within a tensor type-system.  We will demonstrate how homotopism categories resulting from
Theorem~\ref{thm:group-intro} play the important role in solving this problem.
\smallskip

{\bf The problem.}  Suppose we have an index set $I=\prod_{a\in \zrange{\vav}}
\range{d_a}$ and tensor data $t:I\to K$. Next we have another index set
$J:=\prod_{b\in \zrange{\vav}}\range{d_b}$ with a function $f:J\to I$ and we
want to represent $t^f[j_*]:=t[f(j_*)]$.  For some applications it makes sense
simply to copy the necessary data into a new array. Other settings call for
\emph{indirection}, i.e. storing $f$ and the original data $t$ but accessing
$t^f[j_*]$ by calculating $i_*=f(j_*)$ and fetching $t[i_*]$---no data moves but
each call takes longer.  Tensors however carry so much data which is accessed in
large repeated sequences $\mathcal{S}$ (think of a matrix-vector product) that it becomes
profitable to batch the lookups and thus carry forward into lower memory large
chunks of contiguous data all being used within a specific computation.
So-called \emph{polytope} methods analyze the geometry of $f(\mathcal{S})\subset I$
looking for closed polytopes that (subject to affine transformation) can be
processed as contiguous chunks in memory~\cite{polytope}.  Since most tensor
operations are both commutative and associative, this re-arrangement is
harmless. 

An essential requirement to performance of an abstract tensor type is,
therefore, the ability to pass along to the backend not only single instructions
but sequences of instructions.  Already we have seen that tensor networks keep
such information accessible to the backend where polytope analysis can occur,
but there is still the need to capture data manipulations that are \emph{not}
representable within a tensor space.  There are many such manipulations and we
treat them all as functors, see \cite{TensorSpace}*{Chapter~4}. For
demonstration we focus on one family.

\smallskip

A \emph{Knuth-Liebler shuffle} is higher-valence variation of a
transpose~\cite{Wilson:division}. On a multiway array of numbers, it is possible to
interchange two axes almost without concern. However, the effect on the
interpretation maps of tensor spaces is considerable, and the effect on
homotopisms is even more delicate.  The concept of a shuffle takes a
multilinear map $\bra{t}:V_1\times \cdots\times V_{\vav}\bmto V_0$, a
permutation $\pi$ on $\zrange{\vav}$, and an abelian group $W$, and assembles a
new multilinear map.  To demonstrate let $\vav\geq 3$ and $\pi=(0,1)(2,3)$. Then
the multilinear map is framed and defined as 
\begin{align*}
	\langle \cdot |^{(\pi,W)}:(W\oslash V_0)\times V_3\times V_2\times V_4\times
		\cdots \times V_{\vav}\bmto (W\oslash V_1)  \\
	\langle t| \nu:V_0\to W,v_3,v_2,v_4,\ldots,v_{\vav}\rangle \ket{v_1}:=\nu\langle t|v_1,\ldots,v_{\vav}\rangle. 
\end{align*}
\emph{Observe that the terms in the permutation involving $0$ pass through a
duality}. This may cause confusion within calculations but is necessary to be
well-defined. On the level of the underlying data structure, e.g.~a multiway
array, none of this duality is apparent.  In fact, even if applied with delicacy
this transformation also affects homotopisms --- adding a
further layers to track and possibly creating hard-to-find errors.

Now we show the implementation of this functor and the options provided to pass
along vital information to our backend.  We apply functors to our category and
then to our objects, denoted below by $F_1$ and $F_2$ ($F_3$ is the functor application
on morphisms, not shown).  The highest level of
this shuffle modifies the terms of the {\sf TenCat} type. 
\begin{align*}
	F_1(\pi):& (\textsf{tc}(\vav, \sigma):{\sf TenCat})\to (\textsf{tc}(\vav, \sigma^{\pi}):{\sf TenCat}),
	\\
	\sigma^{\pi}(a) & = \left\{\begin{array}{ll}
		-\sigma(\pi(a))  & \pi(a)=0,\\
		-\sigma(\pi(0))  & \pi(0)=a,\\
		\sigma(\pi(a))   & \text{otherwise}.
	\end{array}\right.
\end{align*}
This captures the intuitive aspect of Knuth--Liebler shuffles: they permute the
indices, with a relatively friendly sign change for duality.  Next, we apply the
$F_2$ the terms of type $\textsf{TenSpc}$.
\begin{align*}
	F_2(\pi,W):& (\textsf{ts}(T,V,\bra{\cdot}):{\sf TenSpc})
		\to (\textsf{ts}(T^{(\pi,W)},V^{(\pi,W)},\bra{\cdot}^{(\pi,W)}):{\sf TenSpc})
	\\
	V^{(\pi,W)}_a &  = \left\{\begin{array}{ll}
		W\oslash V_{\pi(a)}  & \pi(a)=0,\\
		W\oslash V_{\pi(0)}  & \pi(0)=a,\\
		V_{\pi(a)}   & \text{otherwise}.
	\end{array}\right.
\end{align*}
While we could apply this shuffle in a lazy fashion, we have opted to illustrate
how to deliver the shuffle to the backend, denoted by $T^{(\pi,W)}$, and its
interpretation.  If polytope methods are used, then the given permutation can be
applied to the polytope compiler to reorder the current sequence of steps without
requiring one to recompile a new polytope decomposition.
\medskip

The clarity of the three levels of abstraction is now evident and relatively
direct to implement.  There are even many alternatives to consider based on
stronger type theories.  But our crucial point is that, because of the
Correspondence Theorem~\ref{thm:correspondence} and its implications such as
Theorem~\ref{thm:group-intro}, we can be confident that this model captures all
possible abstractions at these levels.

\subsection{Simultaneous Sylvester systems}\label{sec:sylver}

The varieties $\Op{S}{P}$ are specified by a system of polynomials in $\sum_a
d_a^2$-variables (see Remark~\ref{rem:make-scheme} \& Proposition~\ref{prop:Z}).
(As always, these are not the polynomials $P$.)  Even so, the formula defining
$\Op{S}{P}$ gives a natural way to construct a set of polynomials that defines
the operators in $\Op{S}{P}$ and that generating set has degree at most the
degree of a generating set given for $P$.  In particular, for a linear
polynomial ideal $P=(\Lambda X-\lambda)$, there is a set of Sylvester equations
(linear constraints) defining $\Op{S}{\Lambda X-\lambda}$. We state the
situation for $\vav=2$ because the general case follows similarly.

Suppose that $\la t|:K^a\times K^b\bmto K^c$ is defined using a list
$[M_1,\ldots,M_c]$ of $(a\times b)$-matrices where
\begin{align}\label{eqn:tensor-ex}
	\la t|v_1,v_2\ra = (v_1^{\top} M_1 v_2,\ldots, v_1^{\top}M_c v_2).
\end{align}
Then consider the equation $\lambda_1\langle t|\omega_1v_1,v_2\rangle
+\lambda_2\langle t|v_1,\omega_2v_2\rangle=0$ that one would solve to determine
$\OpSet[\{1,2\}]{t}{\lambda_1x_1 + \lambda_2x_2}{K}$. Letting $X_i$ be the
matrix representation of $\omega_i$ this translates into the following linear
system:
\begin{align}\label{eqn:Sylvester}
	(\forall k\in \range{c})(X_1^{\top} M_k+M_k X_2=0).
\end{align}
This leads to an $(abc)\times (a^2+b^2)$ matrix when solved directly, and by
row-reducing, the system is solved in $O((abc)(a^2+b^2)^2)$-time, or roughly
$O(d^7)$-time when $a,b,c\in O(d)$.\footnote{Variations of our analysis
considering faster linear algebra can be considered as well but are more detail
than necessary for this section.} For context note that in the dense model, the
input size $n$ is $O(d^3)$ and accepting some randomization the solutions can be
found in quadratic $O(n^2)$-time.  Generalizing this approach to determine a
basis for the derivation algebra of a tensor of valence $\vav$ amounts to a
$O(\vav d^{\vav + 5})$-time algorithm when each $d_a\in O(d)$.  It is worth
mention that these computations are one time costs that can dramatically reduce
the dimension of the work space; more in Section~\ref{sec:Small-densor}. Without
this we are left to continue working in $V_0\oslash\cdots\oslash V_{\vav}$ which
is $d_0\cdots d_{\vav}\in O(d^{1+\vav})$ dimensional. In our experience that on
hard problems like TIP and BDP it is almost always worth the initial cost.

A dual version of the Sylvester system in~\eqref{eqn:Sylvester} is one for
computing the densor or, more generally, $P$-closures of a particular tensor
$\Ten{P}{\OpSet{t}{P}{K}}$. Keeping with the same tensor as
in~\eqref{eqn:tensor-ex} and setting $d=x_0-x_1-x_2$, solving for
$\OpSet{t}{d}{K}$ amounts to solving 
\begin{align}\label{eqn:derivation-system}
	(\forall k\in \range{c})(X_1^{\top} M_k+M_k X_2=(M_1, \dots, M_c)\cdot (X_0)_k)
\end{align}
to determine a basis for derivations $(X_0, X_1, X_2)$, where $(M_1, \dots,
M_c)\cdot (X_0)_k$ is the dot product of the vector of matrices $M_i$ with the
$k$th column of $X_0$. However, solving for a basis for the densor requires a
role-reversal. For some finite set, e.g.~a Lie generating set,
$\mathcal{X}\subset \OpSet{t}{d}{K}$, we solve the same system
in~\eqref{eqn:derivation-system} running through all $(X_0, X_1,
X_2)\in\mathcal{X}$ to determine a basis for the tensors giving as lists of
matrices $[M_1,\dots, M_c]$. Therefore, constructing a basis for the densor
requires a factor $|\mathcal{X}|$ more time than constructing a basis for the
derivations. Since densors are constructed using the same system, computational
improvements to operators distributes to $P$-closures. 

The systems in~\eqref{eqn:Sylvester} and~\eqref{eqn:derivation-system} have
enormous structures and are known as Sylvester systems of equations. For
instance when $c= 1$, the system in~\eqref{eqn:Sylvester} can be solved in time
$O(d^3)$ when $c=1$ and when $c=2$---under a few modest assumptions of
nondegeneracy~\cite{BW:sloped}. For $c>2$, the solution to~\eqref{eqn:Sylvester}
presently runs in $O(d^6)$ time. Even without a further complexity breakthrough,
there is a great deal that can be achieved in profiling the problem.  Within
\cite{TensorSpace}, we have designed an algorithm to layout the required matrix
in a manner that minimizes the movement of repeated information and furthermore
interleaves them so that echelonization can occur in block form.  With this we
reduced the overhead for computing derivations and densors to now run at the
same speed of comparable linear algebra; we invite the reader to experiment
with \cite{TensorSpace} on their own data sets.  Still a better complexity should be
sought if possible.

\begin{quest}\label{quest:quicksylver}
	Is there an algorithm to solve \eqref{eqn:Sylvester} and \eqref{eqn:derivation-system}
	with complexity better than $O(d^6)$ for $3$-tensors?  What about the general valence case?
\end{quest}

\subsection{Calculating with the correspondence}\label{sec:calc-corres}
We now detail how to compute the terms in our Correspondence
Theorem~\ref{thm:correspondence}. We begin by demonstrating how we calculate the
annihilators of transverse operators of a tensor.  Section~\ref{sec:types}
details the input types we consider. 

The most important assumptions we need are placed on $K$, specifically the
ability to solve systems of linear equations in $K$.  We have in mind standard
Gaussian elimination type methods for fields $K$ and Hermite Normal Form for
integer rings.  In a few problems we shall also need the ability to compute a
Gr\"obner basis in a bounded number of variables (see \cite{CLO} for definitions
and discussions of Gr\"obner bases).  We emphasize that our use will strictly 
concern $\vav+1$ variables which makes it possible to bound the complexity of
the Gr\"obner basis computations.

\begin{thm}[Bradet--Faug\`ere--Salvy \cite{BFS}]\label{thm:Groebner}
    The complexity of calculating a Gr\"obner basis (by $F_5$) on inputs
    $(f_1,\ldots,f_m)$ in $K[x_1,\ldots,x_n]$ with maximum degree  $D$ is  
    \begin{align*}
        O\left( m\binom{n+D-1}{D}^{\varpi}\right)
    \end{align*}
    where $2\leq \varpi <3$ is the exponent of matrix multiplication.
\end{thm}

The details are illustrated using the example of
Figure~\ref{fig:ex-ann-nilpotent}.  Recall in that example
\begin{align*}
	M &= \begin{bmatrix}
		1 & 2 & 3 \\ 2 & 3 & 0 
	\end{bmatrix}, & X &= E_{12}\in\mathbb{M}_2(K), & Y &= E_{21}+E_{32}\in \mathbb{M}_3(K).
\end{align*}
We interpret $M$ as tensor $\bra{M}:\mathbb{R}^3\to \mathbb{R}^2$ and
$\omega_0:=X$ and $\omega_1:=Y$ is our transverse operator. Because both $X$ and
$Y$ are nilpotent, we can compute the annihilator of $M$ by computing a finite
number of expressions of the form $X^iMY^j$, call this new matrix $U(i,j)$. From
these, we define a new matrix $U$, whose rows are indexed by $\zrange{2}\times \zrange{3}$
(here truncated to $\zrange{2}\times\zrange{2}$ for space--the rest are 0).
Columns are indexed by the six entries of the matrices
$U(i,j)$. The resulting $(9\times 6)$-matrix $U$ is recorded in
Figure~\ref{fig:matrix-U}. We then write a basis for its cokernel as a matrix
$U^{\bot}$. Extracting the rows of $U^{\bot}$, in Figure~\ref{fig:matrix-U},
permits us to create the polynomials that generate the annihilator as seen in
Figure~\ref{fig:Ann-matrix}.  In total we have the following algorithm

\begin{figure}[!htbp]
	\centering
	\begin{subfigure}{\textwidth}
		\centering
		\begin{align*}
			U & =
			\begin{array}{c|c}
					(i, j) & X^i  M  Y^j\\
				\hline\hline 
				(0, 0) & [[1, 2, 3], [2, 3, 0]]\\
				(1, 0) & [[2, 3, 0], [0, 0, 0]]\\
				(2, 0) & [[0, 0, 0], [0, 0, 0]]\\
				(0, 1) & [[2, 3, 0], [3, 0, 0]]\\
				(1, 1) & [[3, 0, 0], [0, 0, 0]]\\
				(2, 1) & [[0, 0, 0], [0, 0, 0]]\\
				(0, 2) & [[3, 0, 0], [0, 0, 0]]\\
				(1, 2) & [[0, 0, 0], [0, 0, 0]]\\
				(2, 2) & [[0, 0, 0], [0, 0, 0]]\\
			\end{array}
			& 
			U^{\bot} 
			=
			\begin{bmatrix}
				0 & 0 & 0 & 0 & 1 & 0 & -1 & 0 & 0\\
				0 & 0 & 0 & 0 & 0 & 0 & 0 & 1 & 0\\
				0 & 0 & 1 & 0 & 0 & 0 & 0 & 0 & 0\\
				0 & 0 & 0 & 0 & 0 & 1 & 0 & 0 & 0\\
				0 & 0 & 0 & 0 & 0 & 0 & 0 & 0 & 1\\
			\end{bmatrix}
		\end{align*}
        \caption{We compute the matrices $X^iMY^j$ and record 
        them as the rows of a $9\times 6$ matrix $U$.}
		\label{fig:matrix-U}
	\end{subfigure}
	\begin{subfigure}{\textwidth}
		\centering
		\begin{align*}
			\begin{array}{|ccccccccc|c|}
				\hline
				1 & x & x^2 & y & xy & x^2 y & y^2 & xy^2 & x^2 y^2 &\\
				\hline\hline
				0 & 0 & 0 & 0 & 1 & 0 & -1 & 0 & 0 & xy-y^2\\
				0 & 0 & 0 & 0 & 0 & 0 & 0 & 1 & 0 & xy^2\\
				0 & 0 & 1 & 0 & 0 & 0 & 0 & 0 & 0 & x^2\\
				0 & 0 & 0 & 0 & 0 & 1 & 0 & 0 & 0 & x^2 y\\
				0 & 0 & 0 & 0 & 0 & 0 & 0 & 0 & 1 & x^2 y^2\\
				\hline
			\end{array}	   
		\end{align*}
        \caption{Using $U^{\bot}$ we exhibit generators for the annihilator ideal.}
		\label{fig:Ann-matrix}
	\end{subfigure}
    \caption{Computations for determining the annihilator in Figure~\ref{fig:ex-ann-nilpotent}.}
	\label{fig:example-A1}
\end{figure}

The annihilator $\Ann_{K[x,y]}^{(X, Y)}\left(M\right)=(xy-y^2, xy^2,x^2, x^2 y,
x^2 y^2)=(x^2,  xy-y^2,y^3)$. Observe that the last ideal is generated by a
Gr\"obner basis for the ideal.

\subsection{Proof of Theorem~\ref{mainthm:construction}}\label{sec:construction}

We assume $S$, $P$, and $\Delta$ are given as subsets.  These may stand in as
generators of much larger spaces, for instance a basis of $S$, generators of
$P$, or group generators of $\Delta$. We will assume that each $V_a\cong
K^{d_a}$, with fixed basis $\mathcal{X}_a$, and that tensors are given by 
data types as discussed in Section~\ref{sec:types}, see also \cite{TensorSpace}.

The modules $\Ten{P}{\Delta}$ are defined from known equations and are linear.
Indeed as seen in Section~\ref{sec:sylver}, improvements occur when $P$ is
generated by homogeneous linear ideals.  Likewise, Section~\ref{sec:sylver}
deals with solving for $\Op{S}{P}$.  This leaves us to compute $\Id{S}{\Delta}$
which we compute as an intersection of $\Ann_{K[X]}^{\omega}(t)$ for $t\in S$
and $\omega\in\Delta$~\cite{CLO}. To compute these intersections it suffices to
have Gr\"obner bases of each $\Ann_{K[X]}^{\omega}(t)$ \cite{CLO}*{p. 188}. So
our work concentrates on this problem.  In the examples of
Figures~\ref{fig:annihilator} and~\ref{fig:annihilator-2}, we see an example of
the algorithm implied.  

Let $e\in \prod_{a\in A} \zrange{d_a}$ and $v\in \prod_{a\in A}
\mathcal{X}_a$. Define a matrix $U$ indexed by $(e, v)$ such that 
\begin{align*}
	U_{e, v} := \omega_0^{e(0)}\langle t|\omega_1^{e(1)}v_1,\ldots, \omega_{\vav}^{e(\vav)}v_{\vav}\rangle.
\end{align*} 
Compute a basis $\{u_1,\ldots,u_n\}$ for the cokernel of $U$. Since the entries
of $u_j$ are indexed by $e\in \prod_{a\in A} \zrange{d_a}$, compute a
reduced Gr\"obner basis for $\Id{t}{\omega}=(\sum_e u_{j,e} X^e\mid 1\leq j\leq
n)$.  Primary decompositions of $\Id{t}{\Delta}$ are now computed by 
established routines; see \cite{CLO}*{Section~4.7}.

Finally, set $d=d_0\cdots d_{\vav}$, and recall that $\vav$ is fixed. Applying
Bradet--Faug\`ere--Salvy to our setting, we have $m\leq \prod_a (d_a+1)\in
O(d)$, $n=\vav+1$, and $D\leq d$, so the complexity settles into $O(d (d )^{
\vav \omega})=O(d^{\vav \omega+1})$. \qed

%==============================================================================

%==============================================================================
%\input{long-examples.tex}
%==============================================================================

\section{Small rank densors and isomorphism problems}\label{sec:Small-densor}

We close with the application that brought about this study. To that end, we
return to view operator spaces as sets, rather than $K$-schemes, and just write
$\Op{S}{P}$.

\subsection{The Tensor Isomorphism Problem}\label{sec:isomorphism-problem}

As we saw in Section~\ref{subsec:densor}, the Tensor Isomorphism Problem (TIP)
concerns tensors, or tensor spaces $S$ and $T$, and asks if there is a
transverse operator $\omega\in \Omega^{\times}$ such that $\omega  S= T$. This
has applications as varied as SLOCC equivalence in quantum
mechanics~\citelist{\cite{quantum} \cite{SLOCC:Classification}}, group and
algebra isomorphism in mathematics~\citelist{\cite{BOW:graded}
\cite{BMW:genus2}}, and computational complexity in Computer Science; cf.\
\citelist{\cite{Grochow-Qiao} \cite{Li-Qiao}}.  A detailed study of this problem
is outside our scope (see~\citelist{\cite{BW:autotopism}
\cite{BMW:exactseq} \cite{BMW:Der-Densor} \cite{BOW:graded}
\cite{Wilson:Skolem-Noether}}), but the contribution of this note is summarized
in the following.

\begin{prop}\label{prop:gen-op-ten}
	For every  ideal $P\subset K[X]$, the following holds.
	\begin{align*}
		(\exists \omega\in \Omega^{\times} )(\omega  S= T) \Leftrightarrow 
			(\exists \tau\in \Omega^{\times} )(\exists \nu\in \Omega^{\times})
			\left\{\begin{array}{rcl}
				\Op{S}{P}\tau & = & \tau \Op{T}{P}, \\
				\Op{T}{P}\nu & = & \nu \Op{T}{P}, \text{ and}\\
				\tau S & = & \nu T.
			\end{array}\right.
	\end{align*}
\end{prop}

\begin{proof}
	We claim that $\omega^{-1}\Op{S}{P}\omega=\Op{\omega S}{P}$, for all $\omega
	\in \Omega^\times$. Indeed, we have $\tau\in \Op{\omega S}{P}$ if, and only
	if, $(\forall p\in P)(\bra{t}\omega p(\tau)=0)$ if, and only if,
	$(\forall p\in P)(0=\bra{t}\omega p(\tau)\omega^{-1}=\bra{t} p(\omega \tau \omega^{-1}))$
	if, and only if, $\omega \tau \omega^{-1}\in \Op{S}{P}$. Thus,
	$\omega S=T$ implies $\Op{S}{P}\omega=\omega \Op{T}{P}$, and using
	$\nu=1_{\Omega}$, we get the forward direction.  For the converse set
	$\omega = \nu^{-1}\tau$.
\end{proof}

Now we consider the impact of Proposition~\ref{prop:gen-op-ten} on TIP.
Suppose $T$ and $S$ are tensor spaces framed by $(V_0,\ldots,V_{\vav})$ and
$(U_0,\ldots,U_{\vav})$ respectively.  We first assign isomorphisms
$\phi_a:V_a\to U_a$ and regard tensors in $T$ and $S$ as having the same frame.
To decide transverse isomorphism, it is now possible to limit the actions using
the  Correspondence Theorem~\ref{thm:correspondence}. 

Choose an ideal $P$. Good choices include those where we have
shown the sets $\Op{S}{P}$ carry algebraic structure. If the goal is
computational, then we also seek that $\Op{S}{P}$ be efficiently computable, for
example by taking $P$ to be a homogeneous linear ideal.  For theoretical
considerations any ideal can be used.  By applying
Proposition~\ref{prop:gen-op-ten}, we split the search problem up into two
phases.  In the first phase, we work to conjugate $\Op{S}{P}$ to $\Op{T}{P}$.
Here we are free to use the many features we can calculate for these sets.  For
example, if these are both algebras we can appeal to algorithms in computational
algebra to determine structure such as simple factors, radicals, and irreducible
representations, as in \citelist{\cite{BW:isom} \cite{BMW:genus2} 
\cite{Wilson:Skolem-Noether}}.  These must all agree in order that the operator sets be
conjugate.

Once we succeed in finding $\tau$ conjugating $\Op{S}{P}$ to $\Op{T}{P}$, the
second phase searches for $\nu$ which normalizes $\Op{T}{P}$ and transports
$\tau S$ to $T$. Notice this new search takes place in the smaller tensor space
$\Ten{P}{\Op{T}{P}}$ since $\Op{T}{P}=\nu^{-1}\Op{S}{P}\nu$ implies 
\begin{align*}
	\Ten{P}{\Op{T}{P}} &= \Ten{P}{\nu^{-1}\Op{S}{P}\nu} 
	= \Ten{P}{\Op{\nu S}{P}}
	\supset \nu S.
\end{align*} 
Thus, we not only take advantage of the algebro-geometric structure on
$\Op{T}{P}$, but we also work in a potentially much smaller tensor space. In
searching for $\nu$, we may further decrease the search space by writing
$\nu=(\nu_A,\nu_{\bar{A}})$ with $A\subset [\vav]$ and asking for
$S(\nu_A,1_{\bar{A}}) = T(1_A,\nu_{\bar{A}}^{-1})$ while enumerating on relevant
$\nu_A$ and $\nu_{\bar{A}}$ separately.  Notice now this is simply expressing
a functor between the cores of two homotopism categories allowing us to shift
the problem to a completely new context.  Sometimes this moves a non-abelian category,
e.g. $(x_a x_b-1)$-homotopisms, to an abelian category of $(x_a-x_b)$-homotopism, 
where the solution becomes exponentially easier to find.

In retrospect, several recent advances in isomorphism tests can be seen as
examples of this method.  What has become known as the ``adjoint-tensor'' method
of \citelist{\cite{BOW:graded} \cite{LW:iso} \cite{BMW:genus2}
\cite{Wilson:Skolem-Noether}} uses the ideals $(x_a-x_b)$, which gives rise to
associative algebras (Theorem~\ref{mainthm:Lie-Asc}). In light of
Theorem~\ref{mainthm:Densor}, the optimal choice is not an associative algebra
but a Lie algebra.  The last two authors together with Brooksbank are developing
a so-called ``derivation-densor'' method exploiting this optimality
\citelist{\cite{BMW:exactseq} \cite{BMW:Der-Densor}}.

\subsection{Densor dimension formulas}

As explained in Section~\ref{sec:isomorphism-problem}, the difficulty of
deciding tensor isomorphism increases exponentially with the dimension of the
tensor space $\Ten{P}{\Op{S}{P}}$ for a family of polynomials $P$. In
particular, the complexity of the recent densor based algorithms
\citelist{\cite{BMW:exactseq} \cite{BMW:Der-Densor}} depends in the  dimension
of the densor space $\Leftcircle S\Rightcircle$. We now exhibit some  families
of tensors for which the densor space has small dimension relative to the
ambient space of tensors $V_0\oslash \cdots \oslash V_{\vav}$.  

We first present some generalities. Given a Lie algebra $\mathfrak{L}$ and
$\mathfrak{L}$-modules $ V_0,\dots, V_{\vav} $, set 
\begin{align*}
	\Leftcircle V_* \Rightcircle_{\mathfrak{L}} 
	& = \Hom_{\mathfrak{L}}(V_1\otimes \cdots \otimes V_{\vav}, V_0).
\end{align*}
Then one readily checks that the densor $\Den{S}$ of a tensor space $S$ is
obtained as the special case where
$\mathfrak{L}=\Op{S}{x_0-x_1-\dots-x_{\vav}}$. If $V_a=X_a\oplus Y_a$ as
$\mathfrak{L}$-modules for some $a\in \zrange{\vav}$, then there is a canonical
isomorphism
\begin{align*}
	\Leftcircle X_a\oplus Y_a,V_{\bar{a}}\Rightcircle_{\mathfrak{L}}
		& \cong \Leftcircle X_a,V_{\bar{a}}\Rightcircle_{\mathfrak{L}}
		\oplus \Leftcircle Y_a,V_{\bar{a}}\Rightcircle_{\mathfrak{L}}.
\end{align*}
By a recursive application of this rule we reduce to computing the dimension of
$\Leftcircle V_* \Rightcircle_{\mathfrak{L}}  $ when all the $V_a$ are
indecomposable $\mathfrak{L}$-modules.  

Observe further that $\Leftcircle V_* \Rightcircle_{\mathfrak{L}}  $ is
naturally a module over $\bigotimes_{a\in\zrange{\vav}}
\End_{\mathfrak{L}}(V_a)$ (tensoring is over $K$). If some $V_a$ are
indecomposable $\mathfrak{L}$-modules that are not simple, then
$\End_{\mathfrak{L}}(V_a)$ has a nontrivial radical, and therefore so does
$\Leftcircle V_* \Rightcircle_{\mathfrak{L}}  $. This is situation occurs in the
following example. 

\begin{ex}\label{ex:densor-radical}
	Let $A = K[x]/(x^n)$, and fix a basis $\{e_1,\dots, e_n\}$ for $A$, where
	$e_k = x^{k-1} + (x^n)$. Then the multiplication tensor $\bra{t}:A\times
	A\bmto A$ has a densor space $\Den{t}$ spanned by the
	following set of tensors, given via the structure constant
	representation~\cite{TensorSpace}
	\begin{align}\label{eqn:densor-span}
		\left\{
			\begin{bmatrix}
				e_1 & e_2  & \cdots & e_n\\
				e_2 & & \udots & 0\\
				\vdots & \udots & \udots & \\
				e_n & 0 & & 
			\end{bmatrix}, \dots,
			\begin{bmatrix}
				e_1 & e_2 & 0 & \dots\\
				e_2 & 0 & & \\
				0 & & \ddots  & \\
				\vdots & &  & 
			\end{bmatrix},		
			\begin{bmatrix}
				e_1 & 0 & 0 & \cdots \\
				0 & 0 & & \\
				0 & & \ddots & \\
				\vdots & & & 
			\end{bmatrix}\right\}.
	\end{align}
	Label the tensors from~\eqref{eqn:densor-span}, $\{t_0,\dots, t_{n-1}\}$, so
	$t = t_0$. Let $J_0$ be the $n\times n$ Jordan block with $0$ along the
	diagonal. Then $J_0\in\rad (\End_{\Der(t)}(V_a))$ for all $a$, and the
	tensors $\{t_0,\dots, t_{n-1}\}$ satisfy the recurrence $\langle t_{k+1}
	\ket{v} = J_0\langle t_k | J_0 v_2, J_0 v_1 \rangle$.
\end{ex}

Consequently, the most compact densors occur when each
$\End_{\mathfrak{L}}(V_a)$ is a division ring, for example when each $V_a$ is a
simple $\mathfrak{L}$-module.  In this situation, we appeal to the
Littlewood--Richardson rule to determine the formula for the dimension; see
\cite{Fulton}*{Chapters~4 \& 25.3}. The asymptotic behavior of
Littlewood--Richardson numbers is quite difficult to predict, but the point is
that the values are, in general, substantially smaller than $\prod_a \dim V_a$.
Moreover, we can compute them in many special cases as we now demonstrate.

\subsubsection{Densors related to type $A$ simple Lie algebras}

For a square matrix $X$, set $\mathrm{tr}(X)=\sum_i X_{ii}$. We define the
vector space
\begin{align*}
	\mathbb{M}_n(K)^0 
		& = \left\{ X\in \mathbb{M}_n(K) : {\rm tr}(X)=0\right\}.
\end{align*}
We have the usual Lie product $[X,Y]=XY-YX$ on $\mathbb{M}_n(K)^0$ to make it
${\frak sl}_n(K)$, but we want to distinguish the space from the Lie algebra. In
the following, we assume that the characteristic of $K$ is $0$, or sufficiently
large relative to $n$.

\begin{ex}\label{ex:sl_n-natural}
	Let $V_2=\mathbb{M}_n(K)^0$ and $V_0=V_1=K^n$.  If $\mathfrak{L}=\{(\ad_X,
	X,X): X\in {\frak sl}_n(K)\}$ then $\Leftcircle
	V_0,V_1,V_2\Rightcircle_{\mathfrak{L}}$ is $1$-dimensional and spanned by
	the tensor $\bra{t}:V_1\times V_2\bmto V_0$ given by the natural ${\frak
	sl}_n(K)$-module on $K^n$.  Since $\mathfrak{L}\subseteq \Der(t)$, the
	natural ${\frak sl}_n(K)$-module $K^n$ is characterized completely by its
	densor.  In comparison, if we wish to use left, right or middle scalars, the
	smallest Whitney tensor product space, i.e.~of the form $\hom_{L\otimes
	R^{\op}}(V_1\otimes_M V_2,V_0)$, containing $\bra{t}$ is $\hom_{K\otimes
	K}(V_1\otimes_K V_2,V_0)$, which has dimension $O(n^4)$.
\end{ex}

\begin{proof}
	The tensors in $\Leftcircle V_0,V_1,V_2\Rightcircle_{\mathfrak{L}}$ can be
	regarded as ${\frak sl}_n(K)$-morphisms from $ V_2={\frak sl}_n(K) $ to $V_0
	\rversor V_1\cong \mathbb{M}_n(K)={\frak sl}_n(K)\oplus K$. Because the
	modules are irreducible, the images of these homomorphisms are scalar
	multiples of each other.  As the natural module action $\bra{t}$ admits
	$\mathfrak{L}$ as derivations and $\bra{t}\neq 0$, it follows that
	$\Den{t}=K\bra{t}$. Characterizing the maximal Whitney tensor products is
	determined by the nuclei~\citelist{\cite{BW:autotopism}
	\cite{Wilson:Skolem-Noether}}. In this case the nuclei are each copies of
	$K$.
\end{proof}

\begin{ex}\label{ex:ad-sl}
	For $n\geq 2$, set $V_0=V_1=V_2=\mathbb{M}_n(K)^0$ and $\mathfrak{L}=\{(\ad
	X,\ad X,\ad X) \mid X\in \mathfrak{sl}_n(K)\}$.  Let $\bra{t} : V_1 \times
	V_2 \rightarrowtail V_0$ be the multiplication-in-${\frak sl}_n$ tensor
	given by $[X, Y] = XY - YX$. Then $\Leftcircle
	V_*\Rightcircle_{\mathfrak{L}} = \Den{t}$. It turns out that when $n=2$,
	$\Leftcircle V_*\Rightcircle_{\mathfrak{L}} = Kt$, whereas when $n\geq 3$,
	we have $\dim(\Leftcircle V_*\Rightcircle_{\mathfrak{L}}) = 2$. On the other
	hand, the smallest Whitney tensor space, $\hom_{L\otimes
	R^{\op}}(V_1\otimes_{M}V_2,V_0)$, containing the multiplication of
	$\mathfrak{sl}_n(K)$ is $\hom_{K\otimes K}(V_1\otimes_{K}V_2,V_0)$, of
	dimension $O(n^6)$.
\end{ex}

\begin{proof}
	This is proved more generally in~\cite{BMW:Der-Densor}, but we sketch the
	idea. Following the same blueprint as Example~\ref{ex:sl_n-natural}, the
	adjoint representation of $\mathfrak{sl}_n$ corresponds to the Young diagram
	for the partition $\mu = (2, 1, \dots, 1)\vdash n$. The dimension of the
	densor space is equal to the Littlewood--Richardson number for type $A$,
	written $c_{\mu, \mu}^{\mu}$.  In this case, $c_{\mu, \mu}^{\mu}$ is $1$
	when $n=2$ but is $2$ when $n\geq 3$. 
\end{proof}

Tensors of valence greater than $2$ can also have small densors. The following
example is the archetype of a semisimple associative pair algebra.
	
\begin{ex}
	Let $M=\mathbb{M}_{ab}(K)$ and define $\bra{t}:M^3\bmto M$ by $\langle
	t\,|\,X_1,X_2,X_3\rangle=X_1X_2^{\dagger} X_3$.  Then $\Den{t}=tK$.
\end{ex}

\begin{proof}
	This applies because the nuclei are each represented irreducibly on the
	frame.  Since nuclei embed in $\Der(t)$,
	cf.~\cite{BMW:exactseq}*{Theorem~A}, the result follows.
\end{proof}

\subsubsection{Densors related to exceptional Lie algebras}

Moving to exceptional types, consider now the octonions.

\begin{ex}
	Assume $\mathbb{O}$ is an Octonion $K$-algebra over a field $K=6K$.
	The product $\bra{t}:\mathbb{O}\times \mathbb{O}\bmto \mathbb{O}$ 
	has densor of rank $1$.
\end{ex}

\begin{proof}
	This example is similar to Example~\ref{ex:ad-sl}. By a theorem
	Cartan--Jacobson, $\Der(t)\cong {\frak o}(8)$ (\cite{Schafer:nonass}*{p.
	82}) whenever $6K=K$. Each term of the frame is a different
	$\mathfrak{o}(8)$-representation: the natural, the positive spin, and the
	negative spin representations (as implied by Cartan's triality theorem
	\cite{Hall:triality}*{Chapter~19}). These are irreducible and of
	highest-weight. Following \cite{KN:classical-crystals}*{Section~6}, these
	correspond to the generalized Young diagrams of shapes 
	\begin{align*}
		\lambda &= (1, 0, 0, 0), &
		\mu &= (1/2, 1/2, 1/2, 1/2), &
		\nu &= (1/2, 1/2, 1/2, -1/2).
	\end{align*}
	By~\cite{Nakashima:gen-LR}, the generalized Littlewood--Richardson numbers
	(type $D$) satisfy $c_{\lambda, \mu}^{\nu} = c_{\nu, \lambda}^{\mu} =
	c_{\mu, \nu}^{\lambda} = 1$. 
\end{proof}

Next we choose a demonstration of the software \cite{TensorSpace}. We supply to
assist with calculations of densors.  We consider product $A\circ B=1/2(AB+BA)$
on the following exceptional simple Jordan algebra over $K=6K$:
\begin{align*}
	\mathfrak{H}_3(\mathbb{O}) & = \{ A\in \mathbb{M}_3(\mathbb{O})\mid A=\bar{A}^{\dagger}\}.
\end{align*}
This space is 27-dimensional and lives naturally inside a 19,683-dimensional
tensor space. It is therefore rather unlikely to recognize this tensor in an
arbitrary basis by some undirected method.  However, by the algorithms of
\cite{TensorSpace} find the densor space satisfies $\dim\Den{t}=5$ and
\begin{align*}
	\Der(t) & \cong \left\{
	\begin{bmatrix} a & u \\ 0 & B \end{bmatrix}
	~\middle|~ \begin{array}{c} 
	a\in K\\
	u\in K^{26}\\
	B\in F_4
	\end{array}
	\right\}.
\end{align*}
The calculations of course have no a priori knowledge of any structure and are
the same regardless of any unfortunate choice of bases to begin with. This in
effect reduces questions from 19,683-dimensions to just 5.

What has ocurred here is that this calculation has recovered an invariant
property of this algebra explored by Jacobson; see
\cite{Schafer:nonass}*{p.108-112}.  Here is an explanation.

\begin{ex}\label{ex:Albert}
	The multiplication $\bra{t}:{\frak H}_3(\mathbb{O})\times {\frak
	H}_3(\mathbb{O})\bmto {\frak H}_3(\mathbb{O})$ of the exceptional Jordan
	algebra has densor of rank $5$.
\end{ex}

\begin{proof}
	Take $\mathfrak{H}_0(\mathbb{O})$ to be the matrices in
	$\mathfrak{H}_3(\mathbb{O})$ of trace $0$.  Then
	$\mathfrak{H}_3(\mathbb{O})=K\oplus \mathfrak{H}_0(\mathbb{O})$, so elements
	in $\mathfrak{H}_3(\mathbb{O})$ can be expressed uniquely in the form
	$aI_3+X$ with $X\in \mathfrak{H}_0(\mathbb{O})$.  The densor space
	$\Leftcircle t\Rightcircle$ is spanned by the following linearly independent
	set of tensors.
	\begin{align*}
		\langle t_1 | aI_3+X, bI_3+Y\rangle
			& = abI_3\\
		\langle t_2 | aI_3+X, bI_3+Y\rangle
			& = aX+bY\\
		\langle t_3 | aI_3+X, bI_3+Y\rangle
			& = abI_2+aX+bY\\
		\langle t_4 | aI_3+X, bI_3+Y\rangle
			& = \frac{1}{2}(XY+YX)\\
		\langle t_5 | aI_3+X, bI_3+Y\rangle
			& = abI_3+aX+bY+\frac{1}{2}(XY+YX).
	\end{align*}
	In particular $\dim \Leftcircle \mathfrak{H}(\mathbb{O})\Rightcircle=5$.  Of
	the tensors $t_1,\dots,t_5$ above, $\Leftcircle t_5\Rightcircle <\Leftcircle
	t_i\Rightcircle$, for $i<5$.
\end{proof}

\subsubsection{The non-field case}

Recall from Myasnikov's Theorem~\cite{Myasnikov}, the centroid $\Cen(S)$,
defined in Example~\ref{ex:centroid}, can be described as $\Op{S}{(x_a-x_b\mid
a,b\in \zrange{\vav})}$ and is an associative unital $K$-algebra. Furthermore,
if $S$ is fully nondegenerate then $\Cen(S)$ is commutative.  It is the largest
ring over which a tensor is multilinear.

We have so far considered densors of small dimension.  There is a natural reason
to consider a broader class of ``small'' densors where we replace the notion of
dimension with rank over the centroid of the tensor.  The following examples
demonstrate the extended range of such tensors.  In many situations the rank is
already small when we consider a closure over a centroid or over nuclei rather
than the entire derivation algebra.

\begin{prop}\label{prop:comm-assoc-alg}
	Given a commutative associative unital $K$-algebra $A$, its multiplication
	tensor $\bra{t}:A^{2}\bmto A$ given by $\langle t| a\rangle=a_1\cdot a_2$,
	has centroid $\Cen(t)\cong A$ and $\Leftcircle t\Rightcircle$ has rank $1$
	over $\Cen(t)$.
\end{prop}

\begin{proof}
	The regular representation of $A$ in ${\rm End}(A)$ is faithful as $A$ is
	unital.  Furthermore, the multiplication tensor $t$ in $A$ is
	$A$-multilinear because $A$ is commutative.  By definition, $A$ embeds in
	$\Cen(t)$.  As $A$ is unital, it follows that $t$ is nondegenerate, so
	$\Cen(t)$ is faithfully represented on $A$. So $\Cen(t)\cong A$. Suppose
	$s\in \Leftcircle t\Rightcircle$, so $\Der(t)\subseteq\Der(s)$.  Since
	$\Cen(t)$ embeds into $\Der(t)$ in two ways:
	$(\omega_0,\omega_1,\omega_2)\mapsto (\omega_0,0,\omega_2)$ and
	$(\omega_0,\omega_1,\omega_2)\mapsto (0,\omega_1,\omega_2)$.  Hence,
	$\{(\omega_0,0,\omega_2)\mid\omega\in \Cen(t)\}$,
	$\{(0,\omega_1,\omega_2)\mid \omega\in \Cen(t)\}\subset \Der(s)$. So
	$\Cen(t)\subset \Op{s}{(x_0-x_1,x_0-x_2)}=\Cen(s)$.  Hence, $\langle
	s|a_1,a_2\rangle=a_1\cdot a_2\cdot \langle s|1,1\rangle =\lambda_s\langle
	t|a_1,a_2\rangle$ where $\lambda_s:=\langle s|1,1\rangle\in A$. Hence
	$\bra{s}=\lambda_s\bra{s}$.
\end{proof}

As an application of Proposition~\ref{prop:comm-assoc-alg}, we look at two
tensors from Quantum Information Theory, the GHZ and W states. In their
3-partite states, these can be interpreted as the structure constants of
commutative, associative, unital $\mathbb{C}$-algebras $\mathbb{C}^2$ and
$\mathbb{C}[x]/(x^2)$ respectively. 

\begin{ex}\label{ex:GHZ}
	Let $\mathbb{H}$ be an $8$-dimensional Hilbert space, and define the
	following interpretation map $\langle \cdot | : \mathbb{H} \rightarrow
	\mathbb{C}^2\rversor \mathbb{C}^2 \rversor \mathbb{C}^2$. The convention is that $\mathbb{C}^2 =
	\mathbb{C}\langle 0 | \oplus \mathbb{C} \langle 1 |$, so $\mathbb{H}$ has a
	basis denoted $\langle abc |$, where $a,b,c\in\{0,1\}$. The GHZ and W states
	are 
	\begin{align*}
		\langle GHZ | &= \dfrac{\sqrt{2}}{2} \left( \langle 000 | + \langle 111 |\right), & 
		\langle W | &= \dfrac{\sqrt{3}}{3} \left( \langle 100 | + \langle 010 | + \langle 001 | \right).
	\end{align*}
	See also Figure~\ref{fig:annihilator-2},
	The derivation algebra of the GHZ state is the abelian Lie algebra
	$\mathbb{C}^4$. If $\mathfrak{t}_n$ is the solvable Lie algebra of $n\times
	n$ upper triangular matrices, then the derivation algebra of the W state is
	isomorphic to $\mathfrak{t}_2\oplus \mathbb{C}^2$. By
	Proposition~\ref{prop:comm-assoc-alg}, the densors are rank-one over their
	centroids. But as $\mathbb{C}$-vector spaces the densor of GHZ is
	$2$-dimensional, spanned by the two tensors: $\langle 000 |$ and $\langle
	111 |$, and the densor of the W state is $1$-dimensional. 
\end{ex}

Now we consider classes of tensors where the nuclei is enough to contract a
tensor space to a small rank.

\begin{prop}
	The product of an Azumaya algebra has rank $1$ densor over its centroid.  In
	particular every central simple associative algebra has a rank $1$ densor.  
\end{prop}

\begin{proof}
	As $A$ is associative the multiplication tensor $t$ admits a left-, mid-,
	and right- action by $A$, i.e.\ for $0\leq i<j\leq 2$, $A\to
	\Op{t}{(x_i-x_j)}$.  As $A$ is also unital each of these representations is
	faithful.  In all cases $A\otimes_A A\cong A$.  As $A$ is Azumaya,
	$A\otimes_K A^{\op}\cong \End_K(A)$.  Set $P=\cap_{ij} (x_i-x_j)$.  Thus,
	\begin{align*}
        \Ten{P}{\Op{t}{P}}\subset (A\otimes_A A)\lversor_{A\otimes A^{\op}} A 
        \cong A\lversor_{\End_K(A)} A=\End_{\End_K(A)}(A)\cong K.
	\end{align*}
	Since $t\in \Ten{P}{\Op{t}{P}}$, it follows that $\Ten{P}{\Op{t}{P}}\cong
    K$. By Theorem~\ref{mainthm:Densor}, $\Leftcircle t\Rightcircle \cong K$
    because $K\cong Kt\leq \Leftcircle t\Rightcircle\subset
    \Ten{P}{\Op{t}{P}}\cong K$.
\end{proof}

That proof adapts to prove even the following claim, in particular it is not
necessary that the product form an algebra, only that the nuclei act irreducibly
on the frame.
\begin{prop}
	If $A$, $B$, and $C$ are progenerators of composition $\bra{t}:A\rversor
	B\times B\rversor C\bmto A\rversor C$, i.e.\ $\langle t| f,g\rangle=f\circ
	g$, has a rank $1$ densor.
\end{prop}

\begin{coro}
	The matrix multiplication tensor spans is own densor.
\end{coro}

\subsection{Examples we cannot yet explain}

In the years since we began computing densor spaces, there have been many
situations where we encountered a lowering of the dimension of a tensor product
space but for reasons we cannot yet explain.  For example, we found such
compression in the (noise-free) models of chat-room data as given in
\cite{chatroom}*{Section~3}.  We also found proper densor spaces in the tensors
that arise in exchangeable relational data, for instance ones described in
\cite{MFMc}.  We have further carried on with higher qubit SLOCC
classifications, and while this list is infinite, random trials show as many
cases have proper densor spaces as those that do not.  So some information is
captured by this method even as valence grows. Also, in a survey of over
500,000,000 nilpotent groups, the second and third authors found proper densor
spaces (of the commutator tensor) occurred in 80\% of the trials.  This
improves isomorphism testing in as many cases.

Finding smaller spaces and nontrivial Lie algebras to act on is often a direct
benefit to an existing strategy.  However, we do not understand what features of
a tensor lead to large derivation algebras.  For now, we simply compute and
discover.  Certainly it may help to start by explaining those tensors that
support a simple Lie algebra of derivations, as touched on in
Section~\ref{sec:Small-densor} and in \cite{BMW:Der-Densor}.  Yet most of the
derivation algebras discovered in the above unexplored examples are solvable. So
there is a great deal left to discern.

%==============================================================================

%==============================================================================
%\input{closing.tex}
\section{Summary of results \& Open questions}

Motivated by patterns with tensors found throughout the scientific literature,
we introduced a correspondence between tensor spaces, multivariable polynomials
ideals, and transverse operators (Theorem~\ref{thm:correspondence}). This built
on a generalization of concepts of eigen spaces, minimal polynomials, and the
many familiar concepts of endomorphisms, derivations and automorphisms. The
closures of this correspondence gave a disciplined means to generalize
tensor product spaces and match them with the best possible operators -- those 
admitting a universal definition.

Next we proved an optimality condition on the universal linear operators leading
to the definition of a unique smallest (linear) tensor product space which we
called the \emph{densor} (Theorem~\ref{mainthm:Densor}).  An essential
ingredient was the ability to pass to ring extensions through the adoption of a
scheme-theoretic model. We further characterized what families of
non-associative algebras can act transversely finding that Lie algebras are the
natural choice. The traditional use of associative algebras suggested 70 years
ago by Whitney applies, we found, only to pairs of spaces
(Theorem~\ref{mainthm:Lie-Asc}).

In search of the proper data types to use in large tensor calculations, we next
pursued the most natural description of transverse tensor categories. For that
we leaned on the Correspondence Theorem to prove a characterization of which
groups act transversely (Theorem~\ref{thm:group-intro}) which led us to define
the largest possible transverse tensor categories -- homotopism categories.  We
solved this by invoking the theory of toric schemes.  This exploration leaves
open some questions concerning the combinatorics of lattices that we expect
could prove even stronger versions of our claims
(Question~\ref{quest:full-converse}).

With these characterizations in hand, we returned to the study of core questions
about tensors such as finding clusters in data and decompositions more
generally.  We observed these are now natural instances of decompositions in
categories and are thus subject to helpful theory like the Jordan--H\"older and
Krull--Schmidt theorems.  These homotopism categories organize substructure into
abstract simplicial complexes, and we used that combinatorial characterization
to identify precisely the polynomial traits that signal the presence of
substructure (Theorem~\ref{mainthm:singularity}).  The implication being that
when we search for decompositions using transverse operators, we now have a
discrete known target set of traits on which to focus.

Finally, we used the naturally occurring structure uncovered to design a tensor
type-system provably capable of modeling arbitrary higher-order tensor problems,
but designed to retain the efficiencies of highly-tuned low-level tensor
libraries. We also proved our correspondence is polynomial-time computable
(Theorem~\ref{mainthm:construction}), and our implementation of these ideas can
be found in \citelist{\cite{TensorSpace}\cite{magma}}.  We found that solving a
simultaneous system of Sylvester equations is a key bottleneck in efficient
computations with tensors.  This remains an area for optimism, given the
quantity of study on a single instance of Sylvester equations
(Question~\ref{quest:quicksylver}).

Along the way we have selected a number of examples from the literature, but
these are clearly filtered through topics with which we have familiarity.  We
encourage exploration within new domains to learn of the limits and
opportunities beyond what we have observed.  One clear open area is to develop
our Correspondence Theorem and its implications in the context of symmetry.
For $\sigma\in \mathrm{Sym}(\zrange{\vav})$, $t\in T$ is $\sigma$-symmetric 
if $\exists A\subset\zrange{\vav}$, $B\subset \comp{A}$, and 
$\omega\in \Omega_{A,B}^{\times}$ with $\bra{t}\omega=\bra{t}^{\sigma}$
(where the latter is the Knuth-Liebler shuffle of Section~\ref{sec:moving-data}).

\begin{quest}
    For subgroups $G\leq \mathrm{Sym}(\zrange{\vav})$, can the correspondence
    theorem be generalized to allow for $G$-symmetry?  If so, what are the
    minimal members of $\Ten{S}{\Op{S}{P}^G}^G$ as $P$ ranges over $G$-invariant
    homogeneous linear ideals $P$? What families of algebras appear as
    $\Op{S}{P}^G$? Also what are the constraints on $P$ that make $\Op{S}{P}^G$
    a subgroup of $(\Omega^G)^{\times}$?    
\end{quest}
%==============================================================================

%==============================================================================

\end{document}